\documentclass[11pt,oneside]{amsart}
\usepackage{appendix}
\usepackage{amsmath}
\usepackage{amsfonts}
\usepackage{amssymb}
\usepackage{amsthm}
\usepackage[colorlinks]{hyperref} 
\usepackage{wasysym}
\usepackage{graphicx}
\usepackage{dashrule}
\usepackage[all]{xy}
\usepackage{multirow}
\usepackage{enumitem}
\usepackage[margin=1in]{geometry}
\usepackage[dvipsnames]{xcolor}
\usepackage{tikz}
\usetikzlibrary{arrows,decorations.pathmorphing,backgrounds,positioning,fit,calc}
\usepackage{tikz-cd}
\usepackage{framed}
\usepackage{makecell}
\usepackage{mathdots}
\usepackage{tikz-cd}
\usepackage{subcaption}
\usepackage{theoremref}
\usepackage{todonotes}
\usepackage{wrapfig}

\normalfont\upshape

\theoremstyle{plain}
\newtheorem{defn}{Definition}[section]
\newtheorem{thm}[defn]{Theorem}
\newtheorem{propn}[defn]{Proposition}
\newtheorem{lem}[defn]{Lemma}
\newtheorem{cor}[defn]{Corollary}
\newtheorem{rk}[defn]{Remark} 
\newtheorem*{pcalg}{Projective completion algorithm} 
\newtheorem*{repalg}{Replacement algorithm}
\newtheorem{step}{Step} 
\newtheorem*{stepk+1}{Step $\boldsymbol{k+1}$}
\newtheorem{case}{Case}
\newtheorem{subcase}{Case}
\numberwithin{subcase}{case}

\theoremstyle{remark}
\newtheorem{rmk}[defn]{Remark}

\DeclareMathOperator{\GL}{GL}
\DeclareMathOperator{\SL}{SL}

\DeclareMathOperator{\SU}{SU}

\DeclareMathOperator{\proj}{Proj}
\DeclareMathOperator{\spec}{Spec}

\DeclareMathOperator{\sym}{Sym}
\DeclareMathOperator{\conv}{Conv}
\DeclareMathOperator{\stab}{Stab}

\DeclareMathOperator{\Aut}{Aut}
\DeclareMathOperator{\Lie}{Lie}

\newcommand{\nocontentsline}[3]{}
\newcommand{\tocless}[2]{\bgroup\let\addcontentsline=\nocontentsline#1{#2}\egroup}

\newcommand{\act}{\curvearrowright}

\newcommand{\CC}{\mathbb{C}}
\newcommand{\PP}{\mathbb{P}}
\newcommand{\A}{\mathbb{A}}
\newcommand{\GG}{\mathbb{G}}
\newcommand{\OO}{\mathcal{O}}

\newcommand{\dblslash}{/\! \!/}

\newcommand{\ten}{\otimes}
\newcommand{\mc}{\mathcal}
\newcommand{\mf}{\mathfrak}
\newcommand{\mb}{\mathbb}

\newcommand{\ol}{\overline}

\newcommand{\kk}{\Bbbk}
\newcommand{\symdot}{\sym^{\bullet}}
\newcommand{\nss}{\mathrm{nss}}
\newcommand{\ssfg}{\mathrm{ss}}

\newcommand{\rms}{s}
\newcommand{\rmss}{\mathrm{ss}}

\newcommand{\Proj}{{\rm Proj}}
\newcommand{\Spec}{{\rm Spec}}

\newcommand{\nc}{\newcommand}
\nc{\bla}{\phantom{bbbbb}}

\newcommand{\beq}{\begin{equation}}
\newcommand{\eeq}{\end{equation}}
\newcommand{\barr}{\begin{array}}
\newcommand{\earr}{\end{array}}
\newcommand{\beqar}{\begin{eqnarray}}
\newcommand{\eeqar}{\end{eqnarray}}
\newtheorem{theorem}{Theorem}[section]

\newtheorem{prop}[theorem]{Proposition}

\newtheorem{remark}[theorem]{Remark}

\newtheorem{conjecture}[theorem]{Conjecture}
\newtheorem{exit}[theorem]{Example}

\nc{\FF}{ {\mathbb F} }
\nc{\HH}{ {\mathbb H} }

\newcommand{\UU}{{\mathbb U }}
\nc{\AL}{{\mathbb A}}

\newcommand{\calo}{\mathcal{O}}

\newcommand{\calr}{\mathcal{R}}

\DeclareMathOperator{\Stab}{Stab}

\nc{\umax}{{U_{\max}}}
\newcommand{\weight}{\omega}
\newcommand{\Diff}{\mathrm{Diff}}
\newcommand{\CHilb}{\mathrm{CHilb}}
\newcommand{\Hilb}{\mathrm{Hilb}}

\newcommand{\bars}{\overline{s}}
\newcommand{\barss}{\overline{ss}}

\newcommand{\hH}{\hat{H}}
\newcommand{\hU}{\hat{U}}

\newcommand{\liet}{{\mathfrak t}}

\nc{\lieq}{{\mathfrak q}}
\nc{\liez}{{\mathfrak z}}
\nc{\lieqs}{{\lieq}^*}
\nc{\lieg}{{\mathfrak g}}
\nc{\liegs}{{\lieg}^*}
\nc{\liep}{{\mathfrak p}}
\nc{\lieps}{{\liep}^*}

\def\a{\alpha}
\def\b{\beta}

\def\t{\theta}

\def\l{\lambda}

\def\o{\omega}

\newcommand{\zmin}{Z_{\min}}
\newcommand{\xmino}{X_{\min}^0}
\newcommand{\hX}{\hat{X}}
\newcommand{\tX}{\tilde{X}}
\newcommand{\xminsg}{X^{s,\GG_m}_{\min +}}
\newcommand{\xminshu}{X^{s,\hU}_{\min +}}
\newcommand{\hxminshu}{\hX^{s, \hU}_{\min +}}

\newcommand{\xminsh}{X^{s,H}_{\min +}}
\newcommand{\xminssh}{X^{ss,H}_{\min +}}

\newcommand{\xhshu}{X^{\widehat{s},\hU}}
\newcommand{\xhsh}{X^{\hat{s},H}}

\newcommand{\env}{\!
\mathbin{\text{\rotatebox[origin=c]{70}{\scalebox{1.2}{$\approx$}}}} \!}
\newcommand{\wenv}{\, \widehat{\env} \,}
\newcommand{\tenv}{\, \widetilde{\env} \,}
\newcommand{\inenv}{\dblslash \!_{\circ}}

\newcommand{\gitq}{/\!/}

\newcommand{\PC}{\operatorname{PC}}

\title{Projective completions of graded unipotent quotients}

\author{Gergely B\'erczi, Brent Doran, Thomas Hawes, Frances Kirwan\\  \\with an appendix by Eloise Hamilton}

\address[B\'erczi]{Dep. of Mathematics, Aarhus University, Ny Munkegade 118, 8000 Aarhus,
Denmark}
\email{gergely.berczi@math.au.dk}

\address[Doran, Hamilton, Hawes, Kirwan ]{Mathematical Institute, Oxford University, OX2 6GG Oxford UK}
\email{brent.doran@maths.ox.ac.uk}
 \email{thomas.hawes01@gmail.com}
\email{kirwan@maths.ox.ac.uk}
\email{eloise.hamilton@maths.ox.ac.uk}

\thanks{Early work on this project was supported by the Engineering and Physical Sciences 
Research Council [grant numbers   GR/T016170/1,EP/G000174/1].
Brent Doran was partially supported by Swiss National Science Foundation Award 200021-138071. This project has received funding from the European Union’s Horizon 2020 research and innovation
programme under the Marie Sklodowska-Curie grant agreement No 742052}

\begin{document}

\begin{abstract} The aim of this paper is to show that classical geometric invariant theory (GIT)
has an effective analogue for linear actions of a non-reductive algebraic group
$H$ with graded unipotent radical on a projective scheme $X$. Here the linear
action of $H$ is required to extend to a semi-direct product $\hat{H} = H \rtimes \GG_m$ with a multiplicative 
one-parameter group which acts on the Lie algebra of the unipotent radical $U$
of $H$ with all weights strictly positive, and which centralises a Levi subgroup $R \cong H/U$ of $H$. 

We show
that $X$ has an $H$-invariant open subscheme (the 'hat-stable locus') which
has a geometric quotient by the $H$-action. This geometric quotient has a
projective completion which is a categorical quotient (indeed, a good quotient) by $\hat{H}$ of an open
subscheme of a blow-up of the product of $X$ with the affine line; with
additional blow-ups a projective completion which is itself a geometric
quotient can be obtained. Furthermore the hat-stable locus of $X$ and the
corresponding open subsets of the blow-ups of the product of $X$ with the
affine line can be described effectively using Hilbert-Mumford-like criteria
combined with the explicit blow-up constructions.

Applications include the construction of moduli spaces of sheaves and Higgs bundles of fixed Harder--Narasimhan type over a fixed nonsingular projective scheme \cite{BHJK,hamilton}, and of moduli spaces of unstable projective curves of fixed singularity \lq type' \cite{Jackson}.
In \cite{bkcoh,hamilton2} it is shown that cohomology theory for reductive GIT quotients can be extended to the non-reductive situation studied in this paper, and this is used in \cite{bkGGL} to prove the Green--Griffiths--Lang and Kobayashi hyperbolicity conjectures for generic projective hypersurfaces with polynomial bounds on their degree.
\end{abstract}

\maketitle

\tableofcontents

\section{Introduction}\label{sec:intro}
Let $H$ be a linear algebraic group acting   linearly (with respect to an ample line bundle $L$) on a projective scheme  $X$ over an algebraically closed field $\kk$ of characteristic 0. For simplicity in exposition we will assume $X$ to be irreducible. Let $U$ be the unipotent radical of $H$; then $H=U \rtimes R$ where the Levi subgroup $R\cong H/U$ of $H$ is reductive.
Suppose that the linear action of $H$ on $X$ extends to a semi-direct product $\hat{H} = H \rtimes \l(\GG_m) = H \rtimes (R \times \l(\GG_m))$ where the adjoint action of $\l:\GG_m \to \hat{H}$ on the Lie algebra of $U$ has strictly positive weights; 
 that is, $\l(\GG_m)$ grades the unipotent radical of $H$. The aim of this paper is to show that we then have an effective analogue of large parts of classical geometric invariant theory (GIT) for the action of $H$ on $X$. 
 
 More precisely, $X$ has an $H$-invariant open subscheme $X^{\hat{s}}$ (also denoted $X^{\hat{s},H}$, the \lq hat-stable locus') with a geometric quotient $X^{\hat{s}}/H$ and (if $X^{\hat{s}}$ is nonempty) a projective completion $X \wenv H$ of $X^{\hat{s}}/H$, 
 which is a categorical quotient (and  a good quotient in the sense of \cite{Alper,Dolg,GIT,New,new2}) by $\hat{H}$ of an open subscheme 
 of a blow-up of the product $X \times \AL^1$ of $X$ and the affine line $\AL^1$. With additional blow-ups a projective completion $X \tenv H$ which is itself a geometric quotient can be obtained. Here $X \wenv H = (X \wenv U)/\!/R$ is the classical GIT quotient for the induced linear action of the reductive group $R \cong \hat{H}/\hU$ (for $\hU =U\rtimes \l(\GG_m)$) on $X\wenv U$, and  
 $\hat{H} = H \rtimes \l(\GG_m)$ acts diagonally on $X \times \AL^1$ with the given action on $X$ and multiplication by a character $\hat{H} \to \GG_m$ with kernel $H$ on $\AL^1$; the choice of linearisation is such that $X \wenv R = (X \times \AL^1)/\!/ (R \times \l(\GG_m))$ can be identified naturally with the classical GIT quotient $X /\!/R$.
 When $R$ is trivial the hat-stable locus $X^{\hat{s}}$ is always non-empty; in general $X^{\hat{s},H}$ is the pre-image in $X^{\hat{s},U}$ of the stable locus for the induced linear action of 
  $R \cong \hat{H}/\hU$ on 
 the projective completion $X \wenv U$ of $X^{\hat{s},U}/U$.
The  subset $X^{\hat{s},H}$ of $X$ and the corresponding open subschemes of the  blow-ups of $X \times \AL^1$ 
 can be described effectively using Hilbert--Mumford criteria combined with the explicit blow-up constructions.

One special situation (the \lq internal' case) arises when $R$ itself has a central one-parameter subgroup $\l_{int} : \GG_m \to Z(R)$ whose adjoint action on the Lie algebra of $U$ has only strictly positive weights. We say then that $H$ has internally graded unipotent radical, and we can construct an \lq external' grading with associated semi-direct product 
$$\hat{H} = (U \rtimes R) \times \GG_m = U \rtimes (R \times \l_{ext} (\GG_m))$$
where $\l_{ext} : \GG_m \to Z(R \times \GG_m) = Z(R) \times \GG_m \leqslant (U \rtimes R) \times \GG_m$ is given by $\l_{ext}(t) = (\l_{int}(t),t)$. (Conversely, of course, given an external grading of the unipotent radical of $H$ with associated semi-direct product $\hat{H} = H \rtimes \l(\GG_m)$, the one-parameter subgroup $\l: \GG_m \to \hat{H}$ provides an internal grading for $\hat{H}$'s unipotent radical). 
In the internally graded situation, provided that we are willing to twist the linearisation of the $H$-action on $X$ by a suitable character, so that it becomes a \lq well adapted' linearisation, the resulting projective completion $X \wenv H$ of $X^{\hat{s},H}/H$ is a good quotient by $H$ of a blow-up of $X$ itself; there is no need to include the affine line in the construction. This extends the results of \cite{BDHK2} which show that under an additional hypothesis (roughly speaking, that semistability coincides with stability for the action of the unipotent radical) no blow-ups are needed to construct $X \wenv H$ and $X^{\hat{s},U}/U$ is itself projective. In future work \cite{BDKII,BDKIII} we will show that, in internally graded situations with well adapted linearisations, projective completions $X \env H$ of the geometric quotients $X^{\hat{s},H}/H$ can be obtained without any blowing up, and that the construction can even be generalised to apply without a pre-assigned extension of the linear action of $H = U \rtimes R$ to a semi-direct product $\hat{H} = H \rtimes \l(\GG_m)$ where $\l$ grades $U$.
The internally graded situation holds for many non-reductive groups $H$; examples include any parabolic subgroup of a reductive group, and the automorphism group of any complete simplicial toric variety, and the group of $k$-jets of biholomorphisms of $(\CC^p,0)$.

The geometric quotient  $X^{\hat{s}}/H$ and its projective completions 
 can depend in principle not only on the linearisation of the action of $H$ but also (in the external case) on the choice of semi-direct product $\hat{H} = H \rtimes \GG_m$ and extension of the linear action of $H$ to $\hat{H}$ (an \lq extended' or \lq graded' linearisation in the sense of \cite{BDK}) or (in the internal case) the choice of a central one-parameter subgroup of $R$ whose adjoint action on the Lie algebra of $U$ has only strictly positive weights. It is possible to  study the dependence on these choices, and to arrive at a picture similar to classical \lq variation of GIT' \cite{DolgHu,Thaddeus}; see \cite{BJK}. 
 
 In the case $\kk = \CC$ one can also use modifications of the symplectic quotient and implosion constructions to describe these quotients in terms of analogues of moment maps \cite{bkcoh,GJS,KPEN,K5} and to extend the methods of \cite{K,K2} to study their cohomology \cite{bkcoh}, with applications to the Green--Griffiths--Lang and Kobayashi hyperbolicity conjectures for generic hypersurfaces in a projective space \cite{bkGGL}. 
  
The best known examples of moduli spaces which can be constructed using classical GIT are moduli spaces of stable curves and of (semi)stable sheaves over a fixed nonsingular projective scheme. This more general version of GIT can be used (with $H$ the parabolic subgroup associated to a maximally destabilising one-parameter subgroup, or a closely related subgroup with internally graded unipotent radical) to construct moduli spaces of unstable objects with suitable fixed discrete invariants, such as unstable curves \cite{Jackson} or unstable sheaves \cite{BHJK}.
 
Recall that, when $H = R$ is reductive, we can use classical GIT, developed by Mumford in the 1960s \cite{GIT}, to find $H$-invariant open subschemes $X^s \subseteq X^{ss}$ (the stable and semistable loci) of $X$ with a geometric quotient $X^s/H$ and projective scheme $X/\!/H \supseteq X^s/H$  associated to the algebra  of invariants
 $\bigoplus_{k \geq 0} H^{0}(X,L^{\otimes k})^H$. 
Here $X/\!/H$ is the image
of a surjective morphism $\phi$ from the open subscheme $X^{ss}$ of $X$, determined by the inclusion of the algebra of invariants $\bigoplus_{k \geq 0} H^{0}(X,L^{\otimes k})^H$ in $\bigoplus_{k \geq 0} H^{0}(X,L^{\otimes k})$.  If $x,y \in X^{ss}$ then $\phi(x) = \phi(y)$ if and only if the closures of the $H$-orbits of $x$ and $y$ meet in $X^{ss}$ (that is, $x$ and $y$ are \lq S-equivalent').  Moreover the loci $X^s$ and $X^{ss}$ can be described using the Hilbert--Mumford criteria for (semi)stability: any $x \in X$ is (semi)stable for the action of $H$ if and only if it is (semi)stable for the action of every one-parameter subgroup $\lambda: \GG_m \to H$, and (semi)stability for a one-parameter subgroup is easy to describe in terms of the weights of its action.  If $X^{ss} = X^s$ then the  projective scheme $X/\!/H$ is itself a geometric quotient of $X^{ss}$ via the morphism $\phi$;  if $X^{ss} \neq X^s \neq \emptyset$ 
then there is a canonical sequence of blow-ups of $X$ along $H$-invariant closed subschemes which results in a projective scheme $\tilde{X}$ with a linear $H$-action such that $\tilde{X}^{ss} = \tilde{X}^s$, inducing a \lq partial desingularisation' $\psi_H : \tilde{X}/\!/H \to X/\!/H$ of $X/\!/H$ which is an isomorphism over $X^s/H$ \cite{K2}.
When $\kk = \CC$ the GIT quotient $X/\!/H$ can be identified with the symplectic quotient $\mu^{-1}(0)/K$ where $\mu$ is a moment map for the action of a maximal compact subgroup $K$ of $H$ on $X$ \cite{K}.

Some aspects of  Mumford's GIT have been made to work when $H$ is not reductive (cf. for example 
\cite{BDHK, DK,F2,F1,GP1,GP2,KPEN,W}), although GIT cannot be extended directly to non-reductive linear algebraic group actions since the algebra of invariants $\bigoplus_{k \geq 0} H^{0}(X,L^{\otimes k})^H$ is not necessarily finitely generated as a graded algebra. 
We can still define (semi)stable loci $X^{ss}$ and $X^s$, the latter having a geometric quotient $X^s/H$ which is an open subset of an \lq enveloping quotient' $X\env H$ with an $H$-invariant morphism $\phi: X^{ss} \to X\env H$, and if the algebra of invariants 
 $\bigoplus_{k \geq 0} H^{0}(X,L^{\otimes k})^H$ is finitely generated then $X\env H$ is the associated projective scheme
 \cite{BDHK,DK}. However in general $X\env H$ is not necessarily projective, the morphism $\phi$ is not necessarily surjective (indeed its image may not be a subscheme of $X\env H$) and no obvious analogues of the Hilbert--Mumford criteria for (semi)stability have been found.

In \cite{BDHK2} we considered the situation when the unipotent radical $U$ of $H$ has  a semi-direct product $\hat{U} = U \rtimes \GG_m$ by the multiplicative group $\GG_m$ over $\kk$
such that the action of $\GG_m$ by conjugation on the Lie algebra of $U$ has all its weights strictly positive; we call such a $U$ a graded unipotent group. 
Given any action of $U$ on a projective scheme $X$ extending to an action of $\hU$ which is linear with respect to an ample line bundle on $X$, it  was shown in \cite{BDHK2} that, {provided} we are willing to replace the line bundle with a suitable positive tensor power and to twist the linearisation of the action of $\hU$ by a suitable (rational) character of $\hU$, and provided one additional condition is satisfied  (\lq semistability coincides with stability' in the terminology of \cite{BDHK2}), then
 the $\hU$-invariants form a finitely generated algebra, and moreover the associated projective scheme 
 $X\env \hU$ is a geometric quotient $X^{s,\hU}/\hU$ of an open subset $X^{s,\hU}$ of $X$, the stable locus of the $\hU$-action. 
Applying this result with $X$ replaced by $X \times \PP^1$ gives us a projective scheme $(X \times \PP^1) \env \hU$ which is a geometric quotient by $\hat{U}$ of a $\hat{U}$-invariant open subset of $X \times \AL^1$ and contains as an open subset a geometric quotient of a $U$-invariant open subset $X^{\hat{s},U}$ of $X$ by $U$. Here the subsets $X^{s,\hU} $ and $X^{\hat{s},U}$ of $X$ can be described using Hilbert--Mumford criteria.

More generally let $H =U \rtimes R$ be a linear algebraic group over $\kk$ with graded unipotent radical
$U$, in the sense that there is a semi-direct product $\hat{H} = H \rtimes \l(\GG_m) = U \rtimes (R \times \l(\GG_m))$ such that the adjoint action of $\l(\GG_m)$  on the Lie algebra of $U$ has all weights strictly positive and $\l(\GG_m)$ commutes with the reductive subgroup $R \cong H/U$ which is a Levi subgroup of $H$.
Suppose also that the linear action of $H$ on $X$ with respect to the ample line bundle $L$ extends to a linear action of  $\hat{H}$. Then, provided the \lq semistability coincides with stability' condition of \cite{BDHK2} is satisfied for the action of $\hU$, 
we obtain 
a projective scheme 
$$ X \wenv H = (X \times \PP^1) \env \hat{H} = ((X \times \PP^1) \env \hat{U})/\!/ R$$
which 
is a good quotient by $\hat{H}$ of an $\hat{H}$-invariant open subset of $X \times \AL^1$ and contains as an open subset a geometric quotient of an $H$-invariant open subset $X^{\hat{s},H}$ of $X$ by $H$. Moreover the geometric quotient $X^{\hat{s},H}/H$ and its projective completion $X \wenv H = (X \times \PP^1) \env \hat{H}$ can be described using Hilbert--Mumford criteria combined with S-equivalence for the induced linear action of $R \cong \hat{H}/\hU$ on $(X \times \PP^1) \env \hat{U}$.

In  this paper we will first prove a stronger version of the results of \cite{BDHK2} for which the condition that \lq semistability coincides with stability' for the $\hU$-action is replaced with a slightly less stringent condition (which can also be regarded as an interpretation  in this setting of \lq semistability coincides with stability' for the action of $\hU$; in fact in each case a better analogy with classical GIT would be to describe the condition as \lq semistability coincides with stability and the stable locus is non-empty). 

 We will also show that, when this condition  is 
  replaced with a much weaker condition on the action of $\hU$ (roughly equivalent to the stable locus being non-empty in classical GIT), then 
 we can find a sequence of blow-ups of $X$ along $\hat{H}$-invariant subschemes  (analogous to that of \cite{K2} in the reductive case) resulting in a projective scheme $\hat{X}$ with an induced linear action of $\hat{H}$ satisfying the modified \lq stability coincides with semistability' condition for its unipotent radical. In this way, considering an induced action of $\hH$ on $X \times \PP^1$ as above, we obtain a projective scheme $$X \wenv H = \widehat{X \times \PP^1} \env \hat{H}$$ which is a good quotient by $\hat{H}$ of an $\hat{H}$-invariant open subset of a blow-up of $X \times \AL^1$ and contains as an open subset a geometric quotient of an $H$-invariant open subset $X^{\hat{s},H}$ of $X$ by $H$.  After further blow-ups along $\hat{H}$-invariant projective subschemes using the methods of \cite{K2}, we can obtain another projective
completion $X \tenv H$ of $X^{\hat{s},H}/H$ which is itself a geometric  quotient (not just a good quotient) by $\hat{H}$ of an $\hat{H}$-invariant open subset of a blow-up of $X \times \AL^1$. Here the geometric quotient $X^{\hat{s},H}/H$ and its projective completions $X \wenv H$ and $X \tenv H$ have descriptions in terms of Hilbert--Mumford-like criteria, the explicit blow-up constructions  and an analogue of S-equivalence.

Finally we can deal with the situation when even this weaker condition that the $\hU$-stable locus is non-empty is not satisfied, by using \lq quotienting in stages' to reduce to the  case when the unipotent radical is commutative and choosing complements to generic stabilisers.

In \cite{BHK} this construction is used to stratify quotient stacks of the form $[X/H]$, where $X$ is a projective scheme acted on linearly by a linear algebraic group $H$ with internally graded radical, in such a way that each stratum $[S/H]$ has a geometric quotient $S/H$. This can be used to find stratifications of moduli stacks (potential applications include moduli stacks or sheaves or Higgs bundles over a projective scheme \cite{BHJK,hamilton} and of projective curves \cite{Jackson}) such that each stratum has a coarse moduli space.

Recently Edidin and Rydh \cite{ER} have modified the approach of \cite{K2} to show that if $\mathcal{X}$ is a smooth Artin stack with a stable good moduli space $\bf{X}$, then there is a canonical sequence of birational morphisms of smooth Artin stacks inducing proper birational morphisms of associated good moduli spaces whose composition is a \lq partial desingularisation' of $\bf{X}$. When $\mathcal{X} = [X^{ss}/G]$ where $X$ is a nonsingular projective scheme and $G$ is a reductive group acting linearly on $X$ with semistable locus $X^{ss}$, their construction coincides with that of \cite{K2}. Our constructions of $X \wenv H$ and $X\tenv H$ likewise fit into a potential partial generalisation of the results of \cite{ER} for which $\mathcal{X}$ is not required to be smooth, the good moduli space is not required to be stable and even the definition of a \lq good moduli space' is weakened slightly (as it must be to extend to non-reductive quotients) \cite{HK}.

As was observed in \cite{BDHK2} (see also \cite{bunnett}) , 
the automorphism group $H$ of
any complete simplicial toric variety is a linear algebraic group $H = U \rtimes R$ with internally graded unipotent radical $U$. 
Thus the results of this paper can be applied to any linear action of $H$ on an irreducible projective scheme with respect to a very ample linearisation (cf. \cite{bunnett}).
Similarly it was also shown in \cite{BDHK2} that 
 the group $\GG_{k,p}$ of $k$-jets of holomorphic reparametrisations of $(\CC^p,0)$ for any $k, p \geq 1$  has 
an internally graded unipotent radical $\UU_{k,p}$, so 
the results above also apply
to any linear action of the
reparametrisation group  $\GG_{k,p}$.
In particular $\GG_{k,p}$ acts fibrewise in a natural way on the jet bundle $J_{k,p}(T^*Y) \to Y$ 
over a complex manifold $Y$ of dimension $n$ with fibre 
$J_{k,p,x}$
at $x \in Y$ consisting of the $k$-jets at $x$ of holomorphic maps from $(\CC^p,0)$ to $(Y,x)$. 
This action in the case when $p=1$ was the original motivation for our study of graded unipotent group actions (see also \cite{bkgrosshans}), and has led to the application \cite{bkGGL}.

We can even apply the results of this paper to the Nagata counterexamples, which are linear actions of unipotent groups $U$ on  projective space such that the corresponding $U$-invariants are not finitely generated \cite{Mukai}. In these examples the linear action of $U$ extends naturally to a linear action of a semi-direct product $\hat{U} = U \rtimes \GG_m$
such that the action of $\GG_m$ by conjugation on the Lie algebra of $U$ has all its weights strictly positive, and our construction provides a finitely generated subalgebra of  the algebra of $U$-invariants on the projective space $X$ given by those which extend to $\hU$-invariants on $X \times \PP^1$ (for an appropriate choice of linearisation).

The layout of the paper is as follows. $\S$\ref{section1} gives a more precise statement of our results, while $\S$\ref{section2} reviews 
the results of \cite{BDHK,DK} on non-reductive GIT, including preliminary results which will be needed for the proof of our first main theorem, Theorem \ref{mainthm}. 
 In $\S$\ref{section4} we recall the partial desingularisation construction of \cite{K2} for linear actions of reductive groups on projective schemes, while $\S$\ref{sectiona} studies the 
 locus of $\hU$-semistable points $x \in X$ 
  where the dimension of $\stab_U(x)$ is maximal. $\S$\ref{section3} outlines the strategy which will be used to prove Theorem \ref{mainthm}, and then completes its proof. Our second main theorem, Theorem \ref{mainthm2}, which is a non-reductive analogue of the partial desingularisation construction of \cite{K2}, is proved in
$\S$\ref{section5}. A proof is also given of our final main theorem Theorem \ref{mainthm3}, which uses quotienting in stages to extend Theorem \ref{mainthm2} to the most general situation without any assumption of non-emptiness of the stable locus. Actions of non-reductive groups with externally graded unipotent radicals are studied in $\S$\ref{externalsection}, while
$\S$\ref{section6} discusses some examples and applications. In $\S$\ref{sec:applications} we give a brief summary of the proof in  \cite{bkGGL} of the Green--Griffiths--Lang and Kobayashi hyperbolicity conjectures with polynomial degree bounds, which uses a recent extension \cite{bkcoh} of cohomological intersection theory on GIT quotients to linear actions of non-reductive groups with graded unipotent radicals.  Finally Eloise Hamilton has provided an appendix to this article including algorithms and accompanying diagrams which we hope will help readers to follow the constructions in the paper. An index of notation precedes the bibliography.

The authors would like to thank Dominic Bunnett, Eloise Hamilton, Vicky Hoskins and Joshua Jackson for many very helpful discussions and suggestions during this paper's lengthy evolution.

\section{Statement of results} \label{section1}

\subsection{Externally graded unipotent radical} \label{section2.1}

As in \cite{BDHK} we will work over an algebraically closed field $\kk$ of characteristic zero, and we will adopt the conventions described in \cite{BDHK} $\S$1.1.

In order to describe our results more precisely, let $H=U \rtimes R$ be a linear algebraic group  over $\kk$ with externally graded unipotent radical
$U$, in the sense that there is a semi-direct product $\hat{H} = H \rtimes \l(\GG_m) = U \rtimes (R \times \l(\GG_m))$ of $H$ by the multiplicative group $\GG_m$ over $\kk$ such that the adjoint action of the one-parameter subgroup $\l:\GG_m \to \hat{H}$  on the Lie algebra of $U$ has all weights strictly positive, and $\l(\GG_m)$ commutes with the reductive subgroup $R \cong H/U$ which is a Levi subgroup of $H$.
Let $\mathcal{L}$ be a very ample linearisation with respect to a line bundle $L \to X$ 
of  an action of $H$ on an irreducible projective scheme $X$; we assume
that there is an extension of this linearisation to a linearisation (by abuse of notation also denoted by $\mathcal{L}$) of the action of the semi-direct product
$\hat{H}$ of $H$ by $\GG_m$. 
Let $\chi: \hat{H} \to \GG_m$ be a character of $\hat{H}$ with kernel containing $H$; we will identify such characters $\chi $ with integers so that the integer 1 corresponds to the character which defines the exact sequence $H \to \hat{H} \to \GG_m$. 
We can twist the linearisation of the $\hat{H}$-action by multiplying the lift of the $\hH$-action to $L$ by such a character; this will leave the $H$-linearisation on $L$ and the action of $\hH$ on $X$ unchanged.
Note that a linearisation $\mathcal{L}$ of $\hH$ with respect to $L$ induces a linearisation $\mathcal{L}^{\otimes m}$ with respect to the line bundle $L$, for any integer $m\geq 1$, such that twisting $\mathcal{L}$ by $\chi$ corresponds to twisting $\mathcal{L}^{\otimes m}$ by $m\chi$; GIT will be essentially unaffected, so in this way we can work with rational linearisations and rational characters.

 Let $\hU = U \rtimes \l(\GG_m) \leqslant \hat{H}$ and suppose that $\weight_{\min} = \weight_0 < \weight_{1} < 
\cdots < \weight_{\max} $ are the weights with which the one-parameter subgroup $\l(\GG_m)$ of $ \hU$ acts on the fibres of the tautological line bundle $\calo_{\PP((H^0(X,L)^*)}(-1)$ over points of the connected components of the fixed point set $X
^{\GG_m}$ for the action of $\l(\GG_m)$ on $X \subseteq \PP((H^0(X,L)^*)$. 
Let $V=H^0(X,L)^*$, and let $V_{\min}$ be the weight space of weight $\weight_{\min}$ in
$V$.
We will assume that there exist at least two distinct weights, since otherwise the action of the unipotent radical $U$ of $H$ on $X$ would be trivial, in which case the action of $H$ would be via an action of the reductive group $R\cong H/U$ and we could use classical GIT.

\begin{defn} \label{defn2.1}
Let $\chi$ be a rational character of $\GG_m$ (so that $c\chi$ lifts to a character of $\hH$ with kernel containing $H$ as above for a sufficiently divisible positive integer $c$)  such that 
$$ \omega_{\min} = \omega_0 <
{\chi} < \omega_{1}; $$
 we will call rational characters $\chi$  with this property {\em  adapted} to the linear action of $\hat{H}$, and we will call the linearisation {\em adapted} if $\omega_{0} <0 < \omega_{1} $;  we will call $\chi$  {\em borderline adapted} to the linear action of $\hat{H}$ if $\chi = \omega_{0}$, and the linearisation {\em borderline adapted} if $\omega_{0} = 0$. 
 \end{defn}
 
 The linearisation of the action of $\hat{H}$ on $X$ with respect to the ample line bundle $L^{\otimes c}$ can be twisted by the character $c\chi$ so that the weight $\omega_{\min}$ is replaced with $c(\omega_{\min}-\chi)$;
let $\mathcal{L}_\chi^{\otimes c}$ denote this twisted linearisation, which is adapted in the sense above if $\chi > \weight_{\min}$ is sufficiently close to $\weight_{\min}$.

Let $X^{s,\GG_m}_{\min+}$ denote the stable set in $X$ for the linear action of $\GG_m$ with respect to the adapted linearisation $L_\chi^{\otimes c}$. 
By the theory of variation of (classical) GIT \cite{DolgHu,Thaddeus},   $X^{s,\GG_m}_{\min+}$ 
is independent of the choice of adapted rational character $\chi$.
Indeed, by the Hilbert--Mumford criteria, we have  $X^{s,\GG_m}_{\min+} =   X^0_{\min} \setminus Z_{\min}$
where $Z_{\min}$ and $X^0_{\min}$ are defined as follows.

\begin{defn} \label{def:ExPr2} Let
\[
Z_{\min} =X \cap \PP(V_{\min})=\left\{
\begin{array}{c|c}
\multirow{2}{*}{$x \in X$} & \text{$x$ is a $\GG_m$-fixed point and} \\ 
 & \text{$\GG_m$ acts on $L^*|_x$ with weight $\omega_{\min}$} 
\end{array}
\right\}
\]
and
\[
X^0_{\min} =\{x\in X \mid p(x) \in Z_{\min}\}   \,\,\, \mbox{ for } \,\,\, p(x) = \lim_{t \to 0, \,\, t \in \GG_m} t \cdot x .
\]   
\end{defn}
Note that $X^0_{\min}$  is $\hat{U}$-invariant.
By assumption $X$ is irreducible, so it is not hard to see that in fact $Z_{\min}$ and $X^0_{\min}$, and thus also $X^{s,\GG_m}_{\min+}$, depend only on the action of $\GG_m$ on $X$ and not on the linearisation $\mathcal{L}$. Let
 $$X^{s,{\hat{U}}}_{\min+} =  \bigcap_{u \in U} u X^{s,\GG_m}_{\min+} = X^0_{\min} \setminus UZ_{\min};$$ 
 it follows that $X^{s,{\hat{U}}}_{\min+}$ depends only on the action of $\hU$ on $X$.

It will be often convenient to strengthen slightly the requirement that the linearisation is adapted.  We say that a property holds for a  linear action of $\hH$ on $X$ with respect to an ample linearisation twisted by a \lq {\em well adapted}' rational character $\chi$ if there exists $\epsilon >0$ such that if $\chi$ is any rational character of $\GG_m$ (lifting to $\hH$ with trivial restriction to $H$) with 
$$\omega_{\min} < {\chi} < \omega_{\min} + \epsilon,$$
then  the required property holds for the linearisation twisted by $\chi$.

\begin{rmk} \label{defnchio}
We can apply the theory of variation of GIT \cite{DolgHu,Thaddeus} 
 to the action on $X$ of a maximal torus $T$ of $R$ and to the action of $T \times \l(\GG_m)$.
We will identify rational characters on the torus $T$ with (rational) elements of $(\mathrm{Lie}T)^*$. 

Let $\hU_T = U \rtimes (T \times \l(\GG_m))$ act linearly on $X$, and let $\chi_0$ be a rational character of $T$. 
By variation of GIT 
there is $\epsilon >0$ such that if $\chi$ is a rational character  $(\chi_0,\eta)$ of $\hat{T} = T \times \l(\GG_m)$ 
where $\weight_{\min} < \eta < \weight_{\min} + \epsilon$ then 
$X^{ss,\hat{T},\chi}$ and $X^{s,\hat{T},\chi}$ are independent of $\eta$. We will write them as $X^{ss,\hat{T},\chi_0}_{\min+}$ and $X^{s,\hat{T},\chi_0}_{\min+}$, or just $X^{ss,\hat{T}}_{\min+}$ and $X^{s,\hat{T}}_{\min+}$ if no confusion will result. We can now let
$$X^{s,{\hat{H}}}_{\min+} =  \bigcap_{h \in H} h X^{s,\hat{T}}_{\min+},$$
and 
 $$X^{ss,{\hat{H}}}_{\min+} =  \bigcap_{h \in H} h X^{ss,\hat{T}}_{\min+}.$$
\end{rmk}

The first theorem proved in this paper requires the following hypothesis, which can be regarded as a version of the condition that \lq semistability coincides with stability and the stable locus is non-empty' for an adapted 
$\hat{U}$-linearisation $L \to X$:
\begin{equation}  \label{eq:sss}
\text{$\stab_U(z) = \{e\}$ for every $z \in Z_{\min}$.} \tag{$ss\! =\! s\! \neq\! \emptyset \, [\hU]$}           
\end{equation}
Since $U$ is normalised by $\GG_m$, this is equivalent (see Remark \ref{maxdimstab}) to the requirement that
$$ 
\text{$\stab_{U}(x) = \{ e \} $ for every $x \in X^0_{\min}$}. 
$$
Recall that if $W$ is a quasi-projective scheme acted on by $H$ then a good quotient for the action is an affine $H$-invariant morphism $f:W \to Q$ to a quasi-projective scheme $Q$ such that 
$$(f_*\mathcal{O}_W)^H = \mathcal{O}_Y,$$
and such that if $W_1$ and $W_2$ are disjoint $H$-invariant closed subsets of $W$ then $f(W_1)$ and $f(W_2)$ are disjoint closed subsets of $Y$. Recall also that a good quotient is a categorical quotient, and that the classical GIT quotient $X/\!/G$ of a projective scheme $X$ by a linear action of a reductive group $G$ with respect to an ample linearisation $\mathcal{L}$ is a good quotient of the semistable locus $X^{ss}$.

\begin{thm}[GIT for linear algebraic group actions with graded unipotent radicals for which semistability coincides with stability and the stable locus is non-empty]  \label{mainthm} 
Let $H=U \rtimes R$ be a linear algebraic group  over $\kk$ with externally graded unipotent radical
$U$, so that there is a semi-direct product $\hat{H} = H \rtimes \l(\GG_m) = U \rtimes (R \times \l(\GG_m))$ such that the adjoint action of the one-parameter subgroup $\l:\GG_m \to \hat{H}$  on the Lie algebra of $U$ has all weights strictly positive.
Suppose that ${\hat{H}}$  acts linearly on an irreducible projective scheme $X$ with respect to a very  ample line bundle
 $L$. 
Suppose also that the linear action of ${\hat{U}}$ on $X$ satisfies the condition  ($ss\! =\! s\! \neq\! \emptyset\, [U]$) above. 
Then \\    
(i) the open subscheme $X^{s,{\hat{U}}}_{\min+}$ of $X$ has a geometric quotient $X \env \hU = X^{s,{\hat{U}}}_{\min+}/\hat{U}$ by $\hU$ which is a projective scheme, while \\ 
(ii) the open subscheme $X^{s,{\hat{H}}}_{\min+}$ of $X$ has a geometric quotient by $\hH$, and $X^{ss,{\hat{H}}}_{\min+}$ has a good quotient $X \env \hH$ by $\hH$ which is also a projective scheme.

If the linearisation is twisted by a well adapted rational character then for a sufficiently divisible positive integer $c$ we have 

\noindent (iii) 
the associated algebra of invariants $\oplus_{m=0}^\infty H^0(X,L^{\otimes cm})^{\hat{U}}$ is 
finitely generated, with associated projective scheme $X\env \hat{U}$, and

\noindent (iv) 
the algebra of invariants  $\oplus_{m=0}^\infty H^0(X,L^{\otimes m})^{\hat{H}} = (\oplus_{m=0}^\infty H^0(X,L^{\otimes cm})^{\hat{U}})^{R}$ is 
finitely generated, with associated projective scheme
 $X\env \hat{H}$. 
\end{thm}

\begin{rmk} \label{rmkER}
Under the hypotheses of  Theorem \ref{mainthm},  $Z_{\min}$ is  a geometric quotient for the action of $\hU$ on the $H$-invariant subscheme $U Z_{\min}$ of $X$, and its reductive GIT quotient $Z_{\min}/\!/R$ by the action of $R \cong H/U$ is a good quotient for the action of $\hH$ on the $\hH$-invariant open subscheme $U Z_{\min}^{ss,R}$ of $U Z_{\min}$. 

In addition  $Z_{\min}$ is a good quotient for the action of $\hat{U}$ on $X_{\min}^0$, and $Z_{\min}/\!/R$ is a good quotient for the action of $\hat{H}$ on the open subscheme $p^{-1}(Z_{\min}^{ss,R})$ of $X_{\min}^0$, but these are in general far from being geometric quotients, as they identify $x$ with $p(x)$ which in general lies in the boundary of the orbit closure of $x$. Typically $\dim Z_{\min} < \dim X - \dim \hat{U}$ and $\dim Z_{\min}/\!/R  < \dim X - \dim \hat{H}$.

\end{rmk}

\begin{rmk}
The main theorem of \cite{BDHK2}  is a version of Theorem \ref{mainthm} above. It concerns a linear action of a linear algebraic group $\hat{H}$ with graded unipotent radical on a complex projective scheme $X$, which is well adapted in the sense above and satisfies a variant of the condition that \lq semistability coincides with stability' for the linear action of $\hU$.
\end{rmk}

\begin{rmk}
In order to obtain Theorem \ref{mainthm} (ii) and (iv), the condition ($ss\! =\! s\! \neq\! \emptyset\, [U]$) can be weakened to 
\begin{equation}  \label{eq:sssH}
\text{$\stab_U(z) = \{e\}$ for every $z \in Z_{\min}^{ss,R}$} \tag{$ss\! =\! s\! \neq\! \emptyset\, [U\!/\!ssR]$}
\end{equation}
where $Z_{\min}^{ss,R}$ is the semistable locus for the induced linear action of $R=H/U$ on $Z_{\min}$.
\end{rmk}

\begin{rmk}  \label{remweaken}
As is explained in Remark \ref{remweaken3}, the conditions ($ss\! =\! s\! \neq\! \emptyset\, [U]$) and $(ss\! =\! s\! \neq\! \emptyset\, [U\!/\!ssR])$ can be weakened, when the action of $U$ is such that the stabiliser in $U$ of $x \in X$ is always strictly positive-dimensional, to a condition abbreviated as \lq semistability coincides with Mumford stability and the Mumford stable locus is non-empty'. Provided that for every $z \in Z_{\min}$ the dimension of $\stab_U(z)$ is equal to the generic dimension of $\stab_U(x)$ for $x \in X$ (which is equivalent to the requirement that $\dim \stab_U(x)$ is constant for $x \in X^0_{\min}$), and that the same is true for subgroups $U^{(j)}$ appearing in a series $ U \geqslant U^{(1)} \geqslant \cdots \geqslant U^{(s)} = \{e\}$ normalised by $H$ with $U^{(j)}/U^{(j+1)}$ abelian (such as the derived series for $U$), then the conclusions of Theorem \ref{mainthm} still hold, by quotienting in stages (see $\S$5) to reduce to the case when $U$ is abelian. This case can in turn be reduced to the situation of Theorem \ref{mainthm}  by locally choosing complements in $U$ to the stabiliser subgroups (see Remark \ref{remweaken3}).

Similarly if  for every $z \in Z_{\min}^{ss,R}$ the dimension of $\stab_U(z)$ is equal to the generic dimension of $\stab_U(x)$ for $x \in X$, and that the same is true for subgroups $U^{(j)}$ appearing in a series $ U \geqslant U^{(1)} \geqslant \cdots \geqslant U^{(s)} = \{e\}$ normalised by $H$ with $U^{(j)}/U^{(j+1)}$ abelian, then Theorem \ref{mainthm} (ii) and (iv) still hold.
\end{rmk}

\begin{rmk} \label{remweaken2} Let $X^{\l(\GG_m)}$ be the  $\l(\GG_m)$--fixed point set in $X$. 
We can also weaken the requirement that the rational character $\chi$ should be well adapted if we strengthen the condition $(ss\! =\! s\! \neq\! \emptyset\, [U])$ to become ($ss\! =\! s\! \neq\! \emptyset\, [U]_j$) given by 
$$ \text{$\stab_U(z) = \{e\}$ whenever $z \in X^{\l(\GG_m)}$ and $\l(\GG_m)$ acts on $L^*|_z$ with weight at most $\weight_j$;} 
$$
then we can allow ${\chi}$ to be any non-weight in $( \weight_{\min}, \weight_{j} + \epsilon)$ 
where $\epsilon >0$ is sufficiently small, provided that we modify $X^{s,{\hat{U}}}_{\min+}$, $X^{s,{\hat{H}}}_{\min+}$  and $X^{ss,{\hat{H}}}_{\min+}$ appropriately, depending on $\chi$.
\end{rmk}

Having proved Theorem \ref{mainthm} we will modify the partial desingularisation construction of  \cite{K2} in the reductive case to prove the following analogue for linear algebraic groups with graded unipotent radical, which applies to any well adapted linear action of  $\hH$  for which the generic stabiliser  in $U$ of $x \in Z_{\min}$ is trivial. Recall that when $H $ is a linear algebraic group over $\kk$  then $H$ is said to have a graded unipotent radical $U$ if there is a semi-direct product $\hat{H} = H \rtimes\GG_m$  where $\GG_m$ acts by conjugation on the Lie algebra of $U$ with all weights strictly positive, and with trivial induced conjugation action on $R=H/U$.

\begin{thm}[Projective GIT completions for quotients by linear actions with non-empty $\hU$-stable loci by linear algebraic groups with graded unipotent radicals]  \label{mainthm2} 
Let $H=U \rtimes R$ be a linear algebraic group  over $\kk$ with externally graded unipotent radical
$U$, so that there is a semi-direct product $\hat{H} = H \rtimes \l(\GG_m) = U \rtimes (R \times \l(\GG_m))$ such that the adjoint action of the one-parameter subgroup $\l:\GG_m \to \hat{H}$  on the Lie algebra of $U$ has all weights strictly positive, and let $T$ be a maximal torus of the Levi subgroup   $R \cong H/U$ of $H$.
Suppose that $\hH$  acts linearly on an irreducible projective scheme $X$ with respect to a very ample line bundle $L$, and that $\stab_U(x) = \{e\}$ for generic $x \in Z_{\min}$. Then \\
(i)  there is a sequence of blow-ups of $X$ along $\hat{H}$-invariant projective subschemes  resulting in a projective scheme $\hat{X}$ with a well adapted linear action of $\hat{H}$ (with respect to a power of an ample line bundle given by tensoring the pullback of $L$ with small multiples of the exceptional divisors for the blow-ups) which satisfies the  condition that \lq semistability coincides with stability' for the action of $U$ in the sense of (\ref{eq:sss}) above, so that Theorem \ref{mainthm} applies, giving us a projective geometric quotient
$$\hat{X} \env \hat{U} = \hat{X}^{s,\hU}_{\min+}/\hU$$
and a projective good quotient $\hat{X} \env \hH = 
(\hat{X} \env \hat{U} ) /\!/ R$ of 
the $\hH$-invariant open subset $ \bigcap_{h \in H} h \, \hat{X}^{ss,\hat{T}}_{\min+} $ of $\hat{X}$; \\
(ii)  there is a sequence of further blow-ups along $\hat{H}$-invariant projective subschemes resulting in a projective scheme $\tilde{X}$ satisfying the same conditions as $\hat{X}$ and in addition the enveloping quotient $\tilde{X}\env \hat{H} $ is the geometric quotient by $\hH$ of  the  $\hH$-invariant open subset $\tilde{X}^{s,\hH}_{\min+}$;\\
(iii) the blow-down maps $\hat{\psi}:\hat{X} \to X$ and  $\tilde{\psi}:\tilde{X} \to X$ are isomorphisms over the open subscheme
$$ 
\{ x \in X^0_{\min} \setminus UZ_{\min} \,\,  | \,\,  \stab_U(p(x)) = \{ e \} \mbox{ and } p(x) \in Z_{\min}^{s,R} \},$$
and this has a geometric quotient by $\hH$ which can be identified via $\hat{\psi}$ and $\tilde{\psi}$ with open subschemes of the projective schemes $\hat{X}\env \hat{H} $ and $\tilde{X}\env \hat{H} $.
\end{thm}

\begin{rmk}
Here the first step in the construction of $\hat{X}$ from $X$ is to blow $X$ up along the closure of the subscheme of $X^0_{\min}$ where $\dim\stab_U(x)$ is maximal for $x \in X^0_{\min}$; by repeating this step finitely many times we obtain $\hat{X}$ satisfying \eqref{eq:sss}. In order to construct $\tilde{X}\env \hH$ 
 we apply the partial desingularisation construction of \cite{K2} to the induced linear action of $R\cong H/U \cong \hH/\hU$ on $\hat{X}\env \hU$.
\end{rmk}

\begin{rmk} \label{remnovember}
Under the hypotheses of Theorem \ref{mainthm2}, we can obtain an $H$-invariant open subscheme $X^{\hat{s},H}$ of $X$ with a geometric quotient $X^{\hat{s},H}/H$ and projective completion $X \wenv H = (X \times \PP^1) \env \hat{H}$ by applying this theorem to the action of $\hat{H}$  on $X \times \PP^1$. Here $\hat{H} = H \rtimes \l(\GG_m)$ acts diagonally on $X \times \PP^1$ with the given action on $X$ and multiplication by a character $\hat{H} \to \GG_m$ with kernel $H$ on $\PP^1$, and we use the linearisation of the action given by $\mathcal{L} \otimes \mathcal{O}_{\PP^1}(m)$ for $m>\!> 1$. 


\end{rmk}

\begin{rmk} \label{rmk0.6}
The conditions that $\stab_U(x) = \{e\}$ for generic $x \in Z_{\min}$ or for generic $x \in X$ are roughly  analogous to the requirement  studied in \cite{K2} that the stable set should be nonempty for a GIT quotient by a reductive group $G$. If there are semistable points but no stable points  in the reductive case, then the blow-up process of \cite{K2} can still be applied, but it will not terminate with $\tilde{X}$ such that $\tilde{X}^{ss} = \tilde{X}^s \neq \emptyset$.

Recall that in Mumford's terminology \cite{GIT} \lq stable' has a slightly weaker meaning than is now standard. For Mumford a point $x \in X$ is stable for a linear action of a reductive group $G$ with respect to an ample line bundle $L$  if there is an invariant section $\sigma \in H^0(X,L^{\otimes m})^G$ for some integer $m>0$ such that $x \in X_{\sigma} = \{ y \in X | \sigma(y) \neq 0\}$ and the action of $G$ on the affine open $X_{\sigma}$ is closed; a stable point in the modern sense (with the additional requirement that $\dim \stab_G(x) = 0$) was called \lq properly stable' by Mumford. If $X^{Ms}$ denotes the stable locus in Mumford's sense, then $X^{Ms}$ has a geometric quotient $X^{Ms}/G$. If $X$ is irreducible then $X^s \neq \emptyset$ implies that $X^{Ms} = X^s$, but it is possible to have $X^{Ms} \neq \emptyset = X^s$, and in this case the partial desingularisation construction of \cite{K2} terminates with $\tilde{X}$ satisfying $\tilde{X}^{ss} = \tilde{X}^{Ms}$ which has a geometric quotient $\tilde{X}/\!/G = X^{Ms}/G$.
Likewise (cf. Remark \ref{remweaken}; see Theorem \ref{mainthm3} below and $\S$\ref{section5}  for more details),  the condition that $\stab_U(x) = \{ e \}$ for generic $x \in Z_{\min}$  in Theorem \ref{mainthm2} can be weakened to $\stab_U(x) = \{e\}$ for generic $x \in X$, or dropped altogether, and we can still blow $X$ up to obtain projective schemes $\hat{X}$ and $\tilde{X}$ with enveloping quotients $\hat{X} \env \hH$ and $\tilde{X}\env \hH$ which are good and geometric quotients of open subsets of $\hat{X}$ and $\tilde{X}$.
\end{rmk}

\begin{rmk}
As was noted in Remark \ref{rmkER}, $Z_{\min}$ (respectively $Z_{\min}/\!/R$) is a good quotient for the action of $\hat{U}$ (respectively $\hat{H}$) on $X_{\min}^0$ (respectively $p^{-1}(Z_{\min}^{ss,R})$). These are far from being geometric quotients. 
However the quotient stacks $[Z_{\min}/{U}]$ and $[(Z_{\min}/\!/R)/{U}]$ (where $\hat{U}$ acts trivially on $Z_{\min}$ and $Z_{\min}/\!/R$) satisfy the two conditions required for algebraic spaces to be good moduli spaces for the stacks 
$[X_{\min}^0/\hat{U}]$ and $[p^{-1}(Z_{\min}^{ss,R})/\hat{H}]$ in the sense of \cite{Alper, ER}, although they are of course not algebraic spaces. The constructions of $\hat{X}\env \hat{U}$ and $\tilde{X}\env \hat{H}$ can be regarded as the result of applying the Edidin--Rydh construction of canonical reduction of stabilisers to these quotient stacks \cite{HK}.
\end{rmk}

\begin{rmk}
We can apply the results of this paper even to the Nagata counterexamples, which are linear actions of unipotent groups $U$ on  projective space such that the corresponding $U$-invariants are not finitely generated \cite{Mukai}. In these examples the linear action extends to a linear action of  $\hat{U} = U \rtimes \GG_m$ where
 the action of $\GG_m$ by conjugation on the Lie algebra of $U$ has all its weights strictly positive, and $\stab_U(x) = \{e\}$ for generic $x \in Z_{\min}$, so Theorem \ref{mainthm2} applies. Thus the quotient $X \wenv U  = \widehat{X \times \PP^1} \env \hU$ gives us a projective completion of a geometric quotient by $U$ of a $U$-invariant open subset of the projective space which can be determined by Hilbert--Mumford-like criteria and the explicit blow-up construction. Here the subalgebra of $\bigoplus_{k \geq 0} H^{0}(X,L^{\otimes k})^U$ consisting of the $U$-invariants on the projective space $X$ which extend to $\hU$-invariants on $X \times \PP^1$ (for an appropriate choice of linearisation; see Remark \ref{remnovember}) is finitely generated, even though the algebra $\bigoplus_{k \geq 0} H^{0}(X,L^{\otimes k})^U$ itself is not finitely generated.
\end{rmk}

 Theorem \ref{mainthm} will follow from Theorem \ref{thm:ExCo1} below, combined with classical GIT, by quotienting in stages by the action of  $H=U \rtimes R$. We can first consider the action of  $\hU = U \rtimes \GG_m$ and then the reductive group $R/\lambda(\GG_m) = H/\hU$; however, in order to obtain a Hilbert--Mumford description for a suitable notion of (semi)stability for the $H$-action, in Theorem \ref{thm:ExCo1} we do not simply consider the action of $\hU$ but more generally a linear action of a semi-direct product $\hU_T = U \rtimes (T\times \l((\GG_m))$ where $T$ is  a torus and the one-parameter subgroup $\l: \GG_m \to \hU_T$  has adjoint action on the Lie algebra of $U$ with only strictly positive weights. We will apply this when $T$ is a maximal torus of the Levi subgroup $R \cong H/U$ of $H$.

\begin{thm} \label{thm:ExCo1}
Assume the hypotheses of Theorem \ref{mainthm} are satisfied with $R=T$ a torus and $\hH = \hU_T = U \rtimes (T \times \l(\GG_m))$. Then
\begin{enumerate}
\item \label{itm:ExCo1-1} we have $X^0_{\min} \subseteq X^{\rms,U}$,
  with the restriction of the enveloping quotient map for $U \act L \to X$
  defining a geometric $U$-quotient $X^0_{\min} \to
  X^0_{\min}/U$;   
\end{enumerate}
\begin{enumerate} \setcounter{enumi}{1}
\item \label{itm:ExCo1-2} there are equalities
  $X^0_{\min} \setminus (U Z_{\min}) =   X^{ss,\hU }= X^{s,\hU} = X^{s,\hU}_{\min+}$ for the action of $\hU = U \rtimes \l(\GG_m)$, and  $$X^{ss,\hU_T} = \bigcap_{u \in U} uX^{ss,\hat{T}}_{\min+}, \,\,\, \,\,\,\,\,\, X^{s,\hU_T} = \bigcap_{u \in U} uX^{s,\hat{T}}_{\min+};$$   
\item \label{itm:ExCo1-3} the enveloping quotient $X \env \hat{U}_T$ is a projective scheme and for suitably divisible
integers $c>0$ the algebra of invariants $\bigoplus_{k \geq 0} H^0(X,L^{\ten kc})^{\hat{U}_T}$
for the well adapted linearisation  
is finitely generated, with 
\[
X \env \hat{U}_T = \proj(\bigoplus_{k \geq 0} H^0(X,L^{\ten kc})^{\hat{U}_T});
\]
and 
\item \label{itm:ExCo1-4} the enveloping
quotient map $q_{\hU}:X^{\ssfg,\hU} \to X \env \hat{U}$  
is a geometric quotient for the $\hat{U}$-action on
$X^{\ssfg,\hU}$, while the enveloping
quotient map $q_{\hU_T}: X^{\ssfg,\hU_T} \to X \env \hat{U}_T$  
is a good quotient for the $\hat{U}_T$-action on
$X^{\ssfg,\hU_T}$, with $q_{\hU_T}(x) = q_{\hU_T}(y)$ for $x,y \in X^{\ssfg,\hU_T} $ if and only if the closures of the $\hU_T$-orbits of $x$ and $y$ meet in $X^{\ssfg,\hU_T} $.  
\end{enumerate}
\end{thm} 

\begin{rmk}
Under the hypotheses of Theorem \ref{thm:ExCo1},  $UZ_{\min}$ is a closed subscheme of the open subscheme $X^0_{\min}$ of $X$ (see $\S$\ref{usweepsection}), and $Z_{\min}$ is a geometric quotient for the action of $U$ (and of $\hU$) on this closed subscheme via the morphism $p: X^0_{\min} \to Z_{\min}$ defined at Definition \ref{def:ExPr2}. Thus $X^0_{\min}$ has a stratification into two strata $UZ_{\min}$ and $X^0_{\min} \setminus UZ_{\min}$, each of which has a projective geometric quotient for the action of $\hU$.
\end{rmk}

\begin{rmk}
The proof of Theorem \ref{thm:ExCo1}(1) will in fact show \emph{without} using the condition \eqref{eq:sss} that $ X^{\rms,U}$ contains the open subscheme
$$ \{ x \in X^0_{\min} | \stab_U(p(x)) = \{ e \} \}
$$ of $X^0_{\min}$.
\end{rmk}

\begin{rmk} \label{(...)} 
In order to cover the most general situation without any assumptions on the $U$-stabilisers, we use quotienting in stages (see $\S$5 for more details). We can always find subgroups
\begin{equation} \label{UUU}  U = U^{(0)} \geqslant U^{(1)} \geqslant \cdots \geqslant U^{(s)} = \{ e\} \end{equation}
of $U$ which are all normal in $\hH$, have abelian subquotients $U^{(j)}/U^{(j+1)}$ such that for each $j$ the adjoint action of $\l(\GG_m)$ on the Lie algebra of $U^{(j)}/U^{(j+1)}$ has a single weight space, and which satisfy the following property for natural numbers $r_1, \ldots, r_s$:

If $1 \leqslant j \leqslant s$ and $z\in X^{\l(\GG_m)}$ where $\l(\GG_m)$ acts on $L^*|_z$ with weight strictly less than $\weight_{r_j}$ then 
$$ U^{(j)} = U^{(j+1)} \stab_{U^{(j)}}(z),$$
whereas for generic $z\in X^{\l(\GG_m)}$ where $\l(\GG_m)$ acts on $L^*|_z$ with weight equal to $\weight_{r_j}$ then 
$$ \stab_{U^{(j)}}(z) \leqslant U^{(j+1)}.$$
\end{rmk} 

By applying Remark \ref{remweaken} and using induction on $s-j$ to quotient successively by $U^{(j)}/U^{(j+1)}$ after suitable blow-ups (see Remark \ref{remweaken5}), we obtain the following general result.

\begin{thm}[Projective GIT completions for quotients by linear algebraic group actions with graded unipotent radicals: general case]  \label{mainthm3} 
Let $H = U \rtimes R$ be a linear algebraic group over $\kk$ with externally graded unipotent radical $U$ and Levi subgroup $R$ having maximal torus $T$,
 and let $\hat{H} = H \rtimes \GG_m$ define the grading.
Suppose that $\hH$  acts linearly on an irreducible projective scheme $X$ with respect to a very ample line bundle $L$. 
Then for any choice of subgroups
$$U = U^{(0)} \geqslant U^{(1)} \geqslant \cdots \geqslant U^{(s)} = \{ e\}$$
which are normal in $\hH$ and satisfy the conditions of \eqref{UUU} above, \\
(i)  there is a sequence of blow-ups of $X$ along $\hat{H}$-invariant projective subschemes  resulting in a projective scheme $\hat{X}$ with a well adapted linear action of $\hat{H}$ (with respect to a power of an ample line bundle given by tensoring the pullback of $L$ with small multiples of the exceptional divisors for the blow-ups) which satisfies the  condition that \lq semistability coincides with Mumford stability and the Mumford stable locus is non-empty' for the action of $U$ in the sense of Remark \ref{remweaken}, so that the conclusions of Theorem \ref{mainthm} hold, giving us a projective geometric quotient
$$\hat{X} \env \hat{U} = \hat{X}^{s,\hU}_{\min+}/\hU$$
and a projective good quotient $\hat{X} \env \hH = 
(\hat{X} \env \hat{U} ) /\!/ R$ of 
the $\hH$-invariant open subscheme $ \bigcap_{h \in H} h \, \hat{X}^{ss,T}_{\min+} $ of $\hat{X}$; \\
(ii)  there is a sequence of further blow-ups along $\hat{H}$-invariant projective subschemes resulting in a projective scheme $\tilde{X}$ satisfying the same conditions as $\hat{X}$ and in addition the enveloping quotient $\tilde{X}\env \hat{H} $ is the geometric quotient by $\hH$ of  the  $\hH$-invariant open subset $\tilde{X}^{s,\hH}_{\min+}$;\\
(iii) the blow-down maps $\hat{\psi}:\hat{X} \to X$ and  $\tilde{\psi}:\tilde{X} \to X$ are isomorphisms over an $\hH$-invariant open subscheme of
$ X^0_{\min} \setminus UZ_{\min} 
$
with a geometric quotient by $\hH$ which can be identified via $\hat{\psi}$ and $\tilde{\psi}$ with open subschemes of the projective schemes $\hat{X}\env \hat{H} $ and $\tilde{X}\env \hat{H} $.
\end{thm}

By applying these results to an appropriate linear action of $\hH$ on $X \times \PP^1$ (see $\S$8 for details), we obtain an $H$-invariant open subscheme of $X$ (called $X^{\hat{s},H}$ when $\stab_U(z)$ is trivial for generic $z \in Z_{\min}$)  with a geometric quotient $X^{\hat{s},H}/H$ by $H$, which can be identified with an open subscheme of $(\widehat{X \times \PP^1})\env \hH$.

\subsection{Internally graded unipotent radical} \label{section2.2}
 Our results have been stated in $\S$\ref{section2.1} in the externally graded situation when a linear action of $H=U \rtimes R$ extends to $$\hat{H} = H \rtimes \l(\GG_m) = U \rtimes (R \times \lambda(\GG_m))$$ and $\l :\GG_m \to \hH$ grades the unipotent radical $U$ of $H$. We can also consider the internally graded situation, when the Levi subgroup $R$ of $H$ has a central one-parameter subgroup $\l_{int}:\GG_m \to Z(R)$ which grades the unipotent radical $U$. Then $H$ has  an external grading with associated semi-direct product 
$$\hat{H} = (U \rtimes R) \times \GG_m = U \rtimes (R \times \l_{ext} (\GG_m))$$
where $\l_{ext} : \GG_m \to Z(R \times \GG_m) = Z(R) \times \GG_m \leqslant (U \rtimes R) \times \GG_m$ is given by $\l_{ext}(t) = (\l_{int}(t),t)$. Conversely, given an external grading of the unipotent radical of $H$ with associated semi-direct product $\hat{H} = H \rtimes \l(\GG_m)$, the one-parameter subgroup $\l: \GG_m \to \hat{H}$ provides an internal grading for $U$ regarded as the unipotent radical of $\hat{H}$. 

In the internally graded situation we have
$$\widehat{X \times \PP^1} \env \hH \cong \widehat{X \times \PP^1} \env H \times \GG_m \cong \hat{X} \env H$$
for suitable choices of linearisations, so the results of $\S$\ref{section2.1} can be re-interpreted as results about quotients by linear algebraic groups with internally graded unipotent radicals.

Alternatively, by replacing $H$ and $R$ with finite covers, we can assume that $R$ has the form $R = R' \times \lambda_{int}(\GG_m)$. Then $H = \widehat{H'}$ where $H' = U \rtimes R'$, and so the results of $\S$\ref{section2.1} apply directly to quotients by $H$.

\section{Non-reductive geometric invariant theory} \label{section2}

Let $X$ be an irreducible projective scheme over an algebraically closed field $\kk$ of characteristic 0 and let $H$ be a  linear algebraic group acting on $X$ with  an ample  {linearisation} $\mathcal{L}$ of the action; that is, an ample
line bundle $L$ on $X$ and a lift 
 of the action to $L$.

\begin{defn} \label{defn:s/ssred}
When $H=G$ is reductive then $y \in X$ is {\em
semistable} (and we will write $y \in X^{ss}$) for this linear action if there exists some $m > 0$
and $f \in H^0(X, L^{\otimes m})^G$ not vanishing at $y$, and $y$ is {\em stable} ($y \in X^s$) if also the action of $G$
on $X_f$ is closed with all stabilisers finite. Note that $y$ is stable in the original sense of Mumford (and we will write $y \in X^{Ms}$) if the action of $G$ on $X_f$ is closed but the stabilisers are not necessarily finite.

The subalgebra $\bigoplus_{k \geq 0} H^0(X, L^{\otimes k})^G $ of the
homogeneous coordinate ring 
$${\hat{\calo}}_L(X) = \bigoplus_{k \geq 0} H^0(X, L^{\otimes k}) $$
of $X$ is finitely generated graded 
because $G$ is reductive, and
 the GIT quotient $X /\!/ G$ is the projective scheme  $\Proj({\hat{\calo}}_L(X)^G)$. 
\end{defn}

The open subschemes $X^{ss}$ and $X^s$ of $X$ for a linear action of a reductive group $G$ are characterised by the following properties (see
\cite[Chapter 2]{GIT} or \cite{New}). 

\begin{propn} (Hilbert-Mumford criteria for reductive group actions)
\label{sss} 

\noindent (i) A point $x \in X$ is semistable (respectively
stable) for the action of $G$ on $X$ if and only if for every
$g\in G$ the point $gx$ is semistable (respectively
stable) for the action of a fixed maximal torus of $G$.

\noindent (ii) A point $x \in X$ with homogeneous coordinates $[x_0:\ldots:x_n]$
in some coordinate system on $\PP^n$
is semistable (respectively stable) for the action of a maximal 
torus of $G$ acting diagonally on $\PP^n$ with
weights $\a_0, \ldots, \a_n$ if and only if the convex hull
$$\conv \{\a_i :x_i \neq 0\}$$
contains $0$ (respectively contains $0$ in its interior).
\end{propn}

\begin{rmk} \label{remMumf}
When $X$ is quasi-projective but not projective, then we define $X^s$, $X^{Ms}$ and $X^{ss}$ as in Definition \ref{defn:s/ssred} above with the extra requirement that $X_f$ should be affine, which is automatically satisfied when $X$ is projective (cf. \cite{GIT} Definition 1.7).

\end{rmk}

Now let $H$ be any linear algebraic 
group, with unipotent radical $U$, acting linearly on a complex projective scheme $X$ with respect to an ample line bundle $L$. Then $\bigoplus_{k \geq 0} H^0(X, L^{\otimes k})^H$
is not necessarily finitely generated as a graded algebra, so it is not obvious how to extend GIT to this situation. One approach, adopted in \cite{DK}, is to reduce to studying the linear action of the unipotent radical $U$ on $X$; if we can find a sufficiently natural quotient for this $U$-action such that it inherits an induced linear action of the reductive group $R=H/U$ with respect to an induced ample linearisation, then we can apply classical GIT to this $R$-action to obtain a quotient for the original $H$-action on $X$. 

\begin{defn} (See \cite{DK} $\S$4). \label{defnssetc}
Let $I = \bigcup_{m>0} H^0(X,L^{\otimes m})^U$
and for $f \in I$ let $X_f$ be the $U$-invariant affine open subset
of $X$ where $f$ does not vanish, with ${\calo}(X_f)$ its coordinate ring. 
A point $x \in X$ is called {\em naively semistable} if 
 there exists some $f \in I$
which does not vanish at $x$, and the set of naively semistable points 
is denoted $X^{nss}= \bigcup_{f \in I} X_f$.
The {\em finitely generated semistable set} of $X$ is  $X^{ss,
fg} =  \bigcup_{f \in I^{fg}} X_f$ where
$$I^{fg} = \{f
\in I \ | \ {\calo}(X_f)^U
\mbox{ is finitely generated }   \}.$$
The set of {\em naively stable}
points of $X$ is
     $X^{ns} = \bigcup_{f \in I^{ns}} X_f$ where
$$I^{ns} = \{f
\in I^{fg} \ | \  
  q: X_f \longrightarrow
\Spec({\calo}(X_f)^U) \mbox{ is a geometric quotient} \},$$
and the set of {\em locally trivial stable} points is $ X^{lts} =
\bigcup_{f \in I^{lts} } X_f$ where
\begin{eqnarray*} I^{lts}\ \  =\ \  \{f
\in I^{fg}  \ | \  
    q: X_f \longrightarrow \Spec({\calo}(X_f)^U) \mbox{ is a locally trivial
geometric quotient} \}. \end{eqnarray*}
\label{defn:envelopquot}
The {\em enveloped quotient} of
$X^{ss,fg}$ is $q: X^{ss, fg} \rightarrow q(X^{ss,fg})$, where
$q: X^{ss, fg} \rightarrow \Proj({\hat{\calo}}_L(X)^U)$ is the natural
morphism of schemes and
$q(X^{ss,fg})$ is a dense constructible subset of the {\em
enveloping quotient}
$$X \env  U = \bigcup_{f \in I^{ss,fg}}
\Spec({\calo}(X_f)^U)$$ of $X^{ss, fg}$. 
\end{defn}

\begin{rmk}
We call 
a point $x \in X$ 
{\em stable} for the linear $U$-action if $x \in X^{lts}$ and {\em semistable} if $x
\in X^{ss, fg}$, writing $X^s$ (or $X^{s,U}$ 
 or $X^{s,U,L}$ or $X^{s,U,\mathcal{L}}$ or similar) for $X^{lts}$,
and  $X^{ss}$ (or $X^{ss,U}$ etc.) for $X^{ss,fg}$ (cf.  Theorem \ref{thm:main} below and \cite{DK} 5.3.7).

\end{rmk}

\begin{rmk} Here $q(X^{ss})$ is a constructible subset  of $X \env U $ but 
not necessarily a subscheme (cf. \cite{DK} $\S$6). However $q(X^{s})$ is an open subscheme of $X \env U$ and $q(X^{s}) = X^s/U$ is a geometric quotient of $X^s$ by $U$.
\end{rmk}

\begin{rmk} 
If ${\hat{\calo}}_L(X)^U$ is finitely generated then $X\env U$ is the projective
scheme $\Proj({\hat{\calo}}_L(X)^U)$.
\end{rmk}

\begin{rmk}    \label{remclaim} The enveloping quotient 
$X \env  U$ has quasi-projective open subschemes (\lq inner enveloping quotients' $X \inenv H$) which contain the enveloped quotient $q(X^{ss})$  and have ample line bundles pulling back to  positive tensor powers of $L$
under the natural map $q:X^{ss} \to X\env U$ (see \cite{BDHK} for details, and note that the justification for the claim in \cite{DK} that the enveloping quotient $X \env U$ is itself quasi-projective is incorrect; cf. \cite{BDHK} Remark 2.3.4). When  ${\hat{\calo}}_L(X)^U$ is finitely generated then the enveloping quotient $X\env U$ is the unique inner enveloping quotient.
\end{rmk}

Even when ${\hat{\calo}}_L(X)^U$ is finitely generated, the quotient map $q:X^{ss} \to X\env U$ is not in general surjective, and it is not immediately clear how to study the geometry of the enveloping quotient $X \env U$. One way to do this, developed in \cite{DK}, is via the concept of a reductive envelope, given in Definition \ref{defn:envelope} below.
For this, suppose that $G$ is a complex reductive group with the unipotent group $U$
as a closed subgroup,  let $G \times^U X$ denote the quotient of $G \times X$
by the free action of $U$ defined by $u(g,x)=(g u^{-1}, ux)$ for $u \in U$; this is a quasi-projective scheme by \cite{PopVin} Theorem 4.19. There
is an induced $G$-action on $G \times^U X$ given by left
multiplication of $G$ on itself, and if the action of $U$ on $X$ extends to an action of $G$ there is an isomorphism of
$G$-schemes 
\begin{equation*} \label{9Febiso} G \times^U X \cong (G/U) \times X, \quad  [g,x] \mapsto (gU, gx). 
\end{equation*}
If $U$ acts linearly on $X$ with respect to a very ample line bundle $L$ 
and linearisation $\mathcal{L}$ inducing a $U$-equivariant
embedding of $X$ in $\PP^n$, and if $G$ is a reductive subgroup of $SL(n+1; \CC)$,
then the inclusion
$$ G \times^U X \hookrightarrow G \times^U \mathbb{P}^n  \cong
(G/U) \times \mathbb{P}^n,
$$ and the trivial bundle on the quasi-affine scheme
$G/U$ induce
 a very ample $G$-linearisation (which we also denote by $L$)
 on $G \times^U X$, such that 
\begin{equation*} \label{name}  \bigoplus_{m \geq 0} H^0(G \times^U X, L^{\otimes m})^G \cong
\bigoplus_{m \geq 0} H^0(X, L^{\otimes m})^U = {\hat{\calo}}_L(X)^U.\end{equation*}

Note that if a linear algebraic group $H = U \rtimes R$ with unipotent radical $U$ is a subgroup of a reductive group $G$ then there is an induced right action of $R$ on $G/U$ which commutes with the left action of $G$. 
Similarly if $H$ acts on a projective scheme $X$ then there is an induced action of $G\times R$ on $G\times_{U}X$ with an induced $G\times R$-linearisation. The same is true if we replace the requirement that $H$ is a subgroup of $G$ with the existence of a group homomorphism $H\to G$ whose restriction to $U$ is injective.

\begin{defn} \label{def:GiSt1.3}
A group homomorphism $H \to G$ from a linear algebraic group $H$ with unipotent radical $U$ to a reductive group $G$ will be called $U$-faithful if its restriction to $U$ is injective.  
\end{defn} 

Note also that it is crucial here that $G/U$ is quasi-affine so that the trivial line bundle on $G/U$ is ample; this is true for any unipotent closed subgroup $U$ of the reductive group $G$, but is not true for an arbitrary closed subgroup $H$ of $G$.

\begin{defn} (See \cite{DK} $\S$5).
The sets of {\em Mumford stable points} and
{\em Mumford semistable points} for the action of $U$ on $X$ are $X^{ms} = i^{-1}((G \times^U
X)^s)$ and $X^{mss} = i^{-1}((G \times^U
X)^{ss})$
where $i: X \rightarrow G \times^U X$ is the inclusion given by
$x \mapsto [e,x]$ for $e$ the identity element of $G$.
 Here $(G \times^U X)^s$ and $(G \times^U X)^{ss}$ are
defined as in Remark \ref{remMumf}  for the induced linear
action of $G$ on the quasi-projective scheme $G \times^U X$.
\end{defn}

\begin{rmk} \label{remMumf2}
It is claimed in \cite{DK} that  $X^{ms}$ and $X^{mss}$ are equal and independent of the choice of $G$, but unfortunately there is an error in the proof of \cite{DK} Lemma 5.1.7 that $X^{mss} = X^{ms}$, which is related to the difference between stability in Mumford's original sense and in the modern sense (see Definition \ref{defn:s/ssred} above). Let $X^{Ms} = i^{-1}((G \times^U
X)^{Ms})$
where $i: X \rightarrow G \times^U X$ is the inclusion given by
$x \mapsto [e,x]$ for $e$ the identity element of $G$ and $(G \times^U
X)^{Ms}$ is defined as in Remark \ref{remMumf} for the induced  linear
action of $G$ on the quasi-projective scheme $G \times^U X$. Then $X^s \subseteq X^{Ms}$; moreover if $X$ is irreducible and $G$ is connected so that $G \times^U X$ is irreducible,  then $X^s \neq \emptyset$ implies $X^{Ms} = X^s$. However it can happen that $X^s = \emptyset \neq X^{Ms}$. 

The proof of \cite{DK} Lemma 5.1.7 shows in fact that $X^{mss} = X^{Ms}$. Proposition 5.1.9 of \cite{DK} proves that $X^{mss}$ is independent of the choice of $G$, while \cite{DK} Proposition 5.1.10 proves that $X^{ms} = X^s$ and therefore is independent of the choice of $G$; in neither case is \cite{DK} Lemma 5.1.7 used. Thus we can still deduce that $X^{ms}$ and $X^{mss} = X^{Ms}$ are independent of the choice of $G$.
\end{rmk}

\begin{defn} 
\label{defn:separ}
 A {\em finite separating
set of invariant sections of positive tensor powers of $L$} for the linear action
of $U$ on $X$ is a collection of invariant sections $S=\{f_1,
\ldots, f_n \}$ of positive tensor powers of $L$ such that $X^{\nss}=\bigcup_{f \in S} X_f$ and the set $S$ is
  \emph{separating} in the following sense: whenever $x,y
  \in X^{\nss}$ are distinct points and there exist 
$U$-invariant sections $g_0,g_1 \in H^0(X,L^{\ten r})^U$ (for some $r>0$) such that both $[g_0(x):g_1(x)]$ and $[g_0(y):g_1(y)]$ are defined and distinct points of $\PP^{1}$, then there are sections $f_0,f_1 \in S$ of some common
tensor power of $L$ such that $[f_0(x):f_1(x)]$ and $[f_0(y):f_1(y)]$ are defined and distinct points of $\PP^1$.

 If $G$ is any reductive group containing
$U$, a finite separating set $S$ 
of invariant sections of positive tensor powers of $L$ is a {\em
finite fully separating set of invariants} for the linear $U$-action
on $X$ if

(i) for every $x \in X^{ms}$ there exists $f \in S$ with associated
$G$-invariant $F$ over $G \times^U X$ (under the isomorphism
(\ref{name})) such that $x \in (G \times^U
X)_{F}$ and $(G \times^U X)_{F}$ is affine; and

(ii) for every $x \in X^{ss}$ there exists $f \in S$ such that $x
\in X_f$ and ${\calo}(X_f)^U\cong \kk[S]_{(f)}$ (where $\kk[S]$ is the graded subalgebra of
$\oplus_{k \geq 0} H^0(X,L^{\ten k})^U$ generated by $S$).
\end{defn}

This
definition is also independent of the choice of $G$ by \cite{DK} Remark 5.2.3.

\begin{defn}\label{defn:envelope} (See \cite{DK} $\S$5).
Let $X$ be a quasi-projective scheme with a linear $U$-action with
respect to an ample line bundle $L$ on $X$, and let $G$ be a complex reductive group
containing $U$ as a closed subgroup. A
$G$-equivariant projective completion $\overline{G \times^U X}$ of
$G \times^U X$, together with a $G$-linearisation $\mathcal{L}'$  with respect to
a line bundle $L'$ which
restricts to the given $U$-linearisation on $X$,
is a
{\em reductive envelope} of the linear $U$-action on $X$ if every
$U$-invariant $f$ in some finite fully separating set of invariants
$S$ for the $U$-action on $X$  extends to a $G$-invariant section of a tensor
power of $L'$ over $\overline{G \times^U X}$.
If moreover there exists such an $S$ for which every
$f\in S$ extends to a $G$-invariant section $F$ over $\overline{G
\times^U X}$ such that $(\overline{G \times^U X})_F$ is affine, then
 $(\overline{G \times^U X}, L')$ is a {\em fine reductive
envelope}, and if $L'$ is ample (in which case
$(\overline{G \times^U X})_F$ is always affine) it is an {\em ample
reductive envelope}.
If every $f \in S$
extends to a $G$-invariant $F$ over $\overline{G \times^U X}$ which
vanishes on each 
 codimension 1 component of the boundary of $G
\times^U X$ in $\overline{G \times^U X}$, then 
a reductive envelope for the linear
$U$-action on $X$ is called a {\em strong} reductive envelope.
\end{defn}



\begin{rmk}
In order to find projective completions of quotients of open subschemes of $X$ by the action of $H$, we should consider reductive envelopes $\overline{G
\times^U X}$ which are $G \times R$-equivariant projective completions of ${G
\times^U X}$ equipped with $G\times R$-linearisations restricting to the given linearisation on $X$.
\end{rmk}

\begin{defn}\label{defn:s/ssbar} (See \cite{DK} $\S$5 and \cite{KPEN} $\S$3).
Let $X$ be a
projective scheme with a linear $U$-action and a reductive envelope
$(\overline{G \times^U X},\mathcal{L}')$.  The set of {\em completely stable points} of $X$ with
respect to the reductive envelope is $$X^{\overline{s}} = (j
\circ i)^{-1}(\overline{G \times^U X}^s)$$ and the set of {\em
completely semistable points} is $$X^{\overline{ss}} = (j \circ
i)^{-1}(\overline{G \times^U X}^{ss}),$$
where
 $i: X \hookrightarrow G
\times^U X$ and $j: G \times^U X \hookrightarrow \overline{G
\times^U X}$ are the inclusions, and $\overline{G \times^U X}^{s}$
and $\overline{G \times^U X}^{ss}$ are the stable and semistable sets for
the linear $G$-action on $\overline{G \times^U X}$ with respect to the linearisation $\mathcal{L}'$. 
Let $$X^{\overline{nss}} = (j \circ
i)^{-1}(\overline{G \times^U X}^{nss})$$
where $y \in \overline{G \times^U X}$ 
belongs to the subset $\overline{G \times^U X}^{nss}$ of {\em naively
semistable} points for the linear action of $G$) if there exists some $m > 0$
and $f \in H^0(X, L^{\otimes m})^G$ not vanishing at $y$.
\end{defn}

Note that $X^{\overline{nss}} = X^{\overline{ss}}$ provided that the reductive envelope
is ample.

The following result combines \cite{DK} Theorems 5.3.1 and 5.3.5, corrected as at Remark \ref{remMumf2}.

\begin{thm}\label{thm:main} 
Let $X$ be a normal projective scheme with a linear $U$-action, for $U$ a
connected unipotent group, and let $(\overline{G \times^U X},L)$ be
any fine reductive envelope. Then
$$
X^{\overline{s}}  \subseteq  X^{s} = X^{ms} \subseteq X^{Ms} = X^{mss}
\subseteq  X^{ns}   \subseteq  X^{ss}  \subseteq X^{\overline{ss}} = X^{nss}
\subseteq X^{\overline{nss}}.
$$
The stable sets $X^{\overline{s}}$,
$X^{s} = X^{ms}$, $X^{Ms} = X^{mss}$ and $X^{ns}$ admit quasi-projective
geometric quotients, given by restrictions of the quotient map $q =
\pi \circ j \circ i$
where
$$\pi: (\overline{G \times^U X})^{ss} \to
\overline{G \times^U X}/\!/G$$
is the classical GIT quotient map for
the reductive envelope and $i,j$ are as in Definition \ref{defn:s/ssbar}.
 The quotient map $q $
restricted to the open subscheme $X^{ss}$ is an enveloped
quotient, and there is an open subscheme 
  $X \inenv  U$ of $\overline{G \times^U X}/\!/G$ containing $q(X^{ss})$ which is an inner enveloping quotient of $X$ by the linear action of $U$. Moreover there is an ample
line bundle $L_U$ on $X\inenv U$ which pulls back to a tensor power
$L^{\otimes k}$ of the line bundle $L$ for some $k>0$
and extends to an ample line bundle on 
$\overline{G \times^U X}/\!/G$.

If moreover $\overline{G \times^U X}$ is normal and  provides a fine strong reductive envelope for the linear
$U$-action on $X$, then $X^{\overline{s}} = X^{s}$ and $X^{ss} =
X^{nss}$.
\end{thm}

So there is  a
diagram of quasi-projective schemes
$$ \begin{array}{ccccccccccccc}
X^{\overline{s}} & \subseteq & X^{s} &
\subseteq & X^{Ms} & \subseteq & X^{ss} & \subseteq & X^{\overline{ss}} = X^{nss}\\
\downarrow &  & \downarrow &  & \downarrow & &
\downarrow & & \downarrow \\ X^{\overline{s}}/U & \subseteq &
X^{s}/U & \subseteq & X^{Ms}/U & \subseteq & X \inenv U & \subseteq &
\overline{G \times^U X}/\!/G
\end{array}
$$
where all the inclusions are open and all the
vertical morphisms are restrictions of the GIT quotient map
$\pi:(\overline{G \times^U X})^{ss} \to \overline{G \times^U X}/\!/G$,
and each except the last is a restriction of
the map of schemes $q:X^{nss} \to \Proj({\hat{\calo}}_L(X)^U)$ associated to
the inclusion ${\hat{\calo}}_L(X)^U \subseteq {\hat{\calo}}_L(X)$.
Here $\overline{G \times^U X}/\!/G$ is a projective scheme if the reductive envelope is ample but, even then, the inner enveloping quotient  $X \inenv U$ is not necessarily projective,
and (even if the ring of invariants ${\hat{\calo}}_L(X)^U$ is finitely generated
so that $X \inenv U = \Proj ({\hat{\calo}}_L(X)^U)$ is projective)
the morphisms $X^{ss} \to X \inenv U$ and 
$ X^{\overline{ss}} \to 
\overline{G \times^U X}/\!/G$
are not necessarily surjective.

Now suppose that $H$ is a  linear algebraic groups $H$ which may be neither unipotent nor reductive \cite{BDHK,bkgrosshans}.  When $H$ acts linearly on a projective scheme $X$ with respect to an ample line bundle $L$,
 the naively semistable and (finitely generated) semistable sets $X^{nss}$ and $X^{ss} = X^{ss,fg}$, enveloped and enveloping quotients and inner enveloping quotients
$$q: X^{ss} \to q(X^{ss}) \subseteq X \inenv H \subseteq X \env H$$
 are defined in \cite{BDHK} exactly as for the unipotent case, and when $H$ is reductive then $X^{nss} = X^{ss,fg} = X^{ss}$ coincides with the semistable set as defined at Definition \ref{defn:s/ssred}, while the enveloped, enveloping and inner enveloping quotients all coincide with the GIT quotient as defined at Definition \ref{defn:s/ssred}. However the definition in \cite{BDHK}  of the stable set $X^s$ is more complicated and combines (and extends) the unipotent and reductive cases.

\begin{defn} \label{def:GiSt1.1}
Let $H$ be a linear algebraic group acting on an irreducible scheme $X$ and $L \to X$ a
linearisation for the action. The \emph{stable locus} for this linearised $H$-action is the open subset
\[
X^{\rms}= \bigcup_{f \in I^{\rms}} X_f
\]
of $X^{ss}$, where $I^{\rms} \subseteq \bigcup_{r>0} H^0(X,L^{\ten r})^H$ is the subset
of $H$-invariant sections satisfying the following conditions:
\begin{enumerate}
\item \label{itm:GiSt1.1-1} the open set $X_f$ is affine;
\item \label{itm:GiSt1.1-2} the action of $H$ on $X_f$ is closed with all
  stabilisers finite groups; and
\item \label{itm:GiSt1.1-3} the restriction of the $U$-enveloping quotient map
 \[
q_{U}:X_f \to \spec((S^{U})_{(f)})
\]
is a principal $U$-bundle for the action of $U$ on $X_f$. 
\end{enumerate}
\end{defn}

If it is necessary to indicate the group $H$ we may write $X^{s,H}$ 
and $X^{ss,H}$ 
 for $X^s$ and $X^{ss}$.

In general even when the algebra of invariants $\bigoplus_{k \geq 0} H^0( X,L^{\otimes k})^H$ on $X$ is finitely generated, the morphism $X \to X\env H$ may not be surjective, so it is hard to study the geometry of $X\env H$. If, however, we are lucky enough to  find a $G \times R$-equivariant projective completion $\overline{G \times_{U} X}$ with a linearisation on $L$ such that for sufficiently divisible $N$ the line bundle $L'_N$ is  ample  and the boundary $\overline{G \times_{U} X}\setminus G\times_{U} X$ is unstable for $L'_N$, then we  have a situation which is almost as well behaved as for reductive group actions on projective schemes with ample linearisations.

\begin{defn} If $\overline{G \times^U X}$ is a $G\times R$-equivariant reductive envelope, let $X^{\barss}=X\cap \overline{G \times_{U} X}^{ss,G\times R}$ and $X^{\bars}=X\cap \overline{G \times_{U} X}^{s,G\times R}$
where $X$ is embedded in $G \times_{U} X$ in the obvious way as $x\mapsto [1,x]$. 
\end{defn}

\begin{thm}\label{thm1.13gen} {\rm (\cite{DK} Thms 5.3.1 and 5.3.5).}
Let $X$ be a normal projective scheme acted on by a linear algebraic group $H=U\rtimes R$ where $U$ is the unipotent radical of $H$ and let $L$ be a very ample linearisation of the $H$ action defining an embedding $X\subseteq \PP^n$. Let $H \to G$ be a $U$-faithful homomorphism into a reductive subgroup $G$ of $\SL(n+1;\CC)$ with respect to an ample line bundle $L$. Let $(\overline{G \times^U X},\mathcal{L}')$ be
any fine $G \times R$-equivariant reductive envelope. Then
$$
X^{\overline{s}}  \subseteq  X^{s} 
   \subseteq  X^{ss}  \subseteq X^{\overline{ss}} = X^{nss}.
$$
The stable sets $X^{\overline{s}}$ and
$X^{s} $ admit quasi-projective
geometric quotients by the action of $H$.
 The quotient map $q $
restricted to the open subscheme $X^{ss}$ is an enveloped
quotient with $q: X^{ss} \rightarrow X \env  H$ an enveloping
quotient. There is an open subscheme 
  $X \inenv  H$  of $\overline{G \times^U X}/\!/(G\times R)$ which is an inner enveloping quotient of $X$ by the linear action of $H$. Moreover there is an ample
line bundle $L_U$ on $X\inenv H$ which pulls back to a tensor power
$L^{\otimes k}$ of the line bundle $L$ for some $k>0$
and extends to an ample line bundle on 
$\overline{G \times^U X}/\!/(G\times R)$.

If furthermore $\overline{G \times^U X}$ is normal and  provides a fine strong $G \times R$-equivariant reductive envelope, then $X^{\overline{s}} = X^{s}$ and $X^{ss} =
X^{nss}$.
\end{thm}

This theorem gives us good control over the geometry of $X\env H$ when we can find a sufficiently well behaved $G \times R$-equivariant reductive envelope. Unfortunately finding such a reductive envelope is not easy in general, and may not be possible. There is, however, one situation when it is easy: when the additive group $U=\GG_a$ acts linearly on a projective space $\PP^n$. Then we can use Jordan normal form to extend the $U$-action to a linear action of $\SL(2)$, and it follows that the algebra of $U$ invariants is finitely generated (the Weitzenb\"{o}ck Theorem). Moreover we can understand the geometry of the enveloping quotient $X\env U$ by identifying it with a classical GIT quotient  $(\PP^n \times \PP^2) /\!/ \SL(2)$. 

The basic idea behind this paper is to exploit this fact by finding a sequence of normal unipotent subgroups $U_1 \leqslant U_2 \leqslant \ldots \leqslant U$ of $H$ with $\dim U_j = j$ and using induction on the dimension of $U$. However the lack of surjectivity for the quotient map $q:(\PP^n)^{ss,\GG_a} \to \PP^n\env \GG_a$ (which reflects the fact that here $G/U = \kk \setminus \{ 0 \}$ is quasi-affine but not affine) causes difficulties for the inductive argument. In order to make the induction work, we need to ensure that, when $U_1 \cong \GG_a$ is normal in $H$, the complement of the enveloped quotient $q((\PP^n)^{ss, U_1})$ in $\PP^n \env U_1$ is unstable for the induced linear action of $H/U_1$. This is the role of the one-parameter subgroup $\GG_m$ in this paper.


To complete this section we will state four more results needed for the proof of Theorem \ref{mainthm}, with a proof of the first for lack of a suitable reference, though it is doubtless well known.


\begin{lem} \label{lem:Co2Qu0} 
Let $U$ be a unipotent linear algebraic group with normal subgroup $N$ such that the
projection $U \to U/N$ splits and let $X$ be an affine $U$-scheme. Suppose $X$
has the structure of a principal $N$-bundle, and the
quotient $X/N$ is a principal $U/N$-bundle, for the canonical action of $U/N$ on
$X$. Then $X$ is a principal $U$-bundle.  
\end{lem}

\begin{proof}
Let $\pi_N:X \to X/N$ be the quotient map for the $N$-action on $X$ and
$\pi_{U/N}:X/N \to (X/N)/(U/N)$ the quotient map for the $U/N$-action on
$X/N$. Also let $U_1 \subseteq U$ be a subgroup that splits the projection $U
\to U/N$, so that
\[
N \rtimes U_1 \overset{\cong}{\longrightarrow} U, \quad (n;u) \mapsto nu
\]
where multiplication in the semi-direct product is given by $(n_1;u_1) \cdot
(n_2;u_2)=(n_1u_1n_2u_1^{-1};u_1u_2)$. Note that the composition $\pi_{U/N}
\circ \pi_N:X \to (X/N)/(U/N)=X/U$ 
is a geometric quotient for the $U$-action on $X$. Because $U$ and $N$ are
unipotent the quotients $\pi_{U/N}$ and $\pi_U$ are locally trivial in the
Zariski topology \cite[Proposition 14]{ser58}, so by choosing sufficiently fine open covers it suffices to treat the
case where $X$ and $X/N$ are trivial bundles for $U$ and $U_1$, respectively,
where we identify $U_1$ with $U/N$ in the natural way. So
let $X = N \times (X/N)$ and $(X/N)=U_1 \times (X/U)$, with the $N$-action on $X$
(respectively, $U_1$-action on $X/N$) induced by left multiplication on $N$
(respectively, $U_1$) and the quotient maps $\pi_N$ and $\pi_{U/N}$ given by
projecting to the second factor in both cases. Also let 
\[
s_N:N \times (X/N) \to N, \quad s_{U_1}:U_1 \times (X/U) \to U_1
\]  
be the projections to the first factors, and let
\[
\sigma:X/U \to N \times U_1 \times (X/U), \quad z \mapsto (e,e,z)
\] 
be the obvious section to $\pi$ (note $\pi$ is the projection to the factor $X/U$). Given $x \in X$, there are unique $n \in N$ and
$u \in U_1$ such that $x=nu\sigma(\pi(x))$. The assignments
\begin{align*}
\phi_N:X \to N, \quad &x=nu\sigma(\pi(x)) \mapsto n, \\
\phi_{U_1}:X \to U_1, \quad &x=nu\sigma(\pi(x)) \mapsto u
\end{align*}
are morphisms of schemes: for each $x \in X$ we have 
\[
\phi_{U_1}(x)=s_{U_1}(\pi_N(x)), \quad
\phi_N(x)=s_N(x)(s_N(\phi_{U_1}(x)\sigma(\pi(x))))^{-1}. 
\]
It is clear that $\phi_N$ is $N$-equivariant and $\phi_{U_1}$ is
$N$-invariant, and also that $\phi_{U_1}$ is $U_1$-equivariant and
$\phi_N(ux)=u\phi_N(x)u^{-1}$ for all $x \in X$ and $u \in U_1$. Therefore
\[
X \overset{\cong}{\longrightarrow} U \times (X/U), \quad x \mapsto
(\phi_N(x)\phi_{U_1}(x),\pi(x))
\]
defines an $U$-equivariant isomorphism, where $U \times (X/U)$ is the trivial
$U$-bundle with base $X/U$.            
\end{proof}

The following result is Lemma 3.3.1 of \cite{BDHK}.

\begin{lem} \label{lem:GiSt1}
Suppose $H$ is a linear algebraic group, $N$ is a normal subgroup of $H$ and
$X$ is an $H$-scheme (not necessarily assumed 
irreducible). Suppose all the stabilisers for the restricted
action of $N$ on $X$ are finite and this action has a geometric quotient
$\pi:X \to X/N$. Note that $H/N$ acts canonically on $X/N$. Then 
\begin{enumerate}
\item \label{itm:GiSt1-1} for all the $H/N$-orbits in $X/N$ to be closed, it is
  necessary and sufficient that all the $H$-orbits in $X$ are closed;
\item \label{itm:GiSt1-2} given $y \in X/N$, the stabiliser $\stab_{H/N}(y)$ is finite if, and
only if, $\stab_H(x)$ is finite for some (and hence all) $x \in
\pi^{-1}(y)$; and
\item \label{itm:GiSt1-3} if $H/N$ is reductive and $X/N$ is affine, then $X/N$
  has a geometric $H/N$-quotient if, and only if, all $H$-orbits in $X$ are closed.  
\end{enumerate}        
\end{lem} 

The next result is taken from \cite{ad07}.

\begin{propn} \label{prop:TrRe3.1} \cite[Theorem 3.12]{ad07} Suppose $X$ is an affine scheme acted
  upon by a unipotent group $U$ and a locally trivial quotient $X \to X/U$
  exists. Then $X/U$ is affine if, and only if, $X \to X/U$ is a trivial
  $U$-bundle.   
\end{propn}

Finally, the following result comes from Corollary 3.1.20 of \cite{BDHK}. 

\begin{propn} \label{cor:GiFi5} Suppose $H$ is a linear algebraic group, $X$ an
  irreducible $H$-scheme and $L \to X$ a linearisation. If the
  enveloping quotient $X \env H$ is projective, then 
 for suitably
  divisible integers $c>0$  
the algebra of invariants $\bigoplus_{k \geq 0} H^0(X, L^{\otimes ck})^H$ is
finitely generated and the enveloping quotient $X \env H$ is the associated
projective scheme; moreover the line bundle $L^{\otimes c}$ induces an ample
line bundle $L^{\otimes c}_{[H]}$ on $X\env H$ such that  the natural structure
map  
\[
\bigoplus_{k \geq 0} H^0(X, L^{\otimes ck})^H \to \bigoplus_{k \geq 0} H^0(X\env H, L^{\otimes ck}_{[H]})
\]
is an isomorphism. 
\end{propn}

\section{Partial desingularisations of reductive GIT quotients} \label{section4}

\label{pdsection}

Let us consider the classical situation when $G$ is a {reductive} group acting linearly on an irreducible projective scheme $X$ with respect to an ample linearisation $L$.

Suppose that $X$ has some stable points but also has semistable
points which are not stable. 
In \cite{K2} it is described how one can blow up $X^{ss}$ along a sequence of 
  $G$-invariant closed subschemes to obtain a $G$-invariant morphism 
$\psi: \tilde{X}^{ss} \to X^{ss}$ where $\tilde{X}$ is an irreducible projective scheme acted on 
linearly by $G$ such that $\tilde{X}^{ss} = \tilde{X}^s$ and $\psi$ restricts to an isomorphism over $X^s$. The induced birational 
morphism $\psi_G: \tilde{X}/\!/G \to X/\!/G$ of the geometric invariant theoretic quotients 
is an isomorphism over the geometric quotient $X^s/G$. It can be regarded as a partial desingularisation of $X/\!/G$ in the sense that if $X$ is nonsingular then the centres of the blow-ups can be taken to be nonsingular, and $\tilde{X}/\!/G$ 
has only orbifold singularities (it is locally isomorphic to the quotient of a 
nonsingular scheme by a finite group action) whereas the singularities of $X/\!/G$ 
are in general much more serious. Even when $X$ is singular, we can regard the birational 
morphism $\tilde{X}/\!/G \to X/\!/G$ as resolving (most of) the contribution to the singularities of $X/\!/G$ coming from the group action.

The set $\tilde{X}^{ss}$ can be obtained from $X^{ss}$ as follows. There exist 
semistable points of $X$ which are not stable if and only if there exists a 
non-trivial connected reductive subgroup of $G$ fixing a semistable point. 
In fact the closure in $X^{ss}$ of the $G$-orbit of any $x \in X^{ss}$ contains a unique $G$-orbit which is closed in $X^{ss}$; this closed orbit has a reductive stabiliser, so a subgroup of $G$ of maximal dimension among those occurring as stabilisers of semistable points of $X$ is reductive.
Let 
$r>0$ be the maximal dimension of a (reductive) subgroup of $G$ 
fixing a point of $X^{ss}$ and let $\calr(r)$ be a set of representatives of conjugacy 
classes of all connected reductive subgroups $R$ of 
dimension $r$ in $G$ such that 
$$ Z^{ss}_{R} = \{ x \in X^{ss} :  \mbox{$R$ fixes $x$}\} $$
is non-empty. Then
$$
\bigcup_{R \in \calr(r)} GZ^{ss}_{R}
$$
is a disjoint union of  closed subschemes of $X^{ss}$. Furthermore $$GZ^{ss}_{R} \cong G \times^{N_R} Z^{ss}_{R}$$ where $N_R$ is the normaliser of $R$ in $G$ 
(so $GZ_R^{ss}$ is nonsingular if $X$ is nonsingular). The action of 
$G$ on $X^{ss}$ lifts to an action on the blow-up $X_{(1)}$ of 
$X^{ss}$ along $\bigcup_{R \in \calr(r)} GZ_R^{ss}$ which can be linearised so that the complement 
of $X_{(1)}^{ss}$ in $X_{(1)}$ is the proper transform of the 
subset $\phi^{-1}(\phi(GZ_R^{ss}))$ of $X^{ss}$ where $\phi:X^{ss} \to X/\!/G$ is the quotient 
map (see \cite{K2} 7.17). 
Here we use the linearisation with respect to (a tensor power of) the pullback of the ample line bundle $L$ on $X$ perturbed by a sufficiently small multiple of the exceptional divisor $E_{(1)}$. This will give us an ample line bundle on the blow-up  $\psi_{(1)}:X_{(1)} \to X^{ss}$ , and (by the Hilbert--Mumford criteria for (semi)stability) if the perturbation is sufficiently small it will have the property that 
$$   \psi_{(1)}^{-1}(X^{s}) \subseteq  X_{(1)}^{s} \subseteq  X_{(1)}^{ss} \subseteq \psi_{(1)}^{-1}(X^{ss}) = X_{(1)},$$
and $\psi_{(1)}$ restricts to an isomorphism from $X_{(1)}^s \setminus E_{(1)}$ to $X^s$.
Moreover no point of $X_{(1)}^{ss}$ is fixed by a 
reductive subgroup of $G$ of dimension at least $r$, and a point in $ X_{(1)}^{ss} $ 
is fixed by a reductive subgroup $R$ of 
dimension less than $r$ in $G$ if and only if it belongs to the proper transform of the 
subscheme $Z_R^{ss}$ of $X^{ss}$.

\begin{rmk} \label{cfK2} More precisely, in \cite{K2} $X$ is blown  up along the closure $\overline{\bigcup_{R \in \calr(r)} GZ^{ss}_{R} }$ of $\bigcup_{R \in \calr(r)} GZ^{ss}_{R}$ in $X$ (or in a projective completion of $X^{ss}$ with a $G$-equivariant morphism to $X$ which is an isomorphism over $X^{ss}$; in the case when $X$ is nonsingular so that $\bigcup_{R \in \calr(r)} GZ^{ss}_{R}$ is also nonsingular, or when we replace $X$ by the ambient projective space $\PP(H^0(X,L)^*)$, we can  
 resolve the singularities of  $\overline{\bigcup_{R \in \calr(r)} GZ^{ss}_{R} }$ by a sequence of blow-ups along nonsingular $G$-invariant subschemes of the complement of $X^{ss}$ and then blow up  along the proper transform of  $\overline{\bigcup_{R \in \calr(r)} GZ^{ss}_{R} }$). We end up with  a projective scheme $\bar{X}_{(1)}$ and blow-down map $\bar{\psi}_{(1)}: \bar{X}_{(1)} \to X$ restricting to $\psi_{(1)}:X_{(1)} \to X$ where $\bar{\psi}_{(1)}^{-1}(X^{ss}) = X_{(1)}$. We can then choose a sufficiently small perturbation of the pullback to $\bar{X}_{(1)}$ of the linearisation on $X$ such that

\noindent (i) we get an ample linearisation of the $G$-action on the projective scheme $\bar{X}_{(1)}$ for which
$$   \bar{\psi}_{(1)}^{-1}(X^{s}) \subseteq \bar{X}_{(1)}^{s} \subseteq \bar{X}_{(1)}^{ss} \subseteq \bar{\psi}_{(1)}^{-1}(X^{ss}) = X_{(1)},$$
and \\
(ii) the restriction of the linearisation to $X_{(1)}$ is obtained from the pullback of $L$ by perturbing  by a sufficiently small multiple of the exceptional divisor $E_{(1)}$
\end{rmk}

We can now apply the same procedure to $X_{(1)}^{ss}$ to obtain $X_{(2)}^{ss}$ such 
that no reductive subgroup of $G$ of dimension at least 
$r-1$ fixes a point of $X_{(2)}^{ss}$. Under the assumption that $X^s \neq \emptyset$, if we repeat this process enough times, we obtain 
$X_{(0)}^{ss} = X^{ss},X_{(1)}^{ss},X_{(2)}^{ss},\ldots,X_{(r)}^{ss}$ such that
no reductive subgroup of $G$ of positive dimension fixes a point of $X_{(r)}^{ss}$,
and we set $\tilde{X}^{ss} = X_{(r)}^{ss}$. Equivalently we can construct a sequence  
$$X_{(R_0)}^{ss} = X^{ss}, X_{(R_1)}^{ss},\ldots,X_{(R_\tau)}^{ss} = \tilde{X}^{ss}$$ 
where $R_1,\ldots,R_\tau$ are connected reductive subgroups of $G$ with 
$$r= \dim R_1 \geq \dim R_2 \geq \cdots \dim R_\tau \geq 1,$$
 and if $1 \leq l \leq \tau$ then
 $X_{(R_l)}$ is the blow up of $X_{(R_{l-1})}^{ss}$ 
along its closed nonsingular subscheme $GZ_{R_l}^{ss}\cong G \times^{N_l} Z_{R_l}^{ss}$,
where $N_l$ is the normaliser of $R_l$ in $G$. Similarly 
$\tilde{X}/\!/G = \tilde{X}^{ss}/G$ can be obtained from $X/\!/G$ by blowing 
up along the proper transforms of the images $Z_R /\!/N$ 
in $X/\!/G$ of the subschemes $GZ_R^{ss}$ of $X^{ss}$ in decreasing order of $\dim R$.

The blow-down morphism $\psi_G : \tilde{X}/\!/G \to X/\!/G$ restricts to an isomorphism over the geometric quotient $X^s/G$, so both 
$ \tilde{X}/\!/G$ and $X/\!/G$ can be regarded as projective completions of $X^s/G$.

\begin{rmk} \label{stratification}
Associated to the linear action of $G$ on $X$ and an invariant inner product on the Lie algebra of $G$, there is a stratification 
$$ X = \bigsqcup_{\beta \in \mathcal{B}} S_\beta$$ of $X$ by locally closed subschemes $S_\beta$, 
indexed by a partially ordered finite subset $\mathcal{B}$ of a positive Weyl chamber for the reductive group $G$,  such that 

 (i) $S_0 = X^{ss}$, 

\noindent and for each $\beta \in \mathcal{B}$

 (ii) the closure of $S_\beta$ is contained in $\bigcup_{\gamma \geqslant \beta} S_\gamma$, and

 (iii) $S_\beta \cong G \times^{P_\beta} Y_\beta^{ss}$

\noindent  where $$\gamma \geqslant \beta \mbox{ if and only if } \gamma = \beta \mbox{ or } |\!|\gamma|\!| > |\!|\beta|\!|$$
and 
$P_\beta$ is a parabolic subgroup of $G$ acting on  a projective subscheme $\overline{Y}_\beta$ of $X$ with an open subset $Y_\beta^{ss}$ which is determined by the action of a Levi subgroup $L_\beta = \stab_G(\beta)$ of $P_\beta$ with respect to a suitably twisted linearisation \cite{K}. 

Here the original linearisation for the action of $G$ on $L \to X$ is restricted to the action of the parabolic subgroup $P_\beta$ over $\overline{Y}_\beta$, and then twisted by a rational character of $P_\beta$  for the  central one-parameter subgroup $\GG_m$ determined by $\beta$ of the Levi subgroup $L_\beta = \stab_G(\beta)$ of $P_\beta$. This one-parameter subgroup $\GG_m$ acts by conjugation with all weights strictly positive on the Lie algebra of the unipotent radical of $P_\beta$. In the notation of \cite{K} $Z_{\beta}$ denotes $Z_{\min}$ (defined as at Definition
\ref{def:ExPr2}) for the action of $P_\beta$ on $\overline{Y}_\beta$, while $Y_\beta$ plays the role of $X^0_{\min}$; also  $p_\beta: Y_\beta \to Z_\beta$ corresponds to $p: X^0_{\min} \to Z_{\min}$ and $Y_\beta^{ss} = (p_\beta)^{-1}(Z_\beta^{ss})$. Here $Z_\beta^{ss}$ (corresponding to $Z_{\min}^{ss,R}$ in the notation of this paper) is the semistable locus for the action of $L_\beta$ on $Z_\beta$, twisted by the rational character $\beta$ so that  the one-parameter subgroup $\GG_m$ acts on $Z_\beta$ with weight 0.

Since $S_{\beta} = GY_{\beta}^{ss} \cong G \times^{P_\beta} Y_\beta^{ss}$, in order  to construct a quotient by $G$ of an open subset of an unstable stratum $S_\beta$, we can study the linear action on $\overline{Y}_\beta$ of the parabolic subgroup $P_\beta$,  twisting  the linearisation by a well adapted rational character.

\end{rmk}

\begin{rmk} \label{rem:specialcases}
The partial desingularisation construction of a linear action of $G$ on $X$ has been described under the assumption that $X^s \neq \emptyset$; this is the situation considered in \cite{K2}. We can also consider the situation when $X^s = \emptyset$. There are several different cases.

$X$ is irreducible, so if $X^{ss} = \emptyset$ then in the notation of Remark \ref{cfK2}  there is an unstable stratum $S_\beta$ with $\beta \neq 0$ which is a non-empty open subscheme of $X$ and thus $X = \overline{S_\beta}$. Then, as discussed in Remark \ref{stratification}, using GIT to construct a quotient of a non-empty open subscheme of $X$ reduces to non-reductive GIT for the action of the parabolic subgroup $P_\beta$ on $\overline{Y_\beta}$.

If $X^s = \emptyset$ but $X^{ss}$ is non-empty then we can attempt to apply the inductive partial desingularisation procedure to $X^{ss}$. There are different ways in which this procedure can terminate.

It might be the case that $X^{ss} = GZ_R^{ss} \cong  G \times^{N_R} Z_R^{ss}$ for a positive-dimensional connected reductive subgroup $R$ of $G$, where $N_R$ is the normaliser of $R$ in $G$. Then $N_R$ and its quotient group $N_R/R$ are also reductive, and
$$X/\!/G \cong Z_R /\!/ N_R \cong Z_R /\!/ (N_R/R)$$
where $Z_R$ is the closed subscheme of $X$ which is the fixed point set for the action of $R$. So we can apply induction on the dimension of $G$ to study this case. Note that if in addition $Z_R^s \neq \emptyset$ where 
$$Z_R^s = \{ x \in Z_R : x \mbox{ is stable for the induced linear action of $N_R/R$ on $Z_R$ } \}, $$
then we have $X^{Ms} = GZ_R^s$ and this is precisely the situation in which $X^{Ms} \neq \emptyset = X^s$; that is, stability in Mumford's original sense differs from proper stability in Mumford's sense, which is stability in modern terminology (cf. Remarks \ref{remMumf} and \ref{remMumf2}).

If $GZ_R^{ss} \neq X^{ss} \neq X^s = \emptyset$ for every positive-dimensional connected reductive subgroup $R$ of $G$, then we can perform the first blow-up in the partial desingularisation construction to obtain $\psi_{(1)}: X_{(1)} \to X^{ss}$ such that $X_{(1)}^{ss} \subseteq X_{(1)}$ and $X_{(1)}^s = \emptyset$ as above (since $X^s_{(1)}$ is open and $X^s_{(1)} \setminus E_{(1)} = X^s = \emptyset$, where $E_{(1)}$ is the exceptional divisor).

If $X_{(1)}^{ss} = \emptyset$ then $X_{(1)}$ has a dense open subscheme $S_\beta$ for $\beta \neq 0$ and we are reduced to studying non-reductive GIT for a parabolic subgroup $P_\beta$ of $G$ as discussed above for the case when $X^{ss} = \emptyset$. If $X_{(1)}^{ss} = GZ_{(1),R}^{ss}$ for a positive dimensional connected reductive subgroup $R$ of $G$, where $Z_{(1),R}^{ss} = \{ x \in X_{(1)}^{ss} : R \mbox{ fixes }x \}$, then we can use induction on the dimension of $G$ as discussed above for the case when the corresponding statement is true of $X^{ss}$. Otherwise we can repeat the process, until it terminates in one of these two ways.

\end{rmk}

\section{The $U$-sweep of $Z_{\min}$ and stabilisers of maximal dimension}  \label{sectiona}  \label{usweepsection}

It is useful to gather together here some results which do not require the condition \eqref{eq:sss} that semistability coincides with stability for the $\hU$-action, but which, when this condition is satisfied, tell us that $UZ_{\min}$ is a closed subscheme of $X^0_{\min}$.

\begin{defn} \label{defnp}
Let $U$ be a  graded unipotent group
 and let $\hat{U} = U \rtimes \GG_m$ be the semi-direct product of $U$ by $\GG_m$ 
which defines the grading.
Suppose that $\hU$  acts linearly on an irreducible projective scheme $X$ with respect to a very ample line bundle $L$. If $d \geq 0$ let 
$$Z_{\min}^d = \{ z \in Z_{\min} \mid \dim \stab_U(z) = d \},$$
and let $d_{\max}$ be the maximal value of $d$ such that $ Z_{\min}^{d} \neq \emptyset$.
Recall that $ p: X^0_{\min} \to  Z_{\min}$ is the $\GG_m$-invariant morphism 
$$ x \mapsto   \lim_{\begin{array}{c} t \to 0\\t \in \GG_m \end{array} } t \cdot x \,\,.
$$
\end{defn}

\begin{lem} \label{lemmaA}
If $x \in X^0_{\min}$ and $u \in \hU$, then $p(ux) = p(x) $. If $x \in Z_{\min}$ and $u \in U$ then $p(x) = x$ and 
$$ ux \in  Z_{\min} \,\,\, \mbox{ iff } \,\,\, u \in \stab_U(x).$$
\end{lem}
\noindent {\bf Proof:} By definition $p$ is $\GG_m$-invariant and restricts to the identity on $Z_{\min}$. If $u = \exp(\xi)$ for some $\xi \in \Lie U$ which is a weight vector for the action of $\GG_m$, then this follows by choosing coordinates on $\PP(H^0(X,L)^*)$ with respect to which the action of $\GG_m$ is diagonal and the infinitesimal action of $\xi$ is in Jordan form. Moreover $p$ is $\GG_m$-invariant and $ Z_{\min}$ is fixed pointwise by $\GG_m$, so $U$, $\stab_U(p(x))$ and $\{ u \in U \mid up(x) \in Z_{\min} \}$ are all invariant under conjugation by $\GG_m$, and the general result follows. \hfill $\Box$

\begin{lem} \label{lemmaFA}  $U  Z_{\min}^{d_{\max}} \cap  Z_{\min} =  Z_{\min}^{d_{\max}}$, and 
if $x \in Z_{\min}$ then  $T_x(U  Z_{\min}^{d_{\max}}) \cap  T_x Z_{\min} =  T_x Z_{\min}^{d_{\max}}.$
\end{lem}
\noindent {\bf Proof:} This follows from the proof of Lemma \ref{lemmaA}. \hfill $\Box$

\begin{lem} \label{lemmaB}

The $U$-sweep $U  Z_{\min}^{d_{\max}} = \hat{U}  Z_{\min}^{d_{\max}}   $ of $ Z_{\min}^{d_{\max}}$ is a closed subscheme of $ X^0_{\min}$. 
\end{lem}
\noindent {\bf Proof:} 
By definition $X^0_{\min}$ is the open stratum of the Bialynicki-Birula stratification of $X$ associated to the $\GG_m$-action (cf. \cite{BB,K}). Thus the restriction to $ X^0_{\min} \setminus  Z_{\min} = X^{s,\GG_m}_{\min+} = X^{ss,\GG_m}_{\min+}$ of the $\GG_m$-invariant morphism $p:  X^0_{\min} \to  Z_{\min}$ factors through a projective morphism from the geometric quotient $( X^0_{\min} \setminus  Z_{\min})/\GG_m = X/\!/\GG_m$ to $ Z_{\min}$. Moreover the fibres of $p:  X^0_{\min} \to  Z_{\min}$ can be identified with the affine cone associated to the fibre of the morphism from this geometric quotient to $ Z_{\min}$. When $X$ is nonsingular then the fibres of $p:  X^0_{\min} \to  Z_{\min}$ are affine spaces and the fibres of the geometric quotient $( X^0_{\min} \setminus  Z_{\min})/\GG_m$ over $ Z_{\min}$ are weighted projective spaces, and in general this is true for the ambient projective space defined by the linearisation.

Since $\GG_m$ normalises $U$ and fixes $ Z_{\min}$ pointwise, it follows from Lemma \ref{lemmaA} that $p(U  Z_{\min}^{d_{\max}}) =  Z_{\min}^{d_{\max}}$, and if $x \in  Z_{\min}^{d_{\max}}$ then the fibre over $x$ of $p: U  Z_{\min}^{d_{\max}} \to  Z_{\min}^{d_{\max}}$ is isomorphic to $U/\stab_U(x)$. This is an affine space of dimension $\dim U - d_{\max}$ on which $\GG_m$ acts with strictly positive weights, so the induced morphism
$$ (U  Z_{\min}^{d_{\max}} \setminus  Z_{\min}^{d_{\max}})/\GG_m  \to  Z_{\min}^{d_{\max}}
$$
is a weighted projective bundle. In particular it follows that the embedding $ (U  Z_{\min}^{d_{\max}} \setminus  Z_{\min}^{d_{\max}})/\GG_m  \to ( X^0_{\min} \setminus  Z_{\min})/\GG_m$ is closed, and hence so is the corresponding embedding of affine cones $U  Z_{\min}^{d_{\max}} \to  X^0_{\min}$. \hfill $\Box$

\begin{cor} \label{cortolemmaB}

When  condition \eqref{eq:sss} is satisfied, 
the $U$-sweep $U  Z_{\min} $ of $ Z_{\min}$ is a closed subscheme of $ X^0_{\min}$. 
\end{cor}

\noindent {\bf Proof:} This follows immediately from  Lemma \ref{lemmaB} since $Z_{\min}^{d_{\max}} = Z_{\min}$
when  \eqref{eq:sss} is satisfied. \hfill $\Box$

\begin{rmk} \label{rmk:endof2}
Once we have proved Theorem  \ref{thm:ExCo1}(1), we will know that when  \eqref{eq:sss} is satisfied then $X^0_{\min}$ has a geometric quotient $X^0_{\min}/U$. It will be useful for the proof of Theorem  \ref{thm:ExCo1}(2) to know when a point of $X^0_{\min}/U$ represented by $x \in X^0_{\min}$ is fixed under the action of a one-parameter subgroup $\tilde{\lambda}: \GG_m \to \hat{U}_T/U \cong T \times \l(\GG_m)$, and which points of $X^0_{\min}/U$ are fixed by the action of $T$. Every such one-parameter subgroup is conjugate to a one-parameter subgroup of $T \times \l(\GG_m)$, so without loss of generality we may assume $\lambda:\GG_m \to T \times \l(\GG_m)$. 
If $\tilde{\lambda}(\GG_m) \subseteq \hU$ then $x \in UZ_{\min}$, so let us assume that $\tilde{\lambda}(\GG_m) \not\subseteq \hU$.
This one-parameter subgroup fixes the orbit $Ux \in X^0_{\min}/U$ if and only if $\tilde{\lambda}(\GG_m) \subseteq U.\stab_{\hU_T}(x)$, or equivalently $\xi \in \mathrm{Lie}U + \mathrm{Lie}\stab_{\hU_T}(x)$, where $\xi$ generates $\mathrm{Lie}\lambda(\GG_m)$. Since $p$ is $\hU$-invariant and $T$-equivariant, $\mathrm{Lie}\stab_{\hU_T}(x) \subseteq \mathrm{Lie}U + \mathrm{Lie}\stab_T(p(x))$, and since by assumption $\xi \in \mathrm{Lie}T$, which has weight 0 for the adjoint action of $\l(\GG_m)$ while this action has strictly positive weights on $\mathrm{LieU}$,it follows that $\xi \in  \mathrm{Lie}\stab_T(p(x))$. We are assuming that $\xi \not\in \mathrm{Lie}\hU$, so $p(x)$ is fixed by a two-dimensional subtorus $T'$ of $T \times \l(\GG_m)$ containing the one-parameter subgroup $\l(\GG_m)$ of $\hU$. This subtorus $T'$ preserves $p^{-1}(p(x))$, which by the work of Bialynicki-Birula \cite{BiB} is affine with a linear action of $T'$ and affine linear action of $\hU_T$  (when $p(x) \in \overline{T'x}$ is identified with 0).

This argument also shows that the fixed point set $(X^0_{\min}/U)^T$ for the induced action of $T$ on $X^0_{\min}/U$ is $UZ_{\min}^T/U$, where $Z_{\min}^T$ is the fixed point set for $T$ acting on $Z_{\min}$.
\end{rmk}

\begin{rmk} \label{maxdimstab}
After we have proved Theorem  \ref{thm:ExCo1}(1), we will want to consider the situation when condition \eqref{eq:sss} is not satisfied. Then there will exist $x \in X^0_{\min}$ such that $\dim \stab_U(x) >0$, and we will be interested in the closed subscheme of $X^0_{\min}$ where $\dim \stab_U(x)$ is maximal. By Lemma \ref{lemmaA}  the morphism $p:X^0_{\min} \to Z_{\min}$ is $\hU$-invariant, so if $z \in Z_{\min}$ 
then $U$ acts on the fibre $p^{-1}(z)$ of $p$. Since $\GG_m$ normalises $U$ we have $\stab_U(tx) $ is conjugate to $ \stab_U(x)$ when $x \in X^0_{\min}$ and $t \in \GG_m$, so taking the limit as $t \to 0$ it follows that $\dim \stab_U(x) \leqslant \dim \stab_U(p(x))$. Thus the maximal value $d_{\max}$ of the dimension of $\stab_U(z)$ for $z \in Z_{\min}$ is equal to the maximal value of the dimension of $\stab_U(x)$ for $x \in X^0_{\min}$, and
$$ \{ x \in X^0_{\min} \,\, | \,\, \dim\stab_U(x) = d_{\max} \} \subseteq p^{-1}(Z_{\min}^{d_{\max}}).$$
Moreover to find $ \{ x \in X^0_{\min} \,\, | \,\, \dim\stab_U(x) = d_{\max} \}$ it suffices to consider the fixed point set for the action of $\stab_U(z)$ on the fibre $p^{-1}(z)$ of $p$ for $z \in Z_{\min}^{d_{\max}}$.
\end{rmk}

\section{Quotienting in stages}

In the situation of reductive GIT we can quotient in stages. Suppose that $G_1$ is a reductive normal subgroup of a reductive group $G$, and that $G$ acts linearly on a projective variety $X$. Then
there is an induced linear action of the reductive group $G_2/G_1$ on the GIT quotient $X/\!/G_1$, and the associated GIT quotient $(X/\!/G_1)/\!/(G_2/G_1)$ is naturally isomorphic to $X/\!/G$. This follows from the fact that if $G$ acts on an algebra $A$ then the algebras of invariants satisfy $(A^{G_1})^{G/G_1} = A^{G}$.

Now let $H=U \rtimes R$ be a linear algebraic group  over $\kk$ with externally graded unipotent radical
$U$, so that there is a semi-direct product $\hat{H} = H \rtimes \l(\GG_m) = U \rtimes (R \times \l(\GG_m))$ such that the adjoint action of the one-parameter subgroup $\l:\GG_m \to \hat{H}$  on the Lie algebra of $U$ has all weights strictly positive. Suppose that $\hat{H}$ acts linearly on a projective variety $X$. We are aiming to show that in suitable circumstances the $\hU$-invariants are finitely generated and that we can construct $X\env \hat{H}$ as $(X\env \hU)/\!/R$ using the induced action of $R \cong \hat{H}/\hU$ on $X \env \hU$. 

Suppose that $H_1 = U_1 \rtimes R_1$ is a normal subgroup of $\hat{H}$, 
 where $U_1 \leqslant U$ is a normal unipotent subgroup of $\hat{H}$ and $R_1$ is a reductive normal subgroup of $R$. Then $\hat{H}_1$ is a normal subgroup of $\hat{H}$, and if the $\hat{H}_1$-invariants are finitely generated there is an induced action of $\hat{H}/\hat{H}_1$ on $X\env \hat{H}_1$. However $\hat{H}/\hat{H}_1 \cong (U/U_1) \rtimes (R/R_1)$, and in general this has no grading one-parameter subgroup. 
 
 So instead we should consider the quotient $X \wenv H$ which we want to construct as $(\PP^1 \times X) \env \hat{H}$. There is an action of $\GG_m \times \hat{H}$ on $\PP^1 \times X$ where 
$\GG_m$ acts by multiplication on $\PP^1$ and trivially on $X$, while $\hat{H} = H \rtimes \l(\GG_m)$ acts diagonally on $X \times \PP^1$ with the given action on $X$ and multiplication by a character $\hat{H} \to \GG_m$ with kernel $H$ on $\PP^1$. We use the linearisation of the action given by $\mathcal{L} \otimes \mathcal{O}_{\PP^1}(m)$ for $m>\!> 1$. Then there is an induced linear action of 
 $(\GG_m \times \hat{H})/\hat{H}_1$ on $X\wenv H_1 = (\PP^1 \times X)\env \hat{H}_1$, and 
 $$(\GG_m \times \hat{H})/\hat{H}_1 \cong (H/H_1) \rtimes (\GG_m \times \l(\GG_m) / \delta(\GG_m)) \cong (H/H_1) \rtimes \GG_m$$
 where the adjoint action of the quotient  $(\GG_m \times \l(\GG_m) / \delta(\GG_m)) \cong \GG_m$ of $\GG_m \times \l(\GG_m) $ by the diagonal subgroup $\delta(\GG_m)$ on the Lie algebra of the unipotent radical $U/U_1$ of $H/H_1$ has all weights strictly positive. 
 
 It follows that if we have finitely generated invariants for the linear action of $\hat{H}_1$ on $\PP^1 \times X$ and for the induced action of $(\GG_m \times \hat{H})/\hat{H}_1$ on $X \wenv H_1$, then we will have finitely generated invariants for the linear action of $\hat{H}$ on $\PP^1 \times X$ and
 $$ X \wenv H = (\PP^1 \times X)\env \hat{H} = (X \wenv H_1) \wenv (H/H_1)$$
 since $(\PP^1 \times \PP^1)/\delta(\GG_m) \cong \PP^1$. This means in particular that we can reduce our problem to the case when the unipotent radical of $H$ is commutative, and when furthermore the action of $\l(\GG_m)$ on its Lie algebra has a single weight space.

\section{When semistability coincides with stability for the unipotent radical}   \label{section3}
\label{ss=ssection}

In this section we are aiming to prove Theorem \ref{thm:ExCo1}.
We assume that $\mathcal{L}$ is a very ample (rational) linearisation with respect to a line bundle $L \to X$ for an action of the linear algebraic group $H = U \rtimes R$ with externally graded unipotent radical
$U$ on an irreducible projective scheme $X$, and
that there is an extension of this linearisation to a linearisation of 
$\hat{H} = H \rtimes \l(\GG_m) = U \rtimes (R \times \l(\GG_m))$. Let $T$ be a maximal torus of the Levi subgroup $R$ of $H$; in this section we will mainly consider the situation when $R+T$ and $\hH$ coincides with $\hU_T = U \rtimes (T \times \l(\GG_m))$. We let $S=\bigoplus_{k \geq 0} H^0(X,L^{\ten k})$. Note that, because $X$ is irreducible,
for any subgroup $H_0 \subseteq \hat{H}$ and any $H_0$-invariant section $f$ of a
positive tensor power of $L \to X$ we have a canonical identification
$(S^{H_0})_{(f)} = \OO(X_f)^{H_0}$; 
 we will make implicit use of this identification throughout this section. We may use any of 
the following equivalent notation for the stable locus for a linearisation $\mathcal{L}$ given by  $H_0 \act L
\to X$  (as well as using similar notation for the semistable locus etc.):  $X^{s,H_0} =  X^{s,H_0, L} = X^{s, H_0, \mathcal{L}} =  X^{\rms(H_0,L)} =  X^{\rms(H_0,\mathcal{L})} $. We will use the terminology $\GG_m$-action for the action of the one-parameter subgroup $\l:\GG_m \to \hU_Y$.

Let $\weight_{\min}$ be the minimal weight for the $\GG_m$-action on
$V=H^0(X,L)^*$ and let $V_{\min}$ be the weight space of weight $\weight_{\min}$ in
$V$. Note that, equivalently, $\weight_{\min}$ is the minimal $\GG_m$-weight for the
action  
of $\GG_m$ on fibres of $L^*$ over $\GG_m$-fixed points.

For the results of this section, it
will be necessary to require the assumption \eqref{eq:sss} on the
$\hat{U}$-linearisation $L \to X$:
\begin{equation} 
\text{$\stab_U(z) = \{e\}$ for every $z \in Z_{\min}$.} \tag{$ss\! =\! s\! \neq\! \emptyset \, [U]$}
\end{equation}

To prove Theorem \ref{thm:ExCo1}  we shall argue by establishing intermediate results to aid
readability. We first sketch the outline of the argument and
establish some preliminaries. 

If $\xi \in
\Lie U$ is a
$\l(\GG_m)$-weight vector, then it has positive weight $\ell >0$, say. If $W$ is any
representation of $\hat{U}$, then $\xi$ defines a derivation $\xi:W \to W$ and
any weight vector in $W$ of weight $\omega 
\in \mb{Z}$ gets sent to a weight vector of weight $\omega + \ell$ under $\xi$. In
particular, if $W_{\max}$ denotes the $\GG_m$-weight space in $W$ of maximal
possible weight, then we have
\[
W_{\max} \subseteq \bigcap_{\xi \in \Lie U} \ker(\xi:W \to W) = W^U.
\]

Now, given a linearisation $\hat{U} \act L \to X$, let
$H^0(X,L)_{\max}$ be the $\GG_m$-weight space in $H^0(X,L)$ of maximal possible
weight (which is equal to $-\weight_{\min}$, in our previously defined notation). The
open subset $X^0_{\min}$ is covered by the affine open 
subschemes 
$X_{\sigma}$, with $\sigma \in H^0(X,L)_{\max}$. Each $X_{\sigma}$ is invariant
under the $\hat{U}$-action, because $H^0(X,L)_{\max} \subseteq H^0(X,L)^U$ and
$\sigma$ is a $\GG_m$-weight vector, and if we choose each $\sigma$ to be a weight vector for the action of $T$, as we may, then $X_{\sigma}$ is invariant
under the $\hat{U}_T$-action. This cover of $X^0_{\min}$ by open
affines $X_{\sigma}$ enjoys a prominent r\^ole in the proof of
Theorem \ref{thm:ExCo1}. 

\begin{rmk} \label{generalise} By choosing the affine opens $X_{\sigma}$ covering $X^0_{\min}$ carefully, we can generalise  Theorem \ref{thm:ExCo1} by weakening the hypothesis \eqref{eq:sss}. Choose a normal series $U = U^{(0)} \geqslant U^{(1)} \geqslant \cdots \geqslant U^{(s)} = \{e\}$ such that each subquotient $U^{(j)}/U^{(j+1)}$ is abelian (for example the derived series of $U$). Then the hypothesis \eqref{eq:sss} can be weakened to
\begin{equation} \label{eq:ssss}
\text{$\dim \stab_{U^{(j)}}(z) = \min_{x \in X} \dim \stab_{U^{(j)}} (x)$ for every $z \in Z_{\min}$ and $j \in \{0,1,\ldots,s\}$.} 
\end{equation}

To modify the proof to apply in these circumstances, we observe  that if $U$ is abelian then each $\stab_U(z)$ has a complementary subgroup $U'$ in $U$ with $U' \cap \stab_U(z) = \{ e\}$ and $U = U'\stab_U(z)$, and that if $d^U_{0} = \min_{x \in X} \dim \stab_{U} (x)$ then 
 condition (\ref{eq:ssss}) implies that $\dim \stab_U(x) = d^U_{0}$ for every $x \in X^0_{\min}$; the latter is true since $\GG_m$ normalises $U$ and so $\dim \stab_U(p(x)) \geqslant \dim \stab_U(x)$ for any $x \in X$, where $p(x)$ is the limit of $t\cdot x$ as $t \to 0$ with $t \in \GG_m$.
We can  define a $\hU_T$-equivariant morphism from $X^0_{\min}$ to the Grassmannian $\rm{Grass}(d^U_0, \rm{Lie}U)$ of $d^U_{0}$-dimensional subspaces of $ \rm{Lie}U$ by
$$ x \mapsto \rm{Lie} \stab_U (p(x)) .$$
If $x \in X^0_{\min}$ we can choose a $\GG_m$-invariant complementary subgroup $U' \leqslant U$ such that the intersection $\stab_U(p(x)) \cap U'$ is trivial and $U' \stab_U(p(x)) = U$. The condition of complementarity to $U'$ defines an affine open subscheme of $\rm{Grass}(d_0^U, \rm{Lie}U)$ and thus an open affine in $X^0_{\min}$  where quotienting by U is equivalent to quotienting by $U'$, and where $\stab_{U'}(x) = \stab_U(x) \cap U' = \{ e\}$.
\end{rmk}

The proof of
Theorem \ref{thm:ExCo1}, which will be carried out in Section \ref{thmproofsection},  proceeds as follows.
We first establish $X^0_{\min}
\subseteq X^{\rms,U}$ (Theorem
\ref{thm:ExCo1}  (\ref{itm:ExCo1-1}). This is done inductively, using the philosophy of taking
quotients in stages. More precisely, by diagonalising the action of
$\GG_m$-action on $\Lie U$ and using the  
exponential map we may choose a subnormal series   
\[
1=U_0 \trianglelefteq U_1 \trianglelefteq \dots \trianglelefteq U_m=U
\] 
which is preserved by each automorphism in the family $\lambda:\GG_m \to
\Aut(U)$ and such that each successive quotient $U_{j+1}/U_j \cong \GG_a$, with
$\lambda$ acting on $\Lie(U_{j+1}/U_j)$ with positive weight. We will
inductively show that each $X_{\sigma}$ (with $\sigma \in H^0(X,L)_{\max}$) has
a (locally) trivial $U_j$-quotient that is affine, using a combination of
\eqref{eq:sss} and Lemma \ref{lem:Co2Qu0}. 

This results in a locally trivial $U$-quotient $q_U:X^0_{\min} \to
X^0_{\min}/U$. We then use a sufficiently divisible power of the line bundle $L$ to
embed $X^0_{\min}/U$ into a projective space $\PP$, in a $\GG_m$-equivariant
manner. By twisting the linearisation on $LX$ by an appropriate rational
character $\chi$ of $\hat{U}$, we obtain a
$\GG_m$-linearisation $\mathcal{L}'$ over the closure $\ol{X^0_{\min}/U}$ of
$X^0_{\min}/U$ in $\PP$, which pulls back to a positive tensor power of the
twisted rational linearisation $L_{\chi/c} \to X$ and has the properties that
\[
\ol{X^0_{\min}/U}^{\rms,\mathcal{L}'} = \ol{X^0_{\min}/U}^{\rmss,\mathcal{L}'}
\]
and 
\[
q_U^{-1}(\ol{X^0_{\min}/U}^{\rms,\GG_m,\mathcal{L}'}) = X^0_{\min} \setminus (U
Z_{\min}) \subseteq X^{\rms,\hU,L_{\chi/c}}.
\]
These equalities and inclusions are proved using reductive GIT, especially the
Hilbert--Mumford criteria. From here it is then straightforward to show that
$X^0_{\min} \setminus (U Z_{\min})$ has a projective geometric $\hat{U}$-quotient under the
enveloping quotient map $q:X^{\ssfg(\hU)} \to X \env
\hat{U}$, isomorphic to
$\ol{X^0_{\min}/U} \dblslash_\mathcal{L}' \GG_m$, and that $ \bigcap_{u \in U} uX^{ss,T}_{\min+}$ 
has a projective geometric $\hat{U}_T$-quotient under the
enveloping quotient map $q:X^{\ssfg,\hU_T} \to X \env
\hat{U}_T$. The rest of Theorem \ref{thm:ExCo1}
then follows by standard results of non-reductive GIT.


\subsection{Proof of Theorem \ref{thm:ExCo1}} \label{sectionb}   \label{thmproofsection} 
\label{subsec:ExCoProof}

We now begin the proof of Theorem  \ref{thm:ExCo1} as outlined above. 

Suppose we are given a  linearisation $\hat{U} \act L \to X$  satisfying \eqref{eq:sss}. We first set about showing that
$X^0_{\min} \subseteq X^{\rms,U}$ (Theorem \ref{thm:ExCo1}
(\ref{itm:ExCo1-1})). The proof will rely on using the
following lemma in an inductive argument. 

\begin{lem} \label{lem:ExCo2} \cite[Lemma 4.7.5]{dix96} Suppose $X$ is an affine
  scheme with an action of $\GG_a$ and let $\xi \in \Lie(\GG_a)$. If there is
  $f \in \OO(X)$ such that $\xi(f)=1 \in \OO(X)$, then $X$ is a trivial
  $\GG_a$-bundle.  
\end{lem} 

Recall that $X^0_{\min}$ is the union of basic affine opens $X_{\sigma}$
with $\sigma \in H^0(X,L)_{\max}$. Given nonzero $\sigma 
\in H^0(X,L)_{\max}$, we can embed $X$ into a projective space $\PP^n \cong
\PP(H^0(X,L)^*)$ via $L$ using a 
basis of $n+1$ linear sections which are weight vectors for the
$\GG_m$-action, and which includes $\sigma$. Then $X_{\sigma}$ is contained in
an affine coordinate patch $\A^n=(\PP^n)_{\sigma}$ such that the action of
$\GG_m$ on $\A^n$ is diagonal, with all weights $\geq 0$; or equivalently,
$\GG_m$ acts on $\OO(X_{\sigma})$ with all weights $\leq 0$. Note also that each
point $x \in X_{\sigma}$ has a limit point in
  $X_{\sigma} \cap Z_{\min}$ under the 
  action of $t \in \GG_m$ as $t \to 0$. 

By considering the action of $\hat{U_1}=U_1 \rtimes \GG_m$ on
$X_{\sigma}$, one is therefore naturally led to the following lemma. 
\begin{lem} \label{lem:ExCo3}
Let $X$ be an affine scheme with action of $\hat{\GG_a}=\GG_a \rtimes
\GG_m$, where $\GG_m$ acts on $\Lie (\GG_a)$ with strictly positive weight, and let $\xi$ be
a generator of $\Lie(\GG_a)$. Suppose
$\GG_m$ acts on $\OO(X)$ with  weights less than or equal to 0. Then every point in
$X$ has a limit in $X$ under the action of $t \in \GG_m$ as $t \to 0$; let $Z$ be
the set of such limit points in $X$. If $\stab_{\GG_a}(z)=\{e\}$ for each $z
\in Z$, then there is $f \in \OO(X)$ such that $\xi(f)=1 \in \OO(X)$.      
\end{lem}

\begin{proof}
We first show that every point in $X$ has a limit under the action of $t \in
\GG_m$, as $t \to 0$. Fix $x \in X$. To say that $\lim_{t \to 0}
t \cdot x$ exists in $X$ means that the morphism $\phi_x:\GG_m \to X$,
$\phi_x(t)=t \cdot x$, extends to a morphism $\phi_x:\kk \to X$ under the usual
open inclusion $\GG_m \subseteq \kk$ (and then $\lim_{t \to 0} t \cdot x
:=\phi_x(0)$). This is equivalent to saying that the pullback homomorphism
$(\phi_x)^{\#}:\OO(X) \to \OO(\GG_m)=\kk[t,t^{-1}]$ factors through the
localisation map $\kk[t] \to \kk[t,t^{-1}]$. But if $a \in \OO(X)$ is a
$\GG_m$-weight vector of weight $m\leq
0$ then
\[
(\phi_x)^{\#}(a)(t)=a(t \cdot x) = (t^{-1} \cdot a)(x) = t^{-m}a(x)
\]
with $-m \geq 0$. Since $\OO(X)$ is generated by such weight vectors, we see
that $(\phi_x)^{\#}:\OO(X) \to \kk[t,t^{-1}]$ indeed factors through $\kk[t]
\to \kk[t,t^{-1}]$. Let $Z$ be the set of limit points in $X$. 
    
Using the $\GG_m$-action, write $\OO(X)=\bigoplus_{m \leq 0} W_m$ as a
negatively graded algebra, where 
$W_m \subseteq \OO(X)$ is the subspace of weight vectors in $\OO(X)$ of
weight $m \leq 0$. Suppose $\GG_m$ acts on $\xi$ with weight
$\ell >0$. Then we have
\[
\xi(W_m) 
\begin{cases} 
=0 & \text{if $m > -\ell$}, \\
\subseteq W_0 & \text{if $m=-\ell$}, \\
\subseteq \bigoplus_{m<0} W_m & \text{if $m< -\ell$}.
\end{cases}
\]
Let $\tilde{W}=\xi(W_{-\ell})$ be the image of the weight space $W_{-\ell}$
under $\xi$, and consider the vector subspace 
\[
I:=\tilde{W} \oplus \bigoplus_{m<0} W_m.
\]
We claim that $I$ is a $\hat{\GG_a}$-stable ideal of $\OO(X)$. Indeed, $I$ is
$\GG_m$-stable and closed under the action of $\xi$ by construction, so we see
immediately that it is stable under the $\hat{\GG_a}$-action. Let $f
\in I$ and $a \in \OO(X)$. We need to 
show that $af \in I$, for which we may assume that $a \in W_p$ for some $p \leq
0$, without loss of generality. Now if $p<0$, then because multiplication
respects the grading we have $af \in \bigoplus_{m<0} W_m \subseteq
I$. So suppose $p=0$. Write $f=\tilde{f} + g$ where $\tilde{f} \in \tilde{W}$
and $g \in \bigoplus_{m<0} W_m$, so $af=a\tilde{f} + ag$. Then $ag
\in \bigoplus_{m<0} W_m \subseteq I$. Furthermore there is $h \in
W_{-\ell}$ such that $\xi(h)=\tilde{f}$, and because $\xi(a)=0$ we therefore
have $\xi(ah)=a\tilde{f}$, with $ah \in W_{-\ell}$, thus $a\tilde{f} \in
\tilde{W} \subseteq I$. Hence $af \in I$, and the claim is established.    

To finish the proof, we will show that $I=\OO(X)$. We may find a non-trivial
$\GG_m$-invariant complementary subspace $W'$ of $W_0$ such that 
$\OO(X)=W' \oplus I$ as vector spaces. It is easy to see that 
\[
Z=\{x \in X \mid f(x)=0 \text{ for all } f \in \textstyle{\bigoplus_{m<0}} W_m\}
\]
and so the subscheme $V(I):=\{ x\in X \mid f(x) = 0 \text{ for all } f\in I\}$
defined by $I$ is contained in $Z$. Suppose now, for a contradiction, that
$I$ is a proper ideal of $\OO(X)$ and $\mf{m}$ is a
maximal ideal of $\OO(X)$ that contains $I$, and so  $\mf{m}$ defines a point
in $V(I)$. Given $a \in \mf{m}$, write
$a=a'+f$ with $a' \in W'$ and $f \in I$. Since $I \subseteq \mf{m}$
we have $a' \in \mf{m}$, and since $a' \in W' \subseteq \OO(X)^{\hat{\GG_a}}$ and 
$I$ is stable under the $\hat{\GG_a}$-action, we have $\hat{\GG_a} \cdot a'
\subseteq \mf{m}$. So $\mf{m}$ is stable under the $\hat{\GG_a}$-action. But then
$\mf{m}$ defines a point of $V(I) \subseteq Z$ that is fixed by $\hat{\GG_a}$,
which is a contradiction. Hence $I=\OO(X)$. In particular, the constant function
$1 \in W_0=\tilde{W}$, so 
there is $f \in \OO(X)$ such that $\xi(f)=1$.     
\end{proof}

We are now in a position to prove
\begin{propn} \label{prop:ExCo4}
{\rm  (Theorem \ref{thm:ExCo1} (\ref{itm:ExCo1-1})) } Assume that condition \eqref{eq:sss} is satisfied. 
For each nonzero $\sigma \in
H^0(X,L)_{\max}$, the natural map $X_{\sigma} \to 
\spec(\OO(X_{\sigma})^U)$ is a trivial $U$-quotient for the $U$-action on
$X_{\sigma}$. Thus, $X^0_{\min} \subseteq X^{\rms,U}$ and the
restriction of the enveloping quotient map for $U \act L \to X$ restricts to
define a locally trivial $U$-quotient of $X^0_{\min}$.  
\end{propn}

\begin{proof}
As discussed above, we may choose a subnormal series   
\[
1=U_0 \trianglelefteq U_1 \trianglelefteq \dots \trianglelefteq U_m=U
\] 
which is preserved by each automorphism in the one-parameter subgroup $\GG_m$ of
$\Aut(U)$ and such that each successive quotient $U_{j+1}/U_j \cong \GG_a$, with
$\GG_m$ acting on $\Lie(U_{j+1}/U_j)$ with strictly positive weight. We will prove that
$X_{\sigma} \to \spec(\OO(X_{\sigma})^{U_j})$ is a trivial $U_j$-quotient by 
induction on $j$, for $1 \leq j \leq m$.  

For the base case, let $\xi_1 \in \Lie(U_1)$ be non-zero. As observed before
Lemma \ref{lem:ExCo3}, the affine subset $X_{\sigma}$ satisfies the conditions
needed to apply Lemma \ref{lem:ExCo3} with respect to the semi-direct product
$\hat{U_1}=U_1 \rtimes \GG_m$, so there is $f \in \OO(X)$ such that
$\xi_1(f)=1$. It follows from Lemma \ref{lem:ExCo2} that the natural map
$X_{\sigma} \to \spec(\OO(X_{\sigma})^{U_1})$ is a trivial $U_1$-quotient.  

For the induction step, suppose the canonical map $q_j:X_{\sigma} \to
\spec(\OO(X_{\sigma})^{U_j})$ is 
a trivial $U_j$-quotient, for $1 \leq j \leq m-1$. The action of $U_{j+1}
\rtimes \GG_m$ on $X_{\sigma}$ descends to an action of the induced semi-direct product $(U_{j+1}/U_j)
\rtimes \GG_m$ on $\spec(\OO(X_{\sigma})^{U_j})$. Fixing a $\GG_m$-weight vector
$\xi_{j+1} \in \Lie(U_{j+1}) \setminus \Lie(U_j)$, we obtain a generator of
$\Lie(U_{j+1}/U_j)=\Lie(U_{j+1})/\Lie(U_j)$ which acts on
$\OO(X_{\sigma})^{U_j}$ by restricting the action of $\xi_{j+1}$ on
$\OO(X_{\sigma})$ to the subring $\OO(X_{\sigma})^{U_j}$. It is immediate that all
weights for the natural $\GG_m$-action on $\OO(X_{\sigma})^{U_j}$ are
non-positive so, by Lemma \ref{lem:ExCo3}, given a point $y \in
\spec(\OO(X_{\sigma})^{U_j})$ the limit of $y$ under the natural action of
$t \in \GG_m$ as $t \to 0$ exists. If $y=q_j(x)$ for $x \in X_{\sigma}$, then
because $q_j$ is $\GG_m$-equivariant we have
\[
\lim_{t \to 0} t \cdot y = \lim_{t \to 0} q_j(t \cdot x) = q_j\left( \lim_{t \to 0} t
\cdot x \right),
\]
and thus all points in $\spec(\OO(X_{\sigma})^{U_j})$ have limits in
$q_j(X_{\sigma} \cap Z_{\min})$. Let $z \in X_{\sigma} \cap Z_{\min}$ and
suppose $u \in U_{j+1}$ is such that $(uU_j) \in
\stab_{U_{j+1}/U_j}(q_j(z))$. Then there 
is $\tilde{u} \in U_j$ such that $u^{-1}\tilde{u}z=z$. Since $\stab_U(z)$ is
trivial, we conclude that $u=\tilde{u} \in U_j$, so $uU_j=eU_j$. Hence we may
apply Lemma \ref{lem:ExCo3} to the 
action of $(U_{j+1}/U_j) \rtimes \GG_m$ on $\spec(\OO(X_{\sigma})^{U_j})$ to
conclude that there is $f \in \OO(X_{\sigma})^{U_j}$ such that
$\xi_{j+1}(f)=1$. By Lemma \ref{lem:ExCo2}, the natural map
$\spec(\OO(X_{\sigma})^{U_j}) \to \spec(\OO(X_{\sigma})^{U_{j+1}})$ is a trivial 
$U_{j+1}/U_j$-bundle. Since the projection $U_{j+1} \to U_j$ splits, the compositon
\[
X_{\sigma} \to \spec(\OO(X_{\sigma})^{U_j}) \to \spec(\OO(X_{\sigma})^{U_{j+1}})
\]
is a principal $U_{j+1}$-bundle by Lemma \ref{lem:Co2Qu0}, which is in fact
trivial by Proposition \ref{prop:TrRe3.1}. This establishes the induction step. 

Therefore $X_{\sigma} \to \spec(\OO(X_{\sigma})^U)$ is a trivial
$U$-quotient. The rest of the statement of the proposition follows immediately
from the definition of the stable locus for $U \act L \to X$. 
\end{proof}

Note that this argument gives us the following more general result, which does not require condition \eqref{eq:sss}.

\begin{prop} \label{prop:fff}
Let $X$ be an irreducible projective scheme acted upon by a unipotent group $U$
and let 
$L \to X$ be a very ample linearisation. Suppose the
linearisation extends to 
a linearisation of a semi-direct product $\hat{U}_T=U \rtimes T$ of $U$ where $T$ is a torus containing a one-parameter subbgroup $\GG_m$  whose adjoint action on the Lie algebra of $U$ has only strictly positive weights. Then
 we have 
$$\{ x \in X^0_{\min} | \stab_U(p(x)) = \{ e \} \} \subseteq X^{\rms,U},$$
  with the restriction of the enveloping quotient map
  defining a locally trivial $U$-quotient of open subschemes $\{ x \in X^0_{\min} | \stab_U(p(x)) = \{ e \} \}  \to
 \{ x \in X^0_{\min} | \stab_U(p(x)) = \{ e \} \} /U$.   
\end{prop}

\begin{rmk} \label{rem20191}
Suppose that in the context of Lemma \ref{lem:ExCo3} an action of $\GG_a$ on an affine scheme $X$ extends to an action of $\GG_a$ on an affine scheme $Y$ with $X$ as a closed $\GG_a$-invariant subscheme. If there is $f \in \mathcal{O}(Y)$ with $\xi(f)=1 \in \mathcal{O}(Y)$ then $\xi(f |_X)=1 \in \mathcal{O}(X)$. So to obtain the conclusion of Lemma \ref{lem:ExCo3} it is not necessary to have an action of $\hat{\GG_a} = \GG_a \rtimes \GG_m$ on $X$. It suffices to have an action of $\hat{\GG_a}$ 
on $Y$ satisfying the hypotheses of Lemma \ref{lem:ExCo3}. In particular suppose that $X$ is a closed subscheme of affine space $\AL^n$, and suppose that $\GG_a$ acts linearly on $\AL^n$ via a representation $\rho: \hat{\GG_a} \to GL(n)$. If $\GG_a$ fixes a weight space for the $\GG_m$-action on $\AL^n$ pointwise, then we can modify the representation of $\hat{\GG_a}$ by making an arbitrary chage to the weight with which $\GG_m$ acts on this weight space, without affecting the representation of $\GG_a$. This may mean that $X$ is no longer $\hat{\GG_a}$-invariant, but nevertheless if the hypotheses of Lemma \ref{lem:ExCo3} are satisfied for the new $\hat{\GG_a}$-action on $\AL^n$, then the conclusion of the lemma is still valid for $X$.

We can apply this argument to generalise Proposition \ref{prop:fff} to the situation where the following property holds:

\noindent (*) If $z\in X^{\l(\GG_m)}$ where $\l(\GG_m)$ acts on $L^*|_z$ with weight strictly less than $\weight_{r}$ then 
$$\stab_{U}(z) = U.$$
In the inductive proof of the Proposition using Lemma \ref{lem:ExCo3}, we replace the action of $\l(\GG_m)$ on the ambient projective space defined by the linearisation with the action for which $\l(\GG_m)$ acts on the weight spaces with weights $\weight_j$ for $j \geqslant r$ as before but $\l(\GG_m)$ now acts with weight $\weight_r$ on the weight spaces which did have weights $\weight_j$ for $j < r$. This gives us a new version $\tilde{p}$ for the ambient projective space of the map $p:X^0_{\min} \to Z_{\min}$  defined by $p(x) = \lim_{t \to 0} \l(t)x$, and we have 
$$\{ x \in X^0_{\min} | \stab_U(\tilde{p}(x)) = \{ e \} \} \subseteq X^{\rms,U},$$
  with the restriction of the enveloping quotient map
  defining a locally trivial $U$-quotient of open subschemes $\{ x \in X^0_{\min} | \stab_U(\tilde{p}(x)) = \{ e \} \}  \to
 \{ x \in X^0_{\min} | \stab_U(\tilde{p}(x)) = \{ e \} \} /U$.   
In particular we can apply this when in addition to (*) we have

\noindent(**) 
For all $z\in X^{\l(\GG_m)}$ where $\l(\GG_m)$ acts on $L^*|_z$ with weight equal to $\weight_{r}$ then 
$$ \stab_{U}(z) = \{e\}.$$
In this case we can identify the open subset $\{ x \in X^0_{\min} | \stab_U(\tilde{p}(x)) = \{ e \} \}$ of $X^0_{\min}$ with the locus corresponding to nonzero component for the weight space $\weight_r$ of the original action. When $\xi \in \mathrm{Lie}U$ is a nonzero weight vector for the adjoint action of $\l:\GG_m \to \hU$ on $\mathrm{Lie}U$  with weight $s$, the additive one-parameter subgroup $\tilde{\l}:\GG_a \to \hU$ of $\hU$ generated by $\xi$ and the multiplicative one-parameter subgroup $\l:\GG_m \to \hU$ together generate a subgroup $\GG_a \rtimes \GG_m$ of $\hU$. 
We can choose coordinates on the ambient projective space such that $\l(\GG_m)$ acts diagonally and $\xi$ acts in Jordan canonical form 
(cf. \cite{BDHK} $\S$5 and \cite{BDHK2} $\S$2). Calculating in these coordinates we find that  if (*) holds then for any $x \in X^0_{\min}$ the $\ell$-jet at 0 of the morphism $\alpha_x: \AL^1 \to X$ defined by
$$\alpha_x (t) = \l(t) \exp(\xi/t^s) x \,\, \mbox{ for $t \neq 0$ and } \,\, \alpha_x (0) = \lim_{t\to 0} \l(t) \exp(\xi/t^s) x $$ 
vanishes if $0<\ell < \weight_r - \weight_{\min}$, and the locus in $X^0_{\min}$ where $\stab_U(\tilde{p}(x))$ is trivial coincides with the locus where the $\weight_r - \weight_{\min}$-jet at 0 of $\alpha_x$ is nonzero.
\end{rmk}

Having established (\ref{itm:ExCo1-1}) of Theorem \ref{thm:ExCo1}, we now turn to
proving statements (\ref{itm:ExCo1-2}--\ref{itm:ExCo1-4}) of the same theorem. So
assume from now on that $X^0_{\min} \neq U Z_{\min}$. Also let
$q_U:X^{\ssfg(U,L)} \to X \env U$ be the enveloping quotient map for the 
linearisation $U \act L \to X$. As noted above, we have $X^0_{\min} \subseteq
X^{\rms,U}$, so the enveloping quotient map restricts to a geometric quotient
\[
q_U:X^0_{\min} \to X^0_{\min} /U \subseteq X \env U,
\]
which can locally be described as $X_{\sigma} \to \spec(\OO(X_{\sigma})^U)$, for
$\sigma \in H^0(X,L)_{\max}$. Let us fix a basis $\sigma_1,\dots,\sigma_\ell$ of
$H^0(X,L)_{\max}$. Each of the algebras $\OO(X_{\sigma_i})^U$ is
finitely generated over $\kk$, so we may find $s>0$ such that 
\[
W:=H^0(X,L^{\ten s})^U
\]
defines an enveloping system adapted to the subset
$\mc{S}=\{\sigma_1^s,\dots, \sigma_\ell^s\}$ (see the proof of \cite{BDHK} Proposition 3.1.18,
1);  this means that each of the
$\kk$-algebras $(S^U)_{(\sigma_i^s)}=(S^U)_{(\sigma_i)}$ (where $S = \bigoplus_{k \geq 0} H^0(X,L^{\ten k})$) has generating set
given by $\{f/(\sigma_i^s) \mid f \in H^0(X,L^{\ten s})^U\}$. For each
$i=1,\dots,\ell$, let 
$\Sigma_i$ denote the section in $H^0(\PP(W^*),\OO(1))$ corresponding to
$\sigma_i^s$ under the identification $H^0(\PP(W^*),\OO(1)) = H^0(X,L^{\ten
  s})^U$. The 
inclusion $W \hookrightarrow H^0(X,L^{\ten s})$ defines a
morphism 
\[
\phi:X^0_{\min} \to \PP(W^*)
\] 
which descends to a locally closed immersion 
\[
\ol{\phi}:X^0_{\min}/U
\hookrightarrow \PP(W^*)
\]
such that each of the restrictions 
$\ol{\phi}:\spec(\OO(X_{\sigma_i})^U) \hookrightarrow \PP(W^*)_{\Sigma_i}$ is a
closed immersion (this is because the pullback maps $\OO(\PP(W^*)_{\Sigma_i}) \to
\OO(X_{\sigma_i})^U$ are surjective, by definition of $W$). Furthermore, the
canonical $\GG_m$-linearisation on 
$\OO_{\PP(W^*)}(1) \to \PP(W^*)$ is compatible with the restricted linearisation
$\GG_m \act L^{\ten s} \to X^0_{\min}$ under $\phi$, and the embedding
$\ol{\phi}:X^0_{\min}/U \hookrightarrow \PP(W^*)$ is equivariant with respect
to the canonically induced $\GG_m$-action on $X^0_{\min}/U$. Let
$\ol{X^0_{\min}/U}$ be the closure of $X^0_{\min}/U$ 
inside $\PP(W^*)$ via $\ol{\phi}$ and, by abuse of notation, let
$\ol{\phi}:X^0_{\min}/U \hookrightarrow \ol{X^0_{\min}/U}$ be the induced
open immersion.  

Let $(W^*)_{\min}$ be the weight space of minimal possible
weight for the natural $\GG_m$-action on $W^*$, let $\PP((W^*)_{\min})$ be the
associated linear subspace of $\PP(W^*)$ and let $\PP(W^*)_{\min}^0$ be the
open subset of points in $\PP(W^*)$ 
that flow to $\PP((W^*)_{\min})$ under the action of $t \in \GG_m$, as $t \to
0$. 


\begin{lem} \label{lem:ExCo5} Assume that condition \eqref{eq:sss} is satisfied.
Then the locally closed immersion $\ol{\phi}:X^0_{\min}/U \hookrightarrow
\PP(W^*)$ has image contained in $\PP(W^*)_{\min}^0$, and the induced embedding
$X^0_{\min}/U \hookrightarrow \PP(W^*)_{\min}^0$ is a closed immersion.   
\end{lem}

\begin{proof}
We introduce some notation. Given a tuple
$K=(k_1,\dots,k_n) \in \mb{N}^n$ 
of non-negative integers such that $k_1+\dots +k_n=s$, let
$\sigma^K:=\sigma_1^{k_1} \cdots \sigma_n^{k_n}$ and let $\Sigma_K$ be the
section in $H^0(\PP(W^*),\OO(1))$ that corresponds to $\sigma^K$ under the
identification $H^0(\PP(W^*),\OO(1)) = H^0(X,L^{\ten s})^U$. Observe that the
maximal weight space $H^0(\PP(W^*),\OO(1))_{\max}$ for the $\GG_m$-action on
$H^0(\PP(W^*),\OO(1))$ is spanned by the $\Sigma_K$ with
$K=(k_1,\dots,k_n)$ running over all tuples in $\mb{N}^n$ such that $k_1+ \dots
+ k_n=s$, thus
$\PP(W^*)_{\min}^0$ is covered by the associated affine open subsets
$\PP(W^*)_{\Sigma_K}$. Because $q_U:X^0_{\min} \to X^0_{\min} /U$ is
surjective, we also have   
\[
(\ol{\phi})^{-1}(\PP(W^*)_{\Sigma_K}) = q_U(\phi^{-1}(\PP(W^*)_{\Sigma_K})) =
q_U(X_{\sigma^K}) = \spec(\OO(X_{\sigma^K})^U).
\]
In particular, choosing $K$ with $i$-th entry equal to $s$ and zero in each other
entry (so that $\Sigma_K=\Sigma_i$), we see that
$(\ol{\phi})^{-1}(\PP(W^*)_{\Sigma_K}) = \spec(\OO(X_{\sigma_i})^U$, which cover
$X^0_{\min}/U $ as $i$ runs from $1$ to $n$. Hence, the image of
$X^0_{\min}/U$ under $\ol{\phi}$ is contained in $\PP(W^*)_{\min}^0$. 

For each tuple $K = (k_1,\dots,k_n)$ (with $k_1+ \dots + k_n=s$), we claim that
the restriction 
\[
\ol{\phi}:\spec(\OO(X_{\sigma^K})^U) \hookrightarrow \PP(W^*)_{\Sigma_K}
\]
is a closed immersion of affine schemes. Note that showing this is enough to
prove the lemma, because closed immersions are local
on the base and the $\PP(W^*)_{\Sigma_K}$ cover $\PP(W^*)_{\min}^0$. To prove
the claim, it is equivalent to show that each pullback
$(\ol{\phi})^{\#}:\OO(\PP(W^*)_{\Sigma_K}) \to \OO(X_{\sigma^K})^U$ is
surjective. Under the usual identifications $(\symdot W)_{(\Sigma_K)} =
\OO(\PP(W^*)_{\Sigma_K})$ and $(S^U)_{(\sigma^K)}=\OO(X_{\sigma^K})^U$, this
amounts to showing that the homomorphism  
\[
\psi_K:(\symdot W)_{(\Sigma_K)} \to (S^U)_{(\sigma^K)}, 
\]
induced by $\frac{f}{\Sigma_K} \mapsto \frac{f}{\sigma^K}$ for $f \in
W=H^0(X,L^{\ten s})^U$, is 
surjective. 

To this end, after relabelling if necessary we may assume, without loss of
generality, that 
$K=(k_1,\dots,k_p,0,\dots,0)$, with $1 \leq p \leq n$ and $k_i > 0$ for
$i=1,\dots,p$. Then as subalgebras of the field of rational
functions $\kk(X) = S_{((0))}$, we have  
\[
(S^U)_{(\sigma^K)}=(S^U)_{(\sigma_1 \cdots
  \sigma_p)}=(S^U)_{(\sigma_1)}\left[\tfrac{\sigma_1}{\sigma_2},\dots,
  \tfrac{\sigma_1}{\sigma_p}\right] 
\]
where the last ring is the subalgebra generated by $(S^U)_{(\sigma_1)}$ and the
rational functions $\frac{\sigma_1}{\sigma_2},\dots,\frac{\sigma_1}{\sigma_p}$
(which are regular on $X_{\sigma^K}$). Observe that
\[
\frac{\sigma_1^{k_1-1} \sigma_2^{k_2} \cdots \sigma_p^{k_p}\sigma_i}{\sigma^K} =
\frac{\sigma_i}{\sigma_1} \in (S^U)_{(\sigma_1)}, \quad i=2,\dots,p, 
\]
and (where $\widehat{\sigma_i^{k_i}}$ means `omit $\sigma_i^{k_i}$' 
in what follows)
\[
\frac{\sigma_1^{k_1} \cdots \widehat{\sigma_i^{k_i}} \cdots
  \sigma_p^{k_p}\sigma_i^{k_i-1} \sigma_1}{\sigma^K} =
\frac{\sigma_1}{\sigma_i} \quad i=2,\dots,p. 
\]
Also, for each $f \in H^0(X,L^{\ten s})^U$ we have
\[
\frac{f}{\sigma^K} \cdot
\left(\frac{\sigma_2}{\sigma_1}\right)^{k_2} \cdots
\left(\frac{\sigma_p}{\sigma_1}\right)^{k_p}=\frac{f}{\sigma_1^s} \in
(S^U)_{(\sigma_1)}.  
\]
and by the choice of $s>0$ the algebra $(S^U)_{(\sigma_1)}$ is generated
by the rational functions $\frac{f}{\sigma_1^s}$ for $f \in H^0(X,L^{\ten s})^U$. We thus see
that the image of $\psi_K$ contains $(S^U)_{(\sigma_1)}$, along with the extra
generators $\frac{\sigma_1}{\sigma_2},\dots,\frac{\sigma_1}{\sigma_p}$, and so
conclude that $\psi_K$ is surjective, as claimed. This completes the proof.
\end{proof}   

 Since $U$ is a normal subgroup of $\hH$ and $R$ centralises the one-parameter subgroup $\GG_m$, this entire set-up is acted on by $H/U = R$. Recall the weaker version   \eqref{eq:sssH} of  \eqref{eq:sss}:
$$
\text{$\stab_U(z) = \{e\}$ for every $z \in Z_{\min}^{ss,R}$}  
$$
where $Z_{\min}^{ss,R}$ is the semistable locus for the induced linear action of $R=H/U$ on $Z_{\min}$.
The proof of Lemma \ref{lem:ExCo5}, using Proposition \ref{prop:fff} instead of Proposition \ref{prop:ExCo4}, gives us 

\begin{lem} \label{blimey} 
Assume that $\hH$ acts linearly on $X$ and that condition  \eqref{eq:sssH} is satisfied. Then the open subscheme
$ \{ x \in X^0_{\min}\,\, | \,\, p(x) \in Z_{\min}^{ss,R} \} $ of $X$ has a geometric quotient by the action of $U$ and there is a closed immersion
$$ \{ x \in X^0_{\min}\,\,  | \,\,  p(x) \in Z_{\min}^{ss,R} \}/U \to \{ y \in \PP(W^*)^0_{\min}\,\,\,  |\,\,\,  p_W(y) \in Z(\PP(W^*))_{\min}^{ss,R} \}/U $$
where $p_W$ and  $Z(\PP(W^*))_{\min}^{ss,R}$ are the analogues of $p$ and $ Z_{\min}^{ss,R}$ for the action of $\hH/U \cong R \times \GG_m$ on $\PP(W^*)$.
\end{lem}

From Lemma \ref{lem:ExCo5} it follows  that the morphism $\ol{\phi}$ induces a
$T$-equivariant isomorphism of quasi-projective schemes $\ol{\phi}:X^0_{\min}/U \overset{\cong}{\longrightarrow} \ol{X^0_{\min}/U} \cap
\PP(W^*)_{\min}^0$. Let us now use this identification freely for the rest of
the argument, writing  
\[
X^0_{\min}/U = \ol{X^0_{\min}/U} \cap \PP(W^*)_{\min}^0. 
\]

Any point in $X^0_{\min}/U$ has a limit under the action of
$t \in \GG_m$, as $t \to 0$, contained in the closed subset
$\ol{X^0_{\min}/U} \cap \PP((W^*)_{\min}) \subseteq \ol{X^0_{\min}/U} \cap
\PP(W^*)_{\min}^0$ under this isomorphism. In particular, given $x \in
X^0_{\min}$ the point $q_U(x) $ lies in $ \ol{X^0_{\min}/U} \cap \PP((W^*)_{\min})$ if,
and only if, 
\[
q_U(x)=\lim_{t \to 0} t \cdot q_U(x) = q_U \left( \lim_{t \to 0} t \cdot x \right).
\]
Thus, for each $x \in X^0_{\min}$, we have  
\begin{equation} \label{eq:ExCo1}
q_U(x) \in \ol{X^0_{\min}/U} \cap \PP((W^*)_{\min})  \iff x \in
U Z_{\min}.  
\end{equation}

By assumption we have $X^0_{\min} \neq U Z_{\min}$, so as a
consequence of \eqref{eq:ExCo1} we may conclude that $\PP((W^*)_{\min}) \neq
\PP(W^*)_{\min}^0$ and hence that the $\GG_m$-action on $W=H^0(X,L^{\ten s})^U$ has at
least two distinct weights. Note that the maximum weight for the $\GG_m$-action
on $H^0(X,L^{\ten s})^U$ is equal to $-s\weight_{\min}$. Let $\epsilon >0$ be a
rational number such 
that $-s(\weight_{\min}+\epsilon)$ lies strictly between $-s\weight_{\min}$ and the next largest
weight for the $\GG_m$-action on $W$ (which must be at most $ -s\weight_{\min}-1$). Let
$\chi$ be the rational character of $\hat{U}$ of weight 
$\weight_{\min} + \epsilon$ and consider the rational $\GG_m$-linearisation
$\OO_{\PP(W^*)}(1)_{s\chi} \to \PP(W^*)$. The rational weights of
$H^0(\PP(W^*),\OO_{\PP(W^*)}(1)_{s\chi})^*$ are arranged such that the minimal
weight is less 
than $0$ and the next smallest weight is greater than $0$, so it follows
immediately 
from the Hilbert--Mumford criteria that the stable
locus for $\GG_m \act \OO_{\PP(W^*)}(1)_{s\chi} \to \PP(W^*)$ is equal to the
semistable locus, which is equal to $\PP(W^*)_{\min}^0 \setminus
\PP((W^*)_{\min})$, and that the semistable locus of the induced action of $T$ is contained in that for $\GG_m$. Let
\[
M=\OO_{\PP(W^*)}(1)^{s\chi}|_{\ol{X^0_{\min}/U}} \to
\ol{X^0_{\min}/U}
\]
be the restriction of the rational $\GG_m$-linearisation
$\OO_{\PP(W^*)}(1)^{(s\chi)} \to \PP(W^*)$ to $\ol{X^0_{\min}/U}$. Then by
restriction of (semi)stable loci we have  
\begin{align*}
\ol{X^0_{\min}/U}^{\rms(M)} =& \ol{X^0_{\min}/U} \cap
\PP(W^*)^{\rms(\OO_{\PP(W^*)}(1)_{s\chi)}} \\
=& \ol{X^0_{\min}/U} \cap
\PP(W^*)^{\rmss(\OO_{\PP(W^*)}(1)_{s\chi)}} \\
=& \ol{X^0_{\min}/U}^{\rmss(M)},  
\end{align*} 
which are furthermore equal to 
\[
(\ol{X^0_{\min}/U} \cap \PP(W^*)_{\min}^0) \setminus
(\ol{X^0_{\min}/U} \cap \PP((W^*)_{\min}) = (X^0_{\min}/U) \setminus
((X^0_{\min}/U) \cap \PP((W^*)_{\min})).    
\]
By \eqref{eq:ExCo1} we therefore have
\[
q_U^{-1} \left( \ol{X^0_{\min}/U}^{\rms(M)} \right) = X^0_{\min}
\setminus (U Z_{\min}).
\]
Let $L^{(\chi)} \to X$ be the rational $\hat{U}$-linearisation obtained by
twisting the linearisation $L \to X$ by the rational character $\chi$. Note 
that the rational $\GG_m$-linearisations 
$\OO_{\PP(W^*)}(1)_{s\chi} \to \PP(W^*)$ and $L^{(s\chi)} \to X^0_{\min}$ 
are compatible via
$\phi:X^0_{\min} \to \PP(W^*)$.

\begin{lem} \label{lem:preceding}
If  the
$\hU_T$-linearisation  is twisted by a well adapted rational character $\chi$ whose restriction to $H$ is trivial, 
then
 $$\phi^{-1}(\PP(W^*)^{\ssfg,T,\OO_{\PP(W^*)}(1)_{s\chi}}) = \bigcap_{u \in U} uX^{ss,T,\chi}$$
and
$$ \phi^{-1}(\PP(W^*)^{s,T,\OO_{\PP(W^*)}(1)_{s\chi}}) = \bigcap_{u \in U} uX^{s,T,\chi}.$$   
\end{lem}

\begin{proof} 
When $T$ is trivial this follows immediately from the preceding arguments. In general we combine these with Remark \ref{rmk:endof2}. By the Hilbert--Mumford criteria for classical GIT we know that (semi)stability for $T$ acting on $\PP(W^*)$ is equivalent to (semi)stability for every choice of one-parameter subgroup $\lambda: \GG_m \to T$, and for this we study orbits under one-parameter subgroups $\tilde{\lambda}: \GG_m \to T$ and their limits $\lim_{t \to 0} \tilde{\lambda}(t)y$. It suffices to understand these limits in an arbitrarily small neighbourhood of the $T$-fixed point set: $y$ is semistable (respectively stable) if and only if 0 lies in the (respectively the interior of) the convex hull of the weights with which $T$ acts on the fibres of the line bundle induced by (a positive tensor power of) $L^*$ over the $T$-fixed point set $\overline{Ty}^T$ in the closure of the orbit of $y$ in $\PP(W^*)$. This convex hull is a convex polyhedron whose faces near any vertex are determined by the weights with which $T$ acts on the Zariski tangent space to $\overline{Ty}$ at a corresponding point of $\overline{Ty}^T$. By the well adapted hypothesis 
 we have seen that $X^{ss,T,\chi} \subseteq X^0_{\min}$, and the result follows using Remark \ref{rmk:endof2}. 
\end{proof}

\begin{propn} \label{prop:ExCo6}
We have $$ \bigcap_{u \in U} uX^{s,\GG_m}_{\min+} = X^0_{\min} \setminus U Z_{\min}  \,\,\, \subseteq \,\,\, 
X^{\rms(\hat{U},L^{(\chi)})}$$ and  $$\bigcap_{u \in U} uX^{ss,T}_{\min+} \subseteq X^{\ssfg(\hat{U}_T,L^{(\chi)})}.$$
\end{propn}

\begin{proof} 
By Lemma \ref{lem:preceding} we have the
equality  
\[
 \bigcap_{u \in U} uX^{s,\GG_m}_{\min+} = X^0_{\min} \setminus UZ_{\min} =
\phi^{-1}(\PP(W^*)^{\rmss(\GG_m, \OO_{\PP(W^*)}(1)^{(s\chi)})}).
\]
 We will show that $$\phi^{-1}(\PP(W^*)^{\rmss(\OO_{\PP(W^*)}(1)^{(s\chi)})}) \subseteq
X^{\rms,\hat{U},L^{(\chi)}}.$$
Suppose $F$ is contained in $H^0(\PP(W^*),\OO(d))^{\GG_m}$,
with $d>0$. Then we may regard $F$ as a 
linear combination of degree $d$ monomials in $\GG_m$-weight vectors in $W$, and
each such monomial must be $\GG_m$-invariant. So, in covering
$\PP(W^*)^{\rmss(\OO_{\PP(W^*)}(1)^{(s\chi)})}$ by open subsets of the form
$\PP(W^*)_F$, we may assume that $F$ is an invariant monomial of weight vectors,
without 
loss of generality. Note also that such a monomial must be divisible by some 
$\Sigma_K \in H^0(\PP(W^*),\OO(1))_{\max}=H^0(X,L^{\ten s})_{\max}$. It follows
that $\phi^*F \in H^0(X,(L^{(\chi)})^{\ten ds})^{\hat{U}}$ is equal to
$g\sigma$, with $\sigma \in H^0(X,L)_{\max}$ and $g \in H^0(X,(L^{(\chi)})^{\ten
  (ds-1)})^{U}$. In particular, the map $q_U:X_{\phi^*F} \to
\spec(\OO(X_{\phi^*F})^U)$ is equal to the canonical morphism
\[
X_{g\sigma} \to \spec((\OO(X_{\sigma})^U)_{a}),
\]
where $(\OO(X_{\sigma})^U)_{a}$ is the localisation of
$\OO(X_{\sigma})^U$ at the function $a=\frac{g}{\sigma^{ds-1}}$. This must be a
locally trivial $U$-quotient, being the restriction of the 
locally trivial quotient $q_U:X_{\sigma} \to \spec(\OO(X_{\sigma})^U)$. Since
$\spec(\OO(X_{\phi^*F})^U)$ maps into
$\PP(W^*)^{\rmss(\OO_{\PP(W^*)}(1)^{(s\chi)})}$ under the embedding $\ol{\phi}$,
by restriction we see that the action of $\GG_m$ on $\spec(\OO(X_{\phi^*F})^U)$
is closed with all stabilisers finite. The same is true for the $\GG_m$-action
on $X_{\phi^*F}$ by Lemma \ref{lem:GiSt1}. Finally, the open set $X_{\phi^*F} $
is affine because $L^{(\chi)} \to X$ is ample as
a line bundle, so we conclude that
$X_{\phi^*F} \subseteq X^{\rms(\hat{U},L^{(\chi)})}$. Thus
$\phi^{-1}(\PP(W^*)^{\rmss(\OO_{\PP(W^*)}(1)^{(s\chi)})})$ is contained in
$X^{\rms(\hat{U},L^{(\chi)})}$, as desired.    A similar argument using $T$-weight vectors shows that $\bigcap_{u \in U} uX^{ss,T}_{\min+} \subseteq X^{\ssfg(\hat{U}_T,L^{(\chi)})}.$
\end{proof}

We are now in a position to complete the proof of Theorem \ref{thm:ExCo1}. Recall that if $\hU_T$ acts on an algebra $A$ then $A^{\hU} = (A^U)^{\GG_m}$ and $A^{\hU_T} = (A^U)^{T}$. The
GIT quotient $$\pi_{\GG_m}:\ol{X^0_{\min}/U}^{\rmss(M)} \to 
\ol{X^0_{\min}/U} \dblslash_M \GG_m$$ is a geometric quotient for the
$\GG_m$-action on 
$\ol{X^0_{\min}/U}^{\rms(M)}=\ol{X^0_{\min}/U}^{\rmss(M)}$, and hence the composition 
\[     \bigcap_{u \in U} uX^{s,\GG_m}_{\min+} =
X^0_{\min} \setminus U Z_{\min} \overset{q_U}{\longrightarrow}
\ol{X^0_{\min}/U}^{\rms(M)} \overset{\pi_{\GG_m}}{\longrightarrow}
\ol{X^0_{\min}/U} \dblslash_M \GG_m
\]
provides a geometric quotient for the $\hat{U}$-action on $ \bigcap_{u \in U} uX^{s,\GG_m}_{\min+}$, with projective quotient
$$( \bigcap_{u \in U} uX^{s,\GG_m}_{\min+})/\hat{U} \cong 
\ol{X^0_{\min}/U} \dblslash_M \GG_m.$$ 
Furthermore, by Proposition
\ref{prop:ExCo6} we know that $ \bigcap_{u \in U} uX^{s,\GG_m}_{\min+}$
is an open $\hat{U}$-stable subset of the stable locus $X^{\rms(L^{(\chi)})}$
for the rational 
linearisation $\hat{U} \act L^{(\chi)} \to X$, so by uniqueness of geometric
quotients we may identify $( \bigcap_{u \in U} uX^{s,\GG_m}_{\min+})/\hat{U}$ with the image of $ \bigcap_{u \in U} uX^{s,\GG_m}_{\min+}$ under the enveloping quotient map
$q:X^{\ssfg(L^{(\chi)})} \to X \env_{\! \! L^{(\chi)}} \hat{U}$. Note that,
because $X$ is irreducible and the enveloping quotient map is dominant,
$( \bigcap_{u \in U} uX^{s,\GG_m}_{\min+})/\hat{U}$ is a dense 
open subscheme of $X \env_{\! \! L^{(\chi)}} \hat{U}$. On the other hand, the
quotient $( \bigcap_{u \in U} uX^{s,\GG_m}_{\min+})/\hat{U}$ is projective,
thus universally closed over $\spec \kk$, and so $( \bigcap_{u \in U} uX^{s,\GG_m}_{\min+})/\hat{U} $ is also a closed subscheme of
$X \env_{\! \! L^{(\chi)}} \hat{U}$, since the latter is separated
over $\spec \kk$ \cite[Tag 01W0]{stacks-project}. Hence $( \bigcap_{u \in U} uX^{s,\GG_m}_{\min+})/\hat{U} = X \env_{\! \! L^{(\chi)}}
\hat{U}$. Because $ \bigcap_{u \in U} uX^{s,\GG_m}_{\min+} \subseteq 
X^{\rms(L^{(\chi)})}$ it follows that $X^{\rms(L^{(\chi)})}/\hat{U}=X \env_{\! \! L^{(\chi)}}
\hat{U}$. Also, a consequence of the definition of the stable locus
$X^{\rms(L^{(\chi)})}$ is that it satisfies 
$q^{-1}(q(X^{\rms(L^{(\chi)})}))=X^{\rms(L^{(\chi)})}$, so we in fact have that 
\[
 \bigcap_{u \in U} uX^{s,\GG_m}_{\min+} = X^{\rms(L^{(\chi)})} =
X^{\ssfg(L^{(\chi)})},      
\]
and that the
enveloping quotient $q:X^{\ssfg(L^{(\chi)})} \to X \env_{\! \! L^{(\chi)}}
\hat{U}$ is a geometric quotient for the $\hat{U}$-action on
$X^{\ssfg(L^{(\chi)})}$. Since $X   \env_{\! \! L^{(\chi)}}
\hat{U}
= \ol{X^0_{\min}/U} \dblslash_M \GG_m$ and $X \env_{\! \! L^{(\chi)}}
\hat{U}_T = (X\env_{\! \! L^{(\chi)}}
\hat{U})/\!/ T$, a
 similar argument using Lemma \ref{lem:preceding} and Remark \ref{rmk:endof2} shows that $X^{\ssfg(\hU_T)} = \bigcap_{u \in U} uX^{ss,T}_{\min+}$ and  $X^{\rms(\hU_T)} = \bigcap_{u \in U} uX^{s,T}_{\min+}$,
which is Theorem \ref{thm:ExCo1}  (\ref{itm:ExCo1-2}). It also shows that the
enveloping quotient map $q:X^{\ssfg(\hU_T, L^{(\chi)})} \to X \env_{\! \! L^{(\chi)}}
\hat{U}_T$ is surjective and is a good quotient for the $\hat{U}_T$-action on
$X^{\ssfg(\hU_T,L^{(\chi)})}$, proving (\ref{itm:ExCo1-4}) of Theorem
\ref{thm:ExCo1}. Finally, by 
Proposition \ref{cor:GiFi5}, for sufficiently divisible $r>0$ the ring of
invariants $\bigoplus_{k \geq 0} H^0(X,(L^{(\chi)})^{\ten kr})^{\hat{U}_T}$ for the
$\hat{U}_T$-linearisation $(L^{(\chi)})^{\ten r} \to X$ is a
finitely generated $\kk$-algebra, which gives (\ref{itm:ExCo1-3}) of Theorem
\ref{thm:ExCo1}. This completes the proof of Theorem \ref{thm:ExCo1}.    

\medskip

The following corollary gives us Theorem \ref{mainthm}.

\begin{cor} \label{cor:ExCo1H}
Assume that the hypotheses of Theorem \ref{mainthm} hold. Then 
 $$X^{\ssfg,\hH} = \bigcap_{h \in H} hX^{ss,T}_{\min+}, \,\,\, \,\,\,\,\,\, X^{\rms,\hH} = \bigcap_{h \in H} hX^{s,T}_{\min+},$$   
and 
 the enveloping
quotient map $q_{\hH}: X^{\ssfg,\hH} \to X \env {\hH}$  
is a good quotient for the $\hH$-action on
$X^{\ssfg,\hH}$, with $q_{\hH}(x) = q_{\hH}(y)$ for $x,y \in X^{\ssfg, \hH} $ if and only if the closures of the $\hH$-orbits of $x$ and $y$ meet in $X^{\ssfg,\hH} $.  
\end{cor} 
\begin{proof}
Let  $T$ be a maximal torus of $R$, 
 and let $\hU = U \rtimes\l( \GG_m)$ and $\hU_T = U \rtimes (T \times \l(\GG_m))$. Then
$X \env \hH = (X \env \hU) /\!/ R$ where $X \env \hU$ is a geometric quotient of $ \bigcap_{u \in U} uX^{s,\GG_m}_{\min+}$ by $\hU$, and $X^{\ssfg,\hH}$ and $X^{\rms,\hH}$ are the pre-images in $ \bigcap_{u \in U} uX^{s,\GG_m}_{\min+}$ of the (semi)stable loci for the action of the reductive group $R$ on $X \env \hU$. By Theorem \ref{thm:ExCo1} the pre-images in $ \bigcap_{u \in U} uX^{s,\GG_m}_{\min+}$ of the (semi)stable loci for the action of  $T$ on $X \env \hU$ are $X^{\ssfg,\hU_T} = \bigcap_{u\in U} uX^{ss,T}_{\min+}$ and $X^{\rms,\hU_T } = \bigcap_{u\in U} uX^{s,T}_{\min+}$. So  by the Hilbert--Mumford criteria for the action of the reductive group $R$ on $X \env \hU$, we have
$$ X^{ss,\hH} = \bigcap_{r \in R} r X^{ss,\hU_T} = \bigcap_{r \in R} \bigcap_{u\in U} ruX^{ss,T}_{\min+} = \bigcap_{h \in H} hX^{ss,T}_{\min+}$$
and similarly $ X^{s,\hH} =  \bigcap_{h \in H} hX^{s,T}_{\min+}$. The result now follows from classical GIT applied to the action of $R$ on the geometric quotient $X \env \hU$.
\end{proof}

\begin{rmk}
It is enough in Corollary \ref{cor:ExCo1H} to assume condition \eqref{eq:sssH} instead of condition \eqref{eq:sss}. To prove this version we use Lemma \ref{blimey} and work with the closure
$$\overline{ \{ x \in X^0_{\min} \,\, | \,\, p(x) \in Z_{\min}^{ss,R} \}/U } $$
in $\PP(W^*)$ of the geometric quotient $ \{ x \in X^0_{\min} \,\, | \,\, p(x) \in Z_{\min}^{ss,R} \}/U$, replacing $X \env \hU$ with the classical GIT quotient 
$$\overline{ \{ x \in X^0_{\min} \,\, | \,\, p(x) \in Z_{\min}^{ss,R} \}/U }/\!/\l(\GG_m)  \subseteq \PP(W^*)/\!/\l(\GG_m).$$
Quotienting by the reductive group $R$, this gives us
$$X \env \hH = \overline{ \{ x \in X^0_{\min} \,\, | \,\, p(x) \in Z_{\min}^{ss,R} \}/U }/\!/R  \subseteq \PP(W^*)/\!/R,$$
and if $\epsilon>0$ is chosen sufficiently small then 
$$\PP(W^*)^{ss,R,\chi} \subseteq \{ y \in \PP(W^*)^0_{\min} \,\, | \,\, p_W(y) \in Z(\PP(W^*))_{\min}^{ss,R} \} $$
so $\overline{ \{ x \in X^0_{\min} \,\, | \,\, p(x) \in Z_{\min}^{ss,R} \}/U }  \cap \PP(W^*)^{ss,R} \subseteq  \{ x \in X^0_{\min} \,\, | \,\, p(x) \in Z_{\min}^{ss,R} \}/U
$ by Lemma \ref{blimey}. Thus
$$X^{ss,\hH} \to (\overline{ \{ x \in X^0_{\min} \,\, | \,\, p(x) \in Z_{\min}^{ss,R} \}/U })/\!/R = X \env \hH$$ 
is surjective and is a good quotient for the action of $H$.
\end{rmk}

\begin{rmk}  \label{remweaken3}  
As noted at Remarks \ref{remweaken} and \ref{generalise}, the hypothesis in Theorem \ref{thm:ExCo1} that $\stab_U(z) = \{ e \}$ for every $z \in Z_{\min}$  ( condition \eqref{eq:sss})
can be weakened, when the action of $U$ is such that the stabiliser in $U$ of $x \in X$ is always strictly positive-dimensional. Provided that for every $z \in Z_{\min}$ the dimension of $\stab_U(z)$ is equal to the generic (or minimum) dimension $d^U_{\min}$ of $\stab_U(x)$ for $x \in X$, and that $\stab_U(z)$ has a complementary subgroup $U'$ in $U$ with $U' \cap \stab_U(z) = \{ e\}$ and $U = U'\stab_U(z)$, then the conclusions of Theorem \ref{thm:ExCo1} and thus Theorem \ref{mainthm} still hold. For observe that this condition implies that $\dim \stab_U(x) = d^U_{\min}$ for every $x \in X^0_{\min}$, since $\GG_m$ normalises $U$ and so $\dim \stab_U(p(x)) \geqslant \dim \stab_U(x)$ for any $x \in X$, where $p(x)$ is the limit of $t\cdot x$ as $t \to 0$ with $t \in \GG_m$.
We can define a $\hU$-invariant morphism from $X^0_{\min}$ to the Grassmannian $\rm{Grass}(d^U_{\min}, \rm{Lie}U)$ of $d^U_{\min}$-dimensional subspaces of $ \rm{Lie}U$ by
$$ x \mapsto \rm{Lie} \stab_U (x) .$$
If $x \in X^0_{\min}$ we choose a $\GG_m$-invariant complementary subgroup $U' \leqslant U$ such that $\stab_U(x) \cap U' = \{ e \}$ and $U' \stab_U(x) = U$. The condition of complementarity to $U'$ defines an affine open subscheme of $\rm{Grass}(d_{\min}, \rm{Lie}U)$ and thus an open subset of $X^0_{\min}$  where quotienting by U is equivalent to quotienting by $U'$, and where $\stab_{U'}(x) = \stab_U(x) \cap U' = \{ e\}$.

The condition that $\stab_U(z)$ has a complementary subgroup $U'$ in $U$ with $U' \cap \stab_U(z) = \{ e\}$ and $U = U'\stab_U(z)$ is always satisfied when the unipotent group $U$ is abelian. 
Thus as described in Remark \ref{generalise} the proof of Theorem \ref{thm:ExCo1} can be extended using quotienting in stages (described in $\S$5) to cover the more general case when  for every $z \in Z_{\min}$ the dimension of $\stab_U(z)$ is equal to the generic (or minimum) dimension $d^U_{\min}$ of $\stab_U(x)$ for $x \in X$, provided that the same is true for every subgroup $U^{(j)}$ in the derived series for $U$. Here  the derived series can be replaced with any series $U = U^{(0)} \geqslant U^{(1)} \geqslant \cdots \geqslant U^{(s)} = \{ e \}$ of normal subgroups with abelian subquotients. 

Moreover if $U$ is abelian  (or we work with normal subgroups 
$$  U = U^{(0)} \geqslant U^{(1)} \geqslant \cdots \geqslant U^{(s)} = \{ e\} $$
of  $\hH$ with abelian subquotients $U^{(j)}/U^{(j+1)}$), and if the adjoint action of $\l(\GG_m)$ on the Lie algebra of $U$ has only one weight space, then the proof of Theorem \ref{thm:ExCo1} can be extended further, using Remark \ref{rem20191}, to cover the case when there is a natural number $r$
such that the following property holds:

(*) If $z\in X^{\l(\GG_m)}$ where $\l(\GG_m)$ acts on $L^*|_z$ with weight strictly less than $\weight_{r}$ then 
$$\stab_{U}(z) = U,$$
whereas for all $z\in X^{\l(\GG_m)}$ where $\l(\GG_m)$ acts on $L^*|_z$ with weight equal to $\weight_{r}$ then 
$$ \stab_{U}(z) = \{e\}.$$
\end{rmk}

\begin{rmk} \label{remweaken4} 
By using Jordan canonical form to classify representations of semi-direct products $\GG_a \rtimes \GG_m$ when $\GG_m$ acts on $\rm{Lie} \, \GG_a$ with strictly positive weight (cf. \cite{BDHK} $\S$5 and \cite{BDHK2} $\S$2), the proof of  Theorem \ref{thm:ExCo1}
can also be extended  to allow  the requirement of well-adaptedness for the rational character $\chi$ to be weakened, provided that the condition $\eqref{eq:sss}$ is strengthened appropriately (cf. Remark \ref{remweaken2}).
Let $X^{\GG_m}$ be the  $\GG_m$--fixed point set in $X$. 
We can  weaken the requirement that the rational character $\chi$ should be well adapted if we strengthen the condition $\eqref{eq:sss}$ to become
$$ 
\text{$\stab_U(z) = \{e\}$ whenever $z \in X^{\GG_m}$ and $\GG_m$ acts on $L^*|_z$ with weight at most $\weight_j$.} 
$$ 
Under this stronger hypothesis we can allow ${\chi} = \weight_{j} + \epsilon$
where $\epsilon >0$ is sufficiently small.

In particular if the action of $\GG_a \rtimes \GG_m$ extends to an action of $\GL(2)$ (up to an appropriate cover), then we can allow any rational character $\chi < 0$. In the situation of the 
Popov-Pommerening conjecture \cite{Popov, PopVin} when $U$ is a subgroup of a reductive group $G$
and is normalized by a maximal torus $T$ of $G$ which contains $\GG_m$, then $U$ is spanned by additive one-parameter subgroups $\tilde{\lambda}: \GG_a \to U$ normalised by $T$ such that the inclusion of each $\GG_a \rtimes \GG_m$ in $G$ extends to a homomorphism from (a finite cover of) $\GL(2)$ to $G$. 
\end{rmk}

\section{When stability and semistability do not coincide for the unipotent radical}

\label{section5}

As before we will suppose that $H =U \rtimes R$ 
is a linear algebraic group over $\kk$ with externally graded unipotent radical $U$, in the sense that there is a semi-direct product
 $\hat{H} = H \rtimes \l(\GG_m) = U \rtimes (R \times \l(\GG_m))$  of $H$ by $\GG_m$ with subgroup $\hat{U} = U \rtimes \l(\GG_m)$
where the adjoint action of $\l(\GG_m)$ on $\mathrm{Lie}U$ has all weights
 strictly positive. 
We also assume as before that $\hH$  acts linearly on an irreducible projective scheme $X$ with respect to a very  ample line bundle $L$, and that $\chi: \hH \to \GG_m$ is a rational character of $\hat{H}$ with kernel containing $H$ which is well adapted for the linear action of $\hU$.

In $\S$\ref{ss=ssection}  we considered the good case when semistability coincides with stability and the stable locus is nonempty for the $\hU$-action on $X$, in the sense that condition \eqref{eq:sss}  is satisfied so every $x \in Z_{\min}$ has trivial stabiliser in $U$. We now want to show that if we only assume that there exists some $x \in Z_{\min}$ with trivial stabiliser\footnote{This weaker condition can itself be removed: see Remark \ref{remweaken5}.}
 in $U$, then there is an analogue of the partial desingularisation construction described in $\S$\ref{pdsection} which allows us to blow  $X$ up along a sequence of 
  $\hH$-invariant closed subschemes to obtain an $\hH$-invariant morphism 
$\psi: \tilde{X} \to X$ where $\tilde{X}$ is an irreducible projective scheme acted on 
linearly by $\hH$ satisfying condition  \eqref{eq:sss} and $\tilde{X} \env \hH = \tilde{X}^{s,\hH}/\hH$. Here $\tilde{X} \env \hH$ is a projective scheme with an open subscheme which is a geometric quotient by $\hH$ of an open subscheme of $X$, over which $\psi$ restricts to an isomorphism.

In fact before constructing $\tilde{X}$ we will first blow  $X$ up along a sequence of 
  $\hH$-invariant closed subschemes to obtain an $\hH$-invariant morphism 
$\psi: \hat{X} \to X$ where $\hat{X}$ is an irreducible projective scheme acted on 
linearly by $\hH$ satisfying condition  \eqref{eq:sss}, for which the enveloping quotient $\hat{X} \env \hH$ is a good quotient but not necessarily a geometric quotient of $ \hat{X}^{ss}$ by $\hH$. We can then blow $\hat{X}$ up along a further sequence of 
  $\hH$-invariant closed subschemes to obtain $\tilde{X}$.

In $\S$\ref{ss=ssection} we proved Theorem \ref{thm:ExCo1} and deduced Theorem \ref{mainthm}. Thus we know that when the linear action of $\hU$ on $X$ satisfies the condition 
 (\ref{eq:sss}), and moreover the linearisation is twisted by a suitable rational character of $\hH$ (with kernel $H$) so that it is well adapted, then when $c>0$ is a sufficiently divisible integer we have

\noindent (i) the algebras of invariants $\oplus_{m=0}^\infty H^0(X,L^{\otimes cm})^{\hat{U}}$ and 
$$\oplus_{m=0}^\infty H^0(X,L^{\otimes cm})^{\hat{H}} = (\oplus_{m=0}^\infty H^0(X,L^{\otimes cm})^{\hat{U}})^{R}$$ 
 are 
finitely generated, and 

\noindent (ii) the enveloping quotient $X\env \hat{U}$ is the projective scheme associated to the algebra of invariants $\oplus_{m=0}^\infty H^0(X,L^{\otimes cm})^{\hat{U}}$ and is a geometric quotient of the open subset $X^{s,\hU}_{\min+}$ of $X$ by $\hat{U}$, while 

\noindent (iii) the enveloping quotient $X\env \hat{H}$ is the projective scheme associated to the algebra of invariants $\oplus_{m=0}^\infty H^0(X,L^{\otimes cm})^{\hat{H}}$ and is the classical GIT quotient of $X\env \hat{U}$ by the induced action of $R \cong H/U \cong \hH/\hat{U}$ with respect to the linearisation induced by a sufficiently divisible tensor power of $L$.

In order to prove Theorem \ref{mainthm2}, just as for Theorem \ref{mainthm}, the main hurdle is to deal with the action of $\hU$; the rest will follow using well known methods from reductive GIT (including the partial desingularisation construction of \cite{K2} described in the last section) once the case when $H=U$ is unipotent is completed. Thus our next aim is to prove

\begin{thm} \label{mainthm2unipotent} 
Let $U$ be a  graded unipotent group
 and let $\hat{U} = U \rtimes \GG_m$ be the semi-direct product of $U$ by $\GG_m$ 
which defines the grading.
Suppose that $\hU$  acts linearly on an irreducible projective scheme $X$ with respect to a very ample line bundle $L$, and that 
 $\stab_U(x) = \{e\}$ for generic $x \in Z_{\min}$.

Then there is a  sequence of blow-ups of $X$ along $\hat{U}$-invariant projective subschemes 
  resulting in a projective scheme $\hat{X}$ with a 
 linear action of $\hat{U}$ (with respect to a power of an ample line bundle given by perturbing the pullback of $L$ by small multiples of the exceptional divisors for the blow-ups) which satisfies the  condition 
(\ref{eq:sss}), so that Theorem \ref{mainthm} applies.
\end{thm}

Once we have proved Theorem \ref{mainthm2unipotent}, and have shown that the sequence of blow-ups is sufficiently canonically defined that the $\hat{U}$ invariant subschemes along which we blow up are in fact $\hH$-invariant, then Theorem \ref{mainthm2} will follow. 


For this construction we will adopt a modification of the strategy described in $\S$\ref{section4} for the reductive case when $X^{ss} \neq X^s \neq \emptyset$. The first step in the strategy described in $\S$\ref{section4} is to blow $X^{ss}$ up along the closed subscheme consisting of the semistable points with maximal dimensional stabiliser.  In the case of a linear $\hH$-action for which the condition \eqref{eq:sss} is not satisfied, but with generic $U$-stabiliser in $Z_{\min}$ trivial, we first blow  $X^0_{\min}$ up along its subscheme consisting of points in $X^0_{\min}$ with maximal dimensional stabiliser in $U$. This subscheme is closed in $X^0_{\min}$ by the upper semi-continuity of the dimension of the stabiliser, and is $\hH$-invariant since $U$ is a normal subgroup of $\hH$. Equivalently we can blow $X$ itself up along the closure of this subscheme to obtain a projective scheme $X_{(1)}$ with a blow-down map $\psi_{(1)}:X_{(1)} \to X$ with $(X_{(1)})^0_{\min} \subseteq \psi_{(1)}^{-1}(X^0_{\min})$. 

The induced action of $\hH$ on $X_{(1)}$ is linear with respect to an ample line bundle which is (a tensor power of) the pullback $\psi_{(1)}^*L$ perturbed by a sufficiently small rational multiple $\epsilon_{(1)} >0$ of the exceptional divisor $E_{(1)}$, and if $\epsilon_{(1)} >0$ is small enough then it follows from the Hilbert--Mumford criteria for the action of $\GG_m$ that  $\psi_{(1)}$ restricts to a morphism
$$ \psi_{(1)} : (X_{(1)})^0_{\min} \to X^0_{\min}.$$
The crucial result is Proposition \ref{propertransform} 
 below, which tells us that the dimension of the stabiliser in $U$ of any 
$x \in (X_{(1)})^0_{\min} $ is strictly less than the maximal dimension of a $U$-stabiliser in $X^0_{\min} $. Thus repeating this process finitely many times will result in an $\hH$-invariant blow-up 
$$ \hat{\psi}: \hat{X} \to X$$
such that the induced linear action of $\hH$ on $\hat{X}$ satisfies \eqref{eq:sss}. This means that if $\hat{Z}_{\min}$ plays the role of $Z_{\min}$ for $\hat{X}$, then $\hat{X}^0_{\min} \setminus U\hat{Z}_{\min}$ has a geometric quotient
$$ q^{\hat{X}}_{\hU} : \hat{X}^0_{\min} \setminus U\hat{Z}_{\min} \to (\hat{X}^0_{\min} \setminus U\hat{Z}_{\min})/\hU$$
by $\hU$ which is a projective scheme $\hat{X} \env \hU$ with an induced ample line bundle $\hat{L}_{\hU}$ and linear action of $R=H/U \cong \hH/\hU$. Moreover we can construct $\hat{X} \env \hH$ as a reductive GIT quotient 
$(\hat{X} \env \hU)/\!/R $
for an appropriate linearisation of the $R$-action on $\hat{X} \env \hU$ with respect to the line bundle $\hat{L}_{\hU}$.

More precisely, for a well-adapted linearisation with respect to an ample line bundle $\hat{L}$ on $\hat{X}$ which is a (tensor power of a) suitable perturbation of $\hat{\psi}^*L$ by a small rational linear combination of the proper transforms of the exceptional divisors, the $\hH$-action on $\hat{X}$ satisfies

\noindent (i) if $c>0$ is a sufficiently divisible integer then the algebras of invariants $\oplus_{m=0}^\infty H^0(\hat{X},\hat{L}^{\otimes cm})^{\hat{U}}$ and 
$$\oplus_{m=0}^\infty H^0(\hat{X},\hat{L}^{\otimes cm})^{\hat{H}} = (\oplus_{m=0}^\infty H^0(\hat{X},\hat{L}^{\otimes cm})^{\hat{U}})^{R}$$ 
 are 
finitely generated, and 

\noindent (ii) the enveloping quotient $\hat{X}\env \hat{U}$ is the projective scheme associated to the algebra of invariants $\oplus_{m=0}^\infty H^0(\hat{X},\hat{L}^{\otimes cm})^{\hat{U}}$ and provides a geometric quotient $q^{\hat{X}}_{\hU}: \hat{X}^{s,\hU}_{\min+}  \to \hat{X}\env \hat{U}$ of the open subscheme $\hat{X}^{s,\hU}_{\min+} = \hat{X}^0_{\min} \setminus U\hat{Z}_{\min}
$ of $\hat{X}$ by $\hat{U}$, while 

\noindent (iii) the enveloping quotient $\hat{X}\env \hat{H}$ is the projective scheme associated to the algebra of invariants $\oplus_{m=0}^\infty H^0(\hat{X},\hat{L}^{\otimes cm})^{\hat{H}}$ and is the classical GIT quotient of $\hat{X}\env \hat{U}$ by the induced action of $R \cong H/U \cong \hH/\hat{U}$ with respect to the linearisation induced by a sufficiently divisible positive tensor power of $\hat{L}$. 

In addition we can apply the partial desingularisation construction of \cite{K2} as described in $\S$\ref{section4} to the action of $R$ on $\hat{X}\env \hat{U}
=(\hat{X}^0_{\min} \setminus U\hat{Z}_{\min})/\hU$, or equivalently blow $\hat{X}$ up along the closure of the subscheme
$$ \{ x \in \hat{X}^0_{\min} \setminus U\hat{Z}_{\min} \,\, | \,\, q^{\hat{X}}_{\hU} (x) \in 
(\hat{X}\env \hat{U})^{ss,R} \mbox{ and $\dim\stab_{\hH}(x)$ is maximal among such $x$} \}
,$$
and repeat until we obtain a morphism $\tilde{\psi}: \tilde{X} \to \hat{X} \to X$ restricting to $\tilde{X}^0_{\min} \to \hat{X}^0_{\min} \to X^0_{\min}$ and a partial desingularisation $\tilde{X} \env \hH \to \hat{X} \env \hH$, where
$\tilde{X}\env \hH = \tilde{X}^{s,\hH}/\hH$ is a geometric quotient of an open subscheme $\tilde{X}^{s,\hH}$ of $\tilde{X}$.


Before stating and proving the crucial Proposition \ref{propertransform},  we need to introduce some notation. Recall that
$$\weight_{\min} = \weight_0 < \weight_1 < \ldots < \weight_{\max}$$
are the weights with which the one-parameter subgroup $\l(\GG_m)$ of $\hU$ acts on the fibres of $L^*$ over points of $X$ fixed by $\l(\GG_m)$.

\begin{defn} \label{defn:RD}
Let $R= \{ \weight_{\min}, \weight_{\min} + 1, 
\ldots, \weight_{\max} \}$ and 
$D^U = \{ d^U_{\min}, d^U_{\min} + 1, \ldots , d^U_{\max} \} $
where 
$$d^U_{\min} = d_{\min} = \min \{ \dim \stab_U(x) \,\, | \,\, x \in X^0_{\min} \}$$
and 
$$d^U_{\max} = d_{\max} = \max \{ \dim \stab_U(x) \,\, | \,\, x \in X^0_{\min} \}.$$
For $r \in R$ let
$$ Z_{((r))} = \{ x \in X^0_{\min} \,\, | \,\, \GG_m \mbox{ acts on the fibre of $L^*$ over $\lim_{t \to \infty} tx$ with weight $\leqslant r$} \}$$
so that $Z_{((\weight_{\min}))} = Z_{\min}$ and $Z_{((\weight_{\max}))} = X$. For $\delta \in D^U$ let
$$ \Delta_U^{>\delta} = \{ x \in X^0_{\min} \,\, | \,\, \dim \stab_U(x) > \delta \}\,\,  \mbox{ and } \,\, 
\Delta_U^{\geqslant \delta} = \{ x \in X^0_{\min} \,\, | \,\, \dim \stab_U(x) \geqslant \delta \} $$.
\end{defn}


\begin{rmk}
$p:X^0_{\min} \to Z_{\min}$ restricts to $p: \Delta_U^{\geqslant d^U_{\max}} \to Z_{\min}^{d^U_{\max}} = Z_{\min} \cap \Delta_U^{\geqslant d^U_{\max}}$ with fibre over $z \in Z_{\min}^{d^U_{\max}}$ given by the fixed point set of $\stab_U(z)$ acting on $p^{-1}(z)$
(cf. Remark \ref{maxdimstab}).
\end{rmk}

\begin{defn} \label{defnblowup}
If $d^U_{\max} = d^U_{\min}$ let $X_{(1)} =  X^0_{\min}$ and let $\bar{X}_{(1)} = X$.
If $d^U_{\max} > d^U_{\min}$ 
let $\psi_{(1)} : X_{(1)} \to  X^0_{\min}$ be the blow-up of $ X^0_{\min}$ along the closed subscheme $\Delta_U^{\geqslant d_{\max}}$ with exceptional divisor $E_{(1)}$, and let $\bar{\psi}_{(1)}: \bar{X}_{(1)} \to X$ be the blow-up of $X$ along the closure $\overline{ \Delta_U^{\geqslant d_{\max}}}$ of $\Delta_U^{\geqslant d_{\max}}$ in $X$ (or in a blow-up of $X$ along a sequence of closed $\hU$-invariant subschemes of the complement of $ X^0_{\min}$, cf. Remark \ref{cfK2}). 
\end{defn}

\begin{propn} \label{propn} Suppose that $d^U_{\max} > d^U_{\min}$.
There are perturbations $\mathcal{L}_{(1),\epsilon}$ (for $\epsilon >0$ sufficiently small and rational) of the pullback to $\bar{X}_{(1)}$ of the given linearisation $\mathcal{L}$ of the action of $\hU$ on $X$, which are ample linearisations for the induced action of $\hU$ on $\bar{X}_{(1)}$ with the following properties:

\noindent (i) their restrictions to $X_{(1)}$ are tensor products of the pullback of $\mathcal{L}$ by arbitrarily small rational multiples $\epsilon$ of the exceptional divisor $E_{(1)}$;

\noindent (ii) the induced subschemes $X_{(1)}^0$ and $Z_{(1)}$ of $\bar{X}_{(1)}$ (defined as $ X^0_{\min}$ and $ Z_{\min}$ with $X$ replaced with $\bar{X}_{(1)}$ using the linearisation $\mathcal{L}_{(1),\epsilon}$ for $\epsilon >0$ sufficiently small) are contained in the open subscheme  $X_{(1)}$ of $\bar{X}_{(1)}$  and are independent of $\epsilon$;

\noindent (iii) $\psi_{(1)}(Z_{(1)}) \subseteq  Z_{\min}$;

\noindent (iv) if $x \in  X^0_{(1)}$ then $\stab_U(x) \leq \stab_U(\psi_{(1)}(x))$ and so $\dim \stab_U(x) \leq d_{\max}$.
\end{propn}
\noindent{\bf Proof:} This follows from the arguments of \cite{K2} (cf. Remark \ref{cfK2}). By embedding $X$ in the projective space  $\PP(H^0(X,L^{\otimes r})^*)$ for $r>\!>1$ we can assume that $X$ is a projective space on which $U$ acts linearly, and thus that $ Z_{\min}$, $ X^0_{\min}$ and $\Delta_U^{\geqslant d_{\max}}$ are all nonsingular $\hH$-invariant locally closed subschemes of $X$. By 
upper semi-continuity of stabiliser dimension, $ \Delta_U^{\geqslant d_{\max}}$ is closed in 
$X^0_{\min}$, so we can resolve the singularities of its closure $\overline{ \Delta_U^{\geqslant d_{\max}}}$ in $X$ by a sequence of blow-ups along nonsingular $\hat{U}$-invariant closed subschemes in the complement of $X^0_{\min}$ and then blow up along the resulting proper transform of $\overline{  \Delta_U^{\geqslant d_{\max}}}$ to get $\bar{X}_{(1)}$. It follows that $\bar{\psi}^{-1}(X^0_{\min})$ is the blow-up of $X^0_{\min}$ along $ \Delta_U^{\geqslant d_{\max}}$. Moreover as in \cite{K2} we can find an ample linearisation of the action of $\hU$ on $\bar{X}_{(1)}$ which is a (positive tensor power of a) perturbation of $\bar{\psi}_{(1)}^*(\mathcal{L})$ by a linear combination of the proper transforms of the exceptional divisors in these blow-ups. Properties (i)-(iv) then follow exactly as in \cite{K2}.   \hfill $\Box$

\begin{defn} Let $\bar{X}_{(0)} = X$ and define $\psi_{(j)}: X_{(j)} \to X^0_{(j-1)} = (\bar{X}_{(j-1)})^0_{\min}$ and 
$\bar{\psi}_{(j)}: \bar{X}_{(j)} \to \bar{X}_{(j-1)}$ recursively on $j \geq 1$ by replacing $X$ with $\bar{X}_{(j-1)}$ and replacing $(1)$ with $(j)$ throughout Definition \ref{defnblowup} and Proposition \ref{propn}.
\end{defn}

\begin{rmk} \label{remsection5} In this recursive definition $X_{(j)}$ is the blow-up of $(\bar{X}_{(j-1)})^0_{\min}$ along its closed subscheme
$$ \{ x \in (\bar{X}_{(j-1)})^0_{\min} \,\,\, | \,\,\, \dim \stab_U(x) = d^U_{\max}(\bar{X}_{(j-1)}) \}$$
where $d^U_{\max}(\bar{X}_{(j-1)}) = \max \{ \dim \stab_U(x) \,\, | \,\, x \in (\bar{X}_{(j-1)})^0_{\min} \}$ is defined for $\bar{X}_{(j-1)}$ as $d^U_{\max} = d^U_{\max}(X)$ is defined for $X$. It will follow from Proposition \ref{propertransform}  below that if $\dim\stab_U(z) = d^U_{\min}$ for generic $z \in Z_{\min}$ then the same is true for each $\bar{X}_{(j)}$, and if
$$d^U_{\max}(\bar{X}_{(j-1)}) > d^U_{\min}(\bar{X}_{(j-1)}) $$
then
$d^U_{\max}(\bar{X}_{(j)}) < d^U_{\max}(\bar{X}_{(j-1)})$. This will ensure that the process terminates with some $\bar{X}_{(j)}= \hat{X}$ with blow-down morphism $\hat{\psi}:\hat{X} \to X$ such that the dimension of $\stab_U(x)$ is constant (and equal to $d_{\min}^U$) for $x \in \hat{X}^0_{\min}$. If this constant dimension is 0 then the induced linear action of $U$ on $\hat{X}$ satisfies the condition \eqref{eq:sss} and hence $ \hat{X}^0_{\min} \setminus U \hat{Z}_{\min}$ has a geometric quotient by $\hU$ which is a projective scheme. This geometric quotient can be regarded as a projective completion of a geometric quotient by $\hU$ of an $\hH$-invariant open subscheme 
 of $X$.

If $\dim\stab_U(z) > 0$ for all $z \in Z_{\min}$ then we can still use Remark \ref{remweaken5} below and this construction to find a geometric quotient by $\hU$ of a nonempty $\hH$-invariant open subscheme 
of $X$, with a projective completion which is a geometric quotient by $\hU$ of an $\hH$-invariant open subscheme of a blow-up of $X$.
\end{rmk}

\begin{propn} \label{propertransform} 
Suppose that $d^U_{\max}$ is strictly greater than $d^U_{\min}$ and that 
$\dim\stab_U(z) = d^U_{\min}$ for generic $z \in Z_{\min}$, so  that   $Z_{\min} \not\subseteq \Delta_U^{\geqslant d^U_{\max}} $. Let $X_{(1)}^0$ and $Z_{(1)}$ be as in Proposition \ref{propn}. Then 

\noindent (a) $Z_{(1)}$ is the proper transform of $ Z_{\min}$ in $X_{(1)}$, and this is the blow-up $ \tilde{Z}_{\min}$ of $ Z_{\min}$ along $ Z_{\min}^{d^U_{\max}}  = Z_{\min} \cap \Delta_U^{\geqslant d^U_{\max}}$;

\noindent (b) no point $z$ in this proper transform such that $y = \psi_{(1)}(z) \in Z^{d^U_{\max}}_{\min}$  is fixed by $\stab_U(\psi_{(1)}(z))$; 

\noindent (c) if $z \in Z_{(1)}$ then $\dim \stab_U(z) < d^U_{\max}$;

\noindent (d) the complement of the exceptional divisor $ E_{(1)}$ in $X_{(1)}^0 \setminus UZ_{(1)}$ is identified via the blow-down map $\psi_{(1)}$ with the open subset
$$ \{ x \in X^0_{\min} \setminus U {Z}_{\min} \mid  \,\, p(x) \notin {Z}_{\min}^{d^U_{\max}} \}$$
of $X^0_{\min} \setminus U Z_{\min}$, where as before
$ {p}(x)$ is the limit $ \lim_{t \to 0} t \cdot x $ with $ t \in \GG_m$. 
\end{propn}
\begin{proof}    
By assumption $Z_{\min}$ is not contained in the centre $\Delta_U^{\geqslant d^U_{\max}} $ of the blow-up. Hence $Z_{(1)}$ is  the proper transform of $Z_{\min}$ in the blow-up $X_{(1)}$ of $X^0_{\min}$ along $\Delta_U^{\geqslant d^U_{\max}}$, which is the blow-up of $Z_{\min}$ along $Z_{\min}^{d^U_{\max}}$ by the proof of Lemma \ref{lemmaB}. 

Now suppose that $z \in Z_{(1)} = \tilde{Z}_{\min}$ and that $y = \psi_{(1)}(z) \in Z^{d^U_{\max}}_{\min}$. Let $U' = \stab_U(y)$. We want to prove that $U'$  does not fix $z$. For this we can assume, as in the proof of  Proposition \ref{propn}, that $X$ is a projective space with a linear action of $\hU$. Then
$$(Z_{\min})^{U'} \subseteq Z_{\min}^{d^U_{\max}} \subseteq Z_{\min}$$
where $Z_{\min}$ and $(Z_{\min})^{U'}$ are linear subspaces of $X$ fixed pointwise by $\GG_m$, while $(Z_{\min})^{U'}$ is also fixed pointwise by $U'$. Recall 
the identity $p(ux)=x$ for all $x \in Z_{\min}$ and $u \in U$.
We can find $0 \leq m_1 \leq m_2 \leq m_3$ and homogeneous coordinates $[x_0: \ldots :x_{m_3}]$ on $X$ such that $\GG_m$ acts diagonally and $y=[0:\ldots 0:1]$, 
$$(Z_{\min})^{U'} = \{ [x_0: \ldots :x_{m_3}]  \mid x_j = 0 \mbox{ for } j \leq m_2 \}  \mbox{ and } 
Z_{\min} = \{ [x_0: \ldots :x_{m_3}]  \mid x_j = 0 \mbox{ for } j \leq m_1 \} ,$$
while $p([x_0: \ldots :x_{m_3}] ) = [0: \ldots :0: x_{m_1 + 1}: \ldots :x_{m_3}] $ and $U'$ acts on $X$ via matrices of the block form
$$ \left( \begin{array}{ccc} A & B & 0 \\ 0 & I & 0 \\ 0 & 0 & I \end{array} \right).$$
By considering local coordinates $x_j/x_{m_3}$ on $X$ near $y$, we see that if 
$\xi \in T_yZ_{\min} \setminus T_yZ_{\min}^{U'}$, then there is some  $u \in U'$ such that 
$ u \xi - \xi \notin   T_yZ_{\min}$.
Any $z \in \psi_{(1)}^{-1}(y)$ can be represented by such a $\xi$ and hence is not fixed by $U'$,
 as required for (b).  

(c) follows immediately from (b) and Proposition \ref{propn} (iv).  Finally for (d) observe that  if  we set
$$ p_{(1)}(x) =  \lim_{t \to 0, t \in \GG_m} t \cdot x $$
for $x \in \bar{X}_{(1)}$, then
$p_{(1)}(x) \in Z_{(1)}$ implies $x \in {X}_{(1)}$, while
 if $x \in {X}_{(1)} \setminus (U Z_{(1)} \cup E_{(1)})$ then ${p}_{(1)}(x) \in Z_{(1)}$ if and only if $p(\phi_{(1)}(x)) \in Z_{\min} \setminus Z_{\min}^{d^U_{\max}}$. The result follows.
 \end{proof} 


The following theorem, and therefore also Theorem \ref{mainthm2unipotent}, follow immediately from Proposition \ref{propertransform} . 

\begin{thm} \label{thmblowup}
Suppose that $U$ is a  graded unipotent group
 with $\hat{U} = U \rtimes \GG_m$ the semi-direct product of $U$ by $\GG_m$ 
which defines the grading.
Suppose that $\hU$  acts linearly on an irreducible projective scheme $X$ with respect to a very ample line bundle $L$, and that $\stab_U(x) = \{e\}$ for generic $x \in Z_{\min}$. Then after repeating the blow-up construction of Definition \ref{defnblowup} finitely many times, we obtain a projective scheme $\hat{X}$ with a linear $\hU$-action and a $\hU$-equivariant birational morphism $\psi_{\hat{X}}: \hat{X} \to X$ and an ample linearisation which is a tensor power of a small perturbation of $\psi_{\hat{X}}^*(\mathcal{L})$ and satisfies condition 
(\ref{eq:sss}). 

\end{thm}

\begin{proof}
By Proposition \ref{propertransform} 
 the blow-up construction terminates after finitely many steps with $\hat{X}$ satisfying the conditions required in Theorem \ref{thmblowup}.
\end{proof}

Suppose as usual that $H=U \rtimes R$ is a linear algebraic group over $\kk$ with graded unipotent radical $U$
 and  $\hat{H} = H \rtimes \l(\GG_m) = U \rtimes (R \times \l(\GG_m))$ is the semi-direct product of $H$ by $\GG_m$ 
which defines the grading, and that  $\hH$  acts linearly on an irreducible projective scheme $X$ with respect to a very ample line bundle $L$. If $\stab_U(x) = \{e\}$ for generic $x \in Z_{\min}$, then since $U$ and $\hU$ are normal subgroups of $\hH$ and the subgroup $\hU /U \cong \l(\GG_m)$ in $\hH/U\cong R \times \l(\GG_m)$ is central, $\hat{X}$ can be constructed from $X$ in Theorem \ref{thmblowup} in an $\hH$-invariant way. Thus Theorem \ref{mainthm2} follows by combining Theorems \ref{thmblowup} and \ref{mainthm}.

\begin{rmk} \label{remweaken5}
If we do not have the condition that $\stab_U(x)$ is trivial for generic $x\in Z_{\min}$,  then the proof of Proposition \ref{propertransform} can be extended to show that the blow-up construction will still terminate after finitely many steps with $\hat{X}$ not satisfying condition (\ref{eq:sss}) but instead satisfying  $\dim\stab_U(x) = d^U_{\min}$ for every
$x \in \hat{X}^0_{\min}$. However in this more general situation, instead of using the blow-up construction directly, we can combine it with quotienting in stages (as described in $\S$5).
Recall from  Remark \ref{remweaken3} that the \eqref{eq:sss} hypothesis in Theorem \ref{mainthm} can be weakened, when the action of $U$ is such that the stabiliser in $U$ of $x \in X$ is always strictly positive-dimensional. Indeed, as was observed in Remark \ref{remweaken3},
provided that $\dim\stab_U(x) = d^U_{\min}$ for every
$x \in \hat{X}^0_{\min}$, and that $\stab_U(x)$ has a complementary subgroup $U'$ in $U$ with $U' \cap \stab_U(x) = \{ e\}$ and $U = U'\stab_U(z)$, then the conclusions of Theorem \ref{mainthm} still hold (although now  $X^{Ms, \hU}_{\min+}$ would be better notation than $X^{s, \hU}_{\min+}$, and $X^{Ms, \hH \geqslant \hU}_{\min+}$ might be better notation than $X^{s, \hH}_{\min+}$). Moreover as was also observed in Remark \ref{remweaken3},
the condition that $\stab_U(z)$ has a complementary subgroup $U'$ in $U$ with $U' \cap \stab_U(z) = \{ e\}$ and $U = U'\stab_U(z)$ is always satisfied when the unipotent group $U$ is abelian. 
Thus as described in Remark \ref{generalise} the proof of Theorem \ref{mainthm} can be extended to cover the more general case when  $\dim\stab_U(x) = d^U_{\min}$ for every
$x \in \hat{X}^0_{\min}$, provided that the same is true for every subgroup $U^{(j)}$ in the derived series for $U$ (or any series $U = U^{(0)} \geqslant U^{(1)} \geqslant \cdots \geqslant U^{(s)} = \{ e \}$ of normal subgroups of $\hH$ with abelian subquotients). Since each $U^{(j)}$ is a characteristic subgroup of $\hH$, the blow-up construction for $U^{(j)}$ is $\hH$-equivariant. So if need be the blow-up construction can be applied iteratively for each $U^{(j)}$, in order to obtain $\hat{X}$ with a linear $\hH$-action for which the conclusions of Theorem \ref{mainthm} hold. Here there is no necessity to extend the proof of Theorem \ref{thmblowup} to show that the blow-up construction always terminates after finitely many steps, even without the condition that $\stab_U(x)$ is trivial for generic $x\in Z_{\min}$. For we can always find
subgroups
$$ U = U^{(0)} \geqslant U^{(1)} \geqslant \cdots \geqslant U^{(s)} = \{ e\} $$
of $U$ which are all normal in $\hH$, have abelian subquotients $U^{(j)}/U^{(j+1)}$ such that for each $j$ the adjoint action of $\l(\GG_m)$ on the Lie algebra of $U^{(j)}/U^{(j+1)}$ has a single weight space, and which satisfy the following property for natural numbers $r_1, \ldots, r_s$:\\
(***) If $1 \leqslant j \leqslant s$ and $z\in X^{\l(\GG_m)}$ where $\l(\GG_m)$ acts on $L^*|_z$ with weight strictly less than $\weight_{r_j}$ then 
$$ U^{(j)} = U^{(j+1)} \stab_{U^{(j)}}(z),$$
whereas for generic $z\in X^{\l(\GG_m)}$ where $\l(\GG_m)$ acts on $L^*|_z$ with weight equal to $\weight_{r_j}$ then 
$$ \stab_{U^{(j)}}(z) \leqslant U^{(j+1)}.$$
Then we quotient in stages, combining the argument above with Remark \ref{remweaken3} to prove
Theorem \ref{mainthm3}.
\end{rmk}

\section{Quotients by linear algebraic groups with externally graded unipotent radicals} \label{externalsection}

Let $\mathcal{L}$ be a very ample linearisation with respect to a line bundle $L \to X$ 
of  the action of $H=U \rtimes R$ on an irreducible projective scheme $X$ over $\kk$. As usual we will assume
that there is an extension of this linearisation to a linearisation (by abuse of notation also denoted by $\mathcal{L}$) of the action of a semi-direct product
 $\hat{H} = H \rtimes \l(\GG_m) = U \rtimes (R \times \l(\GG_m))$ of $H$ by $\GG_m$ 
which defines an external  grading of the unipotent radical $U$ of $H$, in the sense that the weights of the adjoint action of  $\l(\GG_m)$ on the Lie algebra of $U$ are all strictly positive.
Let $\chi: \hat{H} \to \GG_m$ be a character of $\hat{H}$ with kernel containing $H$; we will identify such characters $\chi $ with integers so that the integer 1 corresponds to the character which defines the exact sequence $H \to \hat{H} \to \GG_m$. 
Recall that we can twist the linearisation of the $\hat{H}$-action by multiplying the lift of the $\hH$-action to $L$ by such a character; this will leave the $H$-linearisation on $L$ and the action of $\hH$ on $X$ unchanged.




Suppose that the action of $\hH$ on $X$ satisfies the condition \eqref{eq:sss}; then the same is true for the action of $\hH$ on $X \times \PP^1$, where $\hH$ acts on $\PP^1$ as multiplication by the character which defines the exact sequence $H \to \hat{H} \to \GG_m$. 
We have 
$$ (X \times \PP^1)^0_{\min} = X^0_{\min} \times \AL^1$$
and $p_{X\times \PP^1} : X^0_{\min} \times \AL^1 \to Z_{\min} \times \{ 0 \}$ is given by $p_{X\times \PP^1} (x,t) = (p(x),0)$. 
Applying Theorem \ref{mainthm}  with $X$ replaced by $X \times \PP^1$, with respect to the tensor power of the linearisation $L \to X$ with $\calo_{\PP^1}(M)$ for $M\! >\! >\! 1$, twisted by a suitable character so that it is well adapted, gives us a  projective scheme $X \wenv H = (X \times \PP^1) \env \hat{H}$, which is a good quotient by $\hat{H}$ of an $\hat{H}$-invariant open subscheme of $X \times \AL^1$, given as the inverse image in $(X \times \PP^1)^{s,\hU}_{\min+}$ of the $R$-semistable subset of $(X\times \PP^1) \env \hat{U} = (X \times \PP^1)^{s,\hU}_{\min+}/\hU$. The projective scheme $X \wenv H$
contains as an open subscheme a geometric quotient by $H$ of an $H$-invariant open subset $X^{\hat{s},H}$ of $X$, which is the intersection with $X \times \{ [1:1]\}$ of the inverse image in $(X \times \PP^1)^{s,\hU}_{\min+}$ of the $R$-stable locus of 
$$(X\times \PP^1) \env \hat{U} =  ((X^0_{\min} \times (\AL^1 \setminus \{0\})) \sqcup (X^{s,\hU}_{\min+}  \times \{0\}))/\hU \cong 
(X^0_{\min}/U) \sqcup (X^{s,\hU}_{\min+}  /\hU).$$
Here $X^{\hat{s},H} = \{ x \in X^0_{\min} | p(x) \in Z_{\min}^{s,R}\}$ where $Z_{\min}^{s,R}$ is the stable locus for the induced linear $R$-action on $Z_{\min}$. 

The quasi-projective geometric quotient $X^{\hat{s},H}/H$ and the projective scheme $(X \times \PP^1) \env \hat{H} \supseteq X^{\hat{s},H}/H$ which contains it as an open subscheme can be described using Hilbert--Mumford criteria and an analogue of S-equivalence, combining the description of 
$(X\times \PP^1) \env \hat{U}$ as the geometric quotient $ (X \times \PP^1)^{s,\hU}_{\min+}/\hU$ with classical GIT for the induced linear action on $(X\times \PP^1) \env \hat{U}$ of the reductive group $R \cong H/U$. If $X^{\hat{s},H}$ is nonempty then 
 $(X \times \PP^1) \env \hat{H}$ is a projective completion of  $X^{\hat{s},H}/H$, and there is a surjective $\hH$-invariant morphism $\hat{\phi}$ onto  $(X \times \PP^1) \env \hat{H}$  from the open subscheme 
$$(X \times \AL^1)^{\hat{ss},\hH} =  \bigcap_{h \in H} h \,\, (X \times \AL^1)^{ss,T}_{\min+}$$
of $X \times \AL^1$, where $(X \times \AL)^{ss,T}_{\min+}$ is the $T$-semistable set in the sense of classical GIT for a well adapted linearisation of the action on $X \times \PP^1$ with respect to the tensor power of $L $ with $\calo_{\PP^1}(M)$ for $M>\!>1$.  When $(x,s)$ and $(y,t)$ lie in $(X \times \AL^1)^{\hat{ss},H}$
then $\hat{\phi}(x,s) = \hat{\phi}(y,t)$ if and only if the closures of their $\hH$-orbits meet in $(X \times \AL^1)^{\hat{ss},H}$.

Similarly when the action of $\hH$ on $X$ does not satisfy \eqref{eq:sss} but instead satisfies the weaker condition that $\stab_U(z)$ is trivial for generic $z \in Z_{\min}$ then  the same is true for the action of $\hH$ on $X \times \PP^1$ and we can apply Theorem \ref{mainthm2} to this action.
 Then \\
(i)  there is a sequence of blow-ups of $X \times \PP^1$ along $\hat{H}$-invariant projective subschemes  resulting in a projective scheme $\widehat{X \times \PP^1}$ with a well adapted linear action of $\hat{H}$ (with respect to a power of an ample line bundle given by tensoring the pullback of $L\otimes \calo_{\PP^1}(M)$ for $M\! >\! >\! 1$ with small multiples of the exceptional divisors for the blow-ups) which satisfies the  condition \eqref{eq:sss}, so that Theorem \ref{mainthm} applies, giving us a projective geometric quotient
$$\widehat{X \times \PP^1} \env \hat{U} = (\widehat{X \times \PP^1})^{s,\hU}_{\min+}/\hU$$
and a projective good quotient $\widehat{X \times \PP^1} \env \hH = (\widehat{X \times \PP^1} \env \hat{U} ) /\!/ R$ of 
the $\hH$-invariant open subset $ \bigcap_{h \in H} h \, \widehat{X\times \PP^1}^{ss,T}_{\min+} $ of $\widehat{X\times \PP^1}$; \\
(ii)  there is a sequence of further blow-ups along $\hat{H}$-invariant projective subschemes resulting in a projective scheme $\widetilde{X \times \PP^1}$ satisfying the same conditions as $\widehat{X\times \PP^1}$ and in addition the enveloping quotient $\widetilde{X\times \PP^1}\env \hat{H} $ is the geometric quotient by $\hH$ of  the  $\hH$-invariant open subscheme $\widetilde{X \times \PP^1}^{s,\hH}_{\min+}$;\\
(iii) the blow-down maps $\hat{\psi}:\widehat{X\times \PP^1} \to X \times \PP^1$ and  $\tilde{\psi}:\widetilde{X \times \PP^1} \to X \times \PP^1$ are isomorphisms over the open subscheme
$$ \{ (x,t) \in X^0_{\min} \times (\AL^1 \setminus \{0\}) \,\,  | \,\,  \stab_U(p(x)) = \{ e \} \mbox{ and } p(x) \in Z_{\min}^{s,R} \},$$
and this has a geometric quotient by $\hH$ which can be identified on the one hand via $\hat{\psi}$ and $\tilde{\psi}$ with open subschemes of the projective schemes $\hat{X}\env \hat{H} $ and $\tilde{X}\env \hat{H} $, and on the other hand with the geometric quotient
$$ \{ x \in X^0_{\min}  \,\,  | \,\,  \stab_U(p(x)) = \{ e \} \mbox{ and } p(x) \in Z_{\min}^{s,R} \}/U.$$
%

\begin{defn} \label{defn9.1}
If $\stab_U(z)$ is trivial for generic $z \in Z_{\min}$ then
$$X^{\hat{s},H} = \{ x \in X^0_{\min} \,\,  | \,\,  \stab_U(p(x)) = \{ e \}  \mbox{ and } p(x) \in Z^{s,R}_{\min} \}.$$
\end{defn}

In the most general situation when even the weaker condition that $\stab_U(z)$ is trivial for generic $z \in Z_{\min}$ is not necessarily satisfied, we can apply Theorem \ref{mainthm3}  with $X$ replaced by $X \times \PP^1$, with respect to the tensor power of the linearisation $L \to X$ with $\calo_{\PP^1}(M)$ for $M\! >\! >\! 1$, twisted by a suitable character so that it is well adapted. We obtain projective schemes
$$X \wenv H = \widehat{X \times \PP^1} \env \hat{H} \,\,\, \mbox{ and } \,\,\, X \tenv H = \widetilde{X \times \PP^1} \env \hat{H}$$ which are good (and in the latter case geometric) quotients by $\hat{H}$ of $\hat{H}$-invariant open subschemes of blow-ups of ${X} \times \AL^1$,  and each contain as an open subscheme a geometric quotient of an $H$-invariant open subscheme  of $X$ by $H$. 
This open subscheme  corresponds under $\hat{\phi}$ to the complement of the exceptional divisors for the blow-up $\hat{X}$ of $X$ in $$ \{ x \in \hat{X}^0_{\min} \,\, | \,\, q^{\widehat{X\times \PP^1}}_{\hU} (x,[1:1]) \in  ((\widehat{X \times \PP^1}) \env \hat{U})^{s,R} \}, $$ and similarly for $\tilde{\phi}$. 
If $\stab_U(z)$ is trivial for generic $z \in Z_{\min}$ then this open subscheme is 
$$X^{\hat{s},H} = \{ x \in X^0_{\min} \,\,  | \,\,  \stab_U(p(x)) = \{ e \}  \mbox{ and } p(x) \in Z^{s,R}_{\min} \} 
$$
where $p(x) = \lim_{t\to 0} \l(t)x$.
In general recall from Remark \ref{remweaken5} that we can find subgroups
$$ U = U^{(0)} \geqslant U^{(1)} \geqslant \cdots \geqslant U^{(s)} = \{ e\} $$
of $U$ which are all normal in $\hH$, have abelian subquotients $U^{(j)}/U^{(j+1)}$ such that for each $j$ the adjoint action of $\l(\GG_m)$ on the Lie algebra of $U^{(j)}/U^{(j+1)}$ has a single weight space with weight $s_j$, and which satisfy the following property for natural numbers $r_1, \ldots, r_s$:\\
(***) If $1 \leqslant j \leqslant s$ and $z\in X^{\l(\GG_m)}$ where $\l(\GG_m)$ acts on $L^*|_z$ with weight strictly less than $\weight_{r_j}$ then 
$$ U^{(j)} = U^{(j+1)} \stab_{U^{(j)}}(z),$$
whereas for generic $z\in X^{\l(\GG_m)}$ where $\l(\GG_m)$ acts on $L^*|_z$ with weight equal to $\weight_{r_j}$ then 
$$ \stab_{U^{(j)}}(z) \leqslant U^{(j+1)}.$$
Then we can quotient in stages, to obtain as before projective schemes
$X \wenv H = \widehat{X \times \PP^1} \env \hat{H}$ and $X \tenv H = \widetilde{X \times \PP^1} \env \hat{H}$ which are good (and in the case of $X\tenv H$ geometric) quotients by $\hat{H}$ of $\hat{H}$-invariant open subschemes of blow-ups of ${X} \times \AL^1$,  and each contain as an open subscheme a geometric quotient of an $H$-invariant open subscheme 
$$X^{M\hat{s}, H \geqslant U \geqslant U^{(1)} \geqslant \cdots \geqslant U^{(s)}}$$
 of $X$ by $H$. Here this Mumford-hat-stable locus $X^{M\hat{s}, H \geqslant U \geqslant U^{(1)} \geqslant \cdots \geqslant U^{(s)}}$  is the pre-image in $X^{M\hat{s},U \geqslant U^{(1)} \geqslant \cdots \geqslant U^{(s)}}$ under the quotient map to $X \wenv U = (\widehat{X \times \PP^1}) \env \hU$ of the Mumford stable locus $((\widehat{X \times \PP^1}) \env \hU)^{Ms,R}$ for the induced action of $R$ on $(\widehat{X \times \PP^1}) \env \hU$, and 
$X^{M\hat{s},U \geqslant U^{(1)} \geqslant \cdots \geqslant U^{(s)}}$
 is defined inductively as 
 the pre-image of 
 $$((X\times \PP^1)\env \hU^{(s-1)}) ^{M\hat{s}, U/U^{(s-1)} \geqslant U^{(1)}/U^{(s-1)} \geqslant \cdots \geqslant U^{(s-1)}/U^{(s-1)}}$$
 under the quotient map
 $$X^{M\hat{s},U^{(s-1)}} \to (X\times \PP^1)\env \hU^{(s-1)}.$$
 Here
 $X^{M\hat{s},U^{(s-1)}}$ is the locus in $X^0_{\min}$ where  the $\weight_r$-jet at 0 of 
 the morphism $\alpha_x: \AL^1 \to X$ defined by
$$\alpha_x (t) = \l(t) \exp(\xi/t^{s_j}) x \,\, \mbox{ for $t \neq 0$ and } \,\, \alpha_x (0) = \lim_{t\to 0} \l(t) \exp(\xi/t^{s_j}) x $$ 
  is nonzero for any weight vector $\xi \in \mathrm{Lie}(U^{(j)}) \setminus \mathrm{Lie}(U^{(j+1)})$ 
(see Remark \ref{rem20191}).


\begin{rmk}
Even if the algebra of $H$-invariant sections of powers of $L$ is finitely generated and condition \eqref{eq:sss} is satisfied, the projective schemes $X \wenv H = (X \times \PP^1)  \env \hat{H}$ and $X \env H$ are not in general isomorphic to each other, although they will be birationally equivalent. When $\kk = \CC$ and $H=U$ is a maximal unipotent subgroup of a reductive group $G$ (or more generally the unipotent radical of a parabolic subgroup of $G$) and the linear action of $U$ on $X$ extends to $G$, then the  algebra of $U$-invariant sections of powers of $L$ is finitely generated and the enveloping quotient $X \env U$ can be described using symplectic implosion \cite{GJS,KPEN}. When condition \eqref{eq:sss} is satisfied, there is a similar symplectic description of $X \wenv U = ({X \times \PP^1})  \env \hat{U}$, which is obtained from $X \env U$ via a symplectic cutting construction (cf.  \cite{Ler}).
\end{rmk}

\begin{rmk}
There is some analogy between the action of $\hat{H}$ on $X \times \PP^1$ (and on its blow-ups $\widehat{X \times \PP^1}$ and $\widetilde{X \times \PP^1}$ with GIT quotients  which are projective completions of $X^{\hat{s},H}/H$) and the concept introduced in \cite{DK} of a reductive envelope for a linear action of a unipotent group $U$ on a projective scheme $X$.  The main ingredients for a reductive envelope are a reductive group $G$ containing $U$ as a subgroup and a projective completion of $G \times^U X$ (where $G \times^U X$ is the quotient of $G \times X$ by the diagonal action of $U$ on the right on $G$ and on the left on $X$) with a linear $G$-action restricting to the given linearisation for the $U$-action on $X$.  Similarly, given a linear action of $H = U \rtimes R$ on $X$, we can define a graded envelope to be given by a semi-direct product $\hH = H \rtimes \l(\GG_m) = U \rtimes (R \times \l(\GG_m))$ with the adjoint action of $\l:\GG_m\to \hH$  on the Lie algebra of the unipotent radical $U$ having all weights strictly positive, and a projective  completion of $\hH \times^H X$ with a linear $\hH$-action restricting to the given linearisation for the $H$-action on $X$ (see \cite{BDKIII}. When the action of $H$ on $X$ extends to an action of $\hH$, then $\hH \times^H X$ is naturally isomorphic to $\GG_m \times X$, so the action of $\hat{H}$ on $X \times \PP^1$ gives us a graded envelope for the action of $H$ on $X$ in this sense.
\end{rmk}

\begin{rmk}  \label{remark9.4}
In the internally graded case when $H=U\rtimes R$ and $\l:\GG_m \to Z(R)$ grades $U$, we can identify  $X \wenv H = \widehat{X \times \PP^1} \env \hat{H}$ and $X \tenv H = \widetilde{X \times \PP^1} \env \hat{H}$, defined as after Definition \ref{defn9.1}, with
$$(\hat{X} \env \hat{U})\gitq (R/\l(\GG_m)) \,\,\, \mbox{ and } \,\,\, ( \tilde{X} \env  \hat{U})\gitq (R/\l(\GG_m))$$
 (cf. $\S$\ref{section2.2}).

\end{rmk}

\section{Examples} \label{section6}

In this section we consider some low dimensional examples which can be described very explicitly, and some applications to the construction of moduli spaces.

\subsection{Two points and a line in $\PP^2$}

First let $X = (\PP^2)^2 \times (\PP^2)^*$ with elements $(p,q,L)$ where $p$ and $q$ are points in $\PP^2$ and $L$ is a line in $\PP^2$. Let $X$ have the usual (left) action of the standard Borel subgroup $B = U \rtimes T$ of $\SL(3)$, consisting of upper triangular matrices, which is linear with respect to the product of $\calo(1)$ for each of the three projective planes whose product is $X$. We represent points in $\PP^2$ by column vectors, with the action of a matrix $A$ given by pre-multiplication by $A$; we represent lines
$$L_{(a,b,c)} = \{ [x:y:z] \in \PP^2 \,\, | \,\, ax + by + cz = 0 \}$$
by row vectors $(a,b,c)$ and the action of $A$ is given by post-multiplication by $A^{-1}$. The weights of the action of the maximal torus $T$ of $B$ on $X$ lie in an irregular hexagon (which is their convex hull) in the dual of the Lie algebra of the maximal compact subgroup of $T$. 

In order for a one-parameter subgroup $\lambda: \GG_m \to T$ to satisfy the condition that all its weights for the adjoint action on $\rm{Lie} \, U$ should be positive, we require its derivative to lie in the interior of the standard positive Weyl chamber $\liet_+$ for $\SL(3)$. The minimal weight for such a one-parameter subgroup acting on $X$ then corresponds to the $T$-fixed point 
$$ z_{\min} = ( [0,0,1], [0,0,1], L_{(1,0,0)}) $$
where the points $p$ and $q$ coincide at $[0,0,1]$ and both lie on the line $L$ which is defined by $x=0$. The stabiliser in $U$ of $z_{\min}$ is trivial, and the closure of its $U$-orbit is 
$$ \{ (p,q,L) \in X \,\, | \,\, p=q \in L \},$$
while
$$X^0_{\min} = \{ ([x_1,y_1,z_1], [x_2,y_2,z_2], L_{(a,b,c)}) \in X\,\, | \,\, z_1 \neq 0 \neq z_2, a \neq 0 \}$$
and by  Theorem \ref{thm:ExCo1}
$$X^0_{\min} \setminus Uz_{\min} = X^0_{\min} \setminus \overline{U z_{\min}} $$
has a projective geometric quotient $X \env B = (X^0_{\min} \setminus Uz_{\min})/B$ when the linearisation is twisted by a rational character $\chi$ in the interior of the hexagon near to the $T$-weight for $z_{\min}$.

\begin{rmk}
Here $U$ is a maximal unipotent subgroup of the reductive group $G = \SL(3)$ and the linear action of $U$ on $X$ extends to $G$. Therefore we know that the $U$-invariants on $X$ are finitely generated and the projective scheme $X \env U$ can be described geometrically over $\CC$ as the symplectic implosion (in the sense of \cite{GJS}, cf. also \cite{KPEN}) of $X$ by the action of the maximal compact subgroup $K = \SU(3)$ of $G$. Thus if $\mu: X \to \rm{Lie} \SU(3)^*$ is the moment map for the action of $K=\SU(3)$ on $X$ associated to the linearisation, and $K_\xi$ is the stabiliser of $\xi \in \liet_+$ under the coadjoint action of  $K$, then $X \env U$ can be obtained from $\mu^{-1}( \liet_+)$ by collapsing on the boundary of $\liet_+$ via the equivalence relation given by $x \sim y$ if $\mu(x) = \mu(y) = \xi \in \liet_+$ and $x \in [K_\xi, K_\xi] y$. For compatibility with the conventions in this paper, we should replace $\mu^{-1}(\liet_+)$ with $\mu^{-1}(-\liet_+)$. Quotienting further by $T$ to obtain $X \env B$ gives us 
$$ \mu^{-1}( \chi) / (T \cap K)$$
for $ \chi \in - \liet_+^o$, where $T \cap K$ is the usual maximal torus of $K= \SU(3)$ consisting of diagonal matrices.  This fits with our previous description since 
$$X^0_{\min} \setminus Uz_{\min} = B \times_{T \cap K} \mu^{-1}(\chi).$$
\end{rmk}

In order to find an example requiring a blow-up to achieve condition \eqref{eq:sss}, let us consider the subgroup
$$ U' = \left\{  \left( \begin{array} {ccc} 1 & \alpha & \beta \\ 0 & 1 & \gamma \\ 0 & 0 & 1 \end{array} \right) \in U \quad  | \quad  \gamma = 0 \, \, \right\} $$
of $U$, acting on the subscheme
$$Y = (\PP^2)^3 \times (\PP^1)^* $$
of $(\PP^2)^3 \times (\PP^2)^* $, where $(\PP^1)^* = \{ L_{(a,b,c)} \in (\PP^2)^* \,\, | \,\, a=0 \}$. Here $U'$ is normalised by the subgroup
$\lambda': \GG_m \to U'$ given by
$$ \lambda'(t) =  \left( \begin{array} {ccc} t^2 & 0 & 0 \\ 0 & t^{-1} & 0 \\ 0 & 0 & t^{-1} \end{array} \right) $$
of $T$, with the weights of $\lambda'(\GG_m)$ on the Lie algebra of $U'$ both strictly positive, and the linear action of $U'$ on $Y$ extends to $\hU' = U' \rtimes_{\lambda'} \GG_m$. The subset $Z_{\min}$ of $Y^{\lambda'(\GG_m)}$ where the $\GG_m$-weight is minimal is 
$$Z_{\min} = \{ (p,q,r,L) \in Y\,\, | \,\, p,q,r  \in L_{(1,0,0)} \}$$
corresponding to the points $p$, $q$ and $r$ all lying in the line defined by $x=0$. We have 
$$  \left( \begin{array} {ccc} 1 & \alpha & \beta \\ 0 & 1 & 0 \\ 0 & 0 & 1 \end{array} \right)  \left( \begin{array} {c} 0 \\ y \\ z \end{array} \right) =  \left( \begin{array} {c} \alpha y + \beta z \\ y \\ z \end{array} \right) $$
and 
$$ (0 \, b \, c) \left( \begin{array} {ccc} 1 & \alpha & \beta \\ 0 & 1 & 0 \\ 0 & 0 & 1 \end{array} \right) = (0 \, b \, c), $$
so the stabiliser in $U'$ of $(p,q,r,L) \in Z_{\min} $ is trivial unless $p = q = r$ in which case it is one-dimensional. Thus $d^U_{\max} = 1$ and 
$$ Z_{\min}^{d_{\max}} = \{ (p,q,r,L) \in Y\,\, | \,\, p=q=r \in L_{(1,0,0)} \},$$
while
$$ \Delta_U^{\geqslant d^U_{\max}}  = U' Z_{\min}^{d_{\max}} = \{ (p,q,r,L) \in Y \,\, | \,\, p=q=r \neq [1,0,0] \}.$$
The blow-up $\hat{Y}$ of $Y$ along the closure
$$\overline{\Delta_U^{\geqslant d^U_{\max}} } = \PP^2 \times (\PP^1)^* $$ of
$\Delta_U^{\geqslant d^U_{\max}} $ in $Y$
 is the product of $(\PP^1)^*$ with $(\PP^2)^3$ blown up along its diagonal $\PP^2$, with $\hat{Z}_{\min}$ the proper transform of $Z_{\min} = (\PP^1)^3 \times (\PP^1)^*$. So $\hat{Z}_{\min}$ is the blow-up of $Z_{\min} = (\PP^1)^3 \times (\PP^1)^*$ along $ \PP^1 \times (\PP^1)^*$ embedded diagonally, while the exceptional divisor $E$ is a $\PP^3$-bundle over $ \PP^2 \times (\PP^1)^*$. The open subset
$$ Y^s =  \{ (p,q,L) \in Y \,\, | \,\, p,q,r  \neq [1,0,0] \mbox{ and $p,q,r$ do not all coincide}  \} $$
has a geometric quotient by $\hU$ which has a projective completion
$$ \hat{Y} \env \hU = Y^s/\hU \,\, \sqcup \,\, E \env \hU.$$

\subsection{Moduli spaces of (Higgs) bundles of fixed Harder--Narasimhan type over a nonsingular projective curve} \label{subsec:rk2bundles}

When $G$ is a reductive group over $\kk$, acting linearly on a projective scheme $X$ with respect to an ample line bundle $L$, then given an invariant inner product on the Lie algebra of $G$, there is a stratification 
$$ X = \bigsqcup_{\beta \in \mathcal{B}} S_\beta$$ of $X$ by locally closed subschemes $S_\beta$, 
indexed by a partially ordered finite subset $\mathcal{B}$ of a positive Weyl chamber for the reductive group $G$,  such that 
 $S_0 = X^{ss}$, 
 and for each $\beta \in \mathcal{B}$
 the closure of $S_\beta$ is contained in $\bigcup_{\gamma \geqslant \beta} S_\gamma$. Moreover  $S_\beta \cong G \times_{P_\beta} Y_\beta^{ss}$
 where $$\gamma \geqslant \beta \mbox{ if and only if } \gamma = \beta \mbox{ or } |\!|\gamma|\!| > |\!|\beta|\!|$$
and 
$P_\beta$ is a parabolic subgroup of $G$ acting on  a projective subscheme $\overline{Y}_\beta$ of $X$ with an open subset $Y_\beta^{ss}$ which is determined by the action of the Levi subgroup of $P_\beta$ with respect to a suitably twisted linearisation \cite{K}.
Here the original linearisation for the action of $G$ on $L \to X$ is restricted to the action of the parabolic subgroup $P_\beta$ over $\overline{Y}_\beta$, and then twisted by a rational character of $P_\beta$ which is well adapted  for a central one-parameter subgroup of the Levi subgroup of $P_\beta$ acting with all weights strictly positive on the Lie algebra of the unipotent radical of $P_\beta$.
Thus to construct a quotient by $G$ of (an open subset of) any unstable stratum $S_\beta$ (and thus to study the stack $[X/G]$ via this stratification), we can study the linear action on $\overline{Y}_\beta$ of the parabolic subgroup $P_\beta$, twisted appropriately, and apply the results of this paper.

In particular we can consider moduli spaces of sheaves of fixed Harder--Narasimhan type over a nonsingular projective scheme $W$ (cf.  \cite{hok12}). 
Moduli spaces of semistable pure sheaves on $W$ of fixed  Hilbert polynomial can be constructed as GIT
quotients of linear actions of suitable  special linear groups $G$ on  schemes $Q$ (closely related to 
quot-schemes) which are $G$-equivariantly embedded in projective spaces  \cite{s94}.
 These constructions can
be chosen so that elements of $Q$ which parametrise sheaves of a fixed Harder--Narasimhan
type  form a stratum in the stratification of $Q$ associated to the linear action of $G$ (at least modulo
taking connected components of strata) \cite{hok12}. Thus to construct and study moduli spaces of  sheaves
of fixed Harder--Narasimhan type over $W$  we can study  the associated linear actions of  parabolic subgroups of these special linear groups $G$, appropriately twisted  
 \cite{BHJK}. However in these cases the condition that $\stab_U(x) = \{ e \}$ for generic $x$ is rarely satisfied; from this viewpoint Higgs bundles are easier to deal with \cite{hamilton, hamilton2}.
The simplest non-trivial case is that of unstable vector bundles of rank 2 and fixed Harder--Narasimhan type over a nonsingular projective curve $W$ (cf. \cite{bramn09,Jackson}); then 
the blow-up construction terminates  with the situation that for every $x \in X^0_{\min}$ the dimension of $\stab_U(x)$ is equal to the generic dimension $d^U_{\min}$, and moreover, since $U$ is commutative, each stabiliser $\stab_U(x)$ has a complementary subgroup $U'$ in $U$ (cf. Remarks \ref{remweaken3} and \ref{remweaken5}).

\section{Cohomology of non-reductive GIT quotients}\label{sec:applications}

The results of \cite{bkcoh,hamilton2}
allow us to study cohomological intersection theory on moduli spaces constructed as quotients of projective varieties by non-reductive linear algebraic group actions as in this paper. When a moduli space can be described as the stable locus $X^{s,H}$ in a projective parametrising space $X$ up to a symmetry group $H=U \rtimes R$ with internally graded unipotent radical $U$, the non-reductive GIT quotient $\tilde{X}/\!/H$ of a suitable $H$-equivariant blow-up $\tilde{X}$ of $X$ gives a projective completion of the moduli space. When $X^{s,H}$ is nonsingular $\tilde{X}$ can be chosen to be nonsingular with the induced action of $H$ on $\tilde{X}$ well adapted and $H$-semistability coinciding with $H$-stability, and then the rational cohomology and cohomological intersection numbers of this projective completion can be directly computed from the $T$-action on the normal bundles to the connected components of the $T$-fixed point set, where $T \subset R$ is a maximal torus in $H$ in a way which generalises the reductive situation (cf. \cite{SM}). 

This is being used to study the cohomology of moduli spaces of unstable Higgs bundles of low rank constructed as non-reductive GIT quotients \cite{hamilton2,hamilton}, and has been used recently to give a proof of the polynomial Green--Griffiths--Lang and  Kobayashi conjectures \cite{bkGGL}.
A projective variety $X$ is called Brody hyperbolic if there is no non-constant entire holomorphic curve in $X$; i.e. any holomorphic map $f: \CC \to X$ must be constant. 
\begin{conjecture}[\textbf{Kobayashi conjecture} \cite{kob2}]
A generic hypersurface $X\subseteq \PP^{n+1}$ of degree $d_n$ is Brody hyperbolic if $d_n$ is sufficiently large.
\end{conjecture}

\begin{conjecture}[\textbf{Green-Griffiths-Lang conjecture} \cite{gg,lang}]
Any projective algebraic variety $X$ of general type contains a proper algebraic subvariety $Y\subsetneqq X$ such that every
nonconstant entire holomorphic curve $f:\CC \to X$ satisfies $f(\CC) \subseteq Y$. 
\end{conjecture}

These conjectures are related (a generic projective hypersurface $X\subseteq \PP^{n+1}$ is of general type if $\deg(X)\ge n+3$), and have a long history \cite{demsurvey,dr}.  Siu \cite{siu4} and Brotbek \cite{brotbek} proved  Kobayashi hyperbolicity for projective hypersurfaces of sufficiently high (but not effective) degree. Effective degree bounds were worked out by Deng \cite{deng}, Demailly \cite{demsurvey}, Merker and The-Anh Ta \cite{merker2}; the best known bound based on their techniques was $(n \log n)^n$.     
The GGL conjecture was proved for surfaces by McQuillan \cite{mcquillan} assuming $c^2_1-c_2>0$. Following the strategy developed by Demailly \cite{dem} and Siu \cite{siu1,siu2,siu3,siu4}, 
 the first effective lower bound  in this conjecture was given by Diverio, Merker and Rousseau \cite{dmr}, where the conjecture for generic projective hypersurfaces $X\subseteq \PP^{n+1}$ of degree $\deg(X)>2^{n^5}$ was confirmed; later $\deg(X)>(\sqrt{n}\log n)^n$ was shown to be sufficient \cite{merker2}. 
Recently Riedl and Yang \cite{riedl} showed that if there are integers $d_n$ for all positive $n$ such that the GGL conjecture for hypersurfaces of dimension $n$ holds for degree at least $d_n$ then the Kobayashi conjecture is true for hypersurfaces with degree at least $d_{2n-1}$.

To approach these conjectures Demailly and others study the bundle $\pi_k:J_kX \to X$ of holomorphic $k$-jets of maps $f: (\CC,0) \to X$ with $\pi_k(f)=f(0)$. 
 Polynomial reparametrisations $(\CC,0)\to (\CC,0)$ of order $k$ with nonzero first derivative form a group $\Diff_k$, which acts on $J_kX$ by reparametrisation of jets.  The quotient $J_k^{reg} X/\Diff_k$ is the moduli space of regular $k$-jets in $X$, where regularity means the non-vanishing of the first derivative. 
The reparametrisation group $\Diff_k$ is not reductive. 
In \cite{bkGGL} we use the non-reductive GIT quotient $J_kX \widehat{\env} \Diff_k(1)$ and  the cohomological intersection theory developed in \cite{bkcoh} and the result of Riedl and Yang \cite{riedl} to prove the hyperbolicity conjectures for generic hypersurfaces with polynomial degree bounds.

\begin{thm}[\textbf{Polynomial Green-Griffiths-Lang theorem \cite{bkGGL}}] \label{mainthmtwo}
Let $X\subseteq \PP^{n+1}$ be a generic smooth projective hypersurface
of degree $\deg(X)\ge 16n^5(5n+4)$. Then there is a proper algebraic subvariety $Y\subsetneqq X$ containing all nonconstant entire holomorphic curves in $X$. 
\end{thm}

\begin{cor}[\textbf{Polynomial Kobayashi theorem \cite{bkGGL}}] \label{mainthmtwob}
A generic smooth projective hypersurface $X\subseteq \PP^{n+1}$  
of degree $\deg(X)\ge 16(2n-1)^5(10n-1)$ is Brody hyperbolic. 
\end{cor}  

The moduli space of $k$-jets in a nonsingular projective variety $X$ also plays a crucial role in enumerative geometry, and so potentially cohomology intersection theory on non-reductive GIT quotients will give new formulas here. Examples include \cite{b0} where it is shown that 
$J_k^{reg}X/\Diff_k$ forms a dense open subset of the curvilinear component $\CHilb^{k+1}(X)$ of the Hilbert scheme $\Hilb^{k+1}(X)$ of $k+1$ points in $X$,  \cite{bsz2} where a formula is obtained connecting integrals on $\CHilb^{k+1}(X)$ with integrals on $\Hilb^k(X)$,
\cite{bsz,bsz3} where Thom polynomials of singularities for maps $f:M\to N$ between compact manifolds are studied as equivariant intersection numbers of $\CHilb^k(M)$ and multisingularity classes are developed, and also \cite{gottsche,lehn,mop1,mop2}.


\newpage

\hspace{3em}
\begin{center}
\begin{large} \textbf{PROJECTIVE COMPLETIONS FOR INTERNALLY GRADED QUOTIENTS} \\ APPENDIX by Eloise Hamilton \end{large}
\end{center}
\hspace{3em}
  
Appendix \ref{sec:diagrams} provides diagrams depicting the constructions of this paper and Appendix \ref{sec:algorithm} provides an explicit version of a resulting algorithm used in \cite{BHK} to find non-empty geometric quotients with projective completions for linear actions of linear algebraic groups with internally graded unipotent radicals.   
  
\appendix
  
\section{Diagrams for Geometric Invariant Theory, reductive and non-reductive} \label{sec:diagrams}

In this section we summarise the main results from Geometric Invariant Theory (GIT), both in the classical setting of \cite{GIT} ($\S$\ref{subsec:classicalredgit}) and in the non-reductive setting of the present paper ($\S$\ref{subsec:gradedunip} and $\S$\ref{subsec:internallygraded}), providing diagrams to illustrate the constructions involved. 

\subsection{Classical GIT for reductive groups} \label{subsec:classicalredgit}

\begin{figure}[b]
    \centering
    \begin{subfigure}{0.45\textwidth}
\centering
        \begin{tikzpicture}[scale=1,x=(15:8cm),y=(0:4.5cm), z=(90:1.5cm),
    cross line/.style={preaction={draw=white,-,line width=3pt}},
    equal/.style={double distance=1pt}]

\node (c1) at (1,-1.55,0.6) {  } ;
\node (c2) at (1,-1.55,2.4) { };
\node (c3) at (1,0.25,0.6) { };
\node (c4) at (1,0.25,2.4) { };

\filldraw  [fill = yellow, draw=black,loosely dotted, opacity=0.2] (c1) rectangle (c4);

\node (16) at (1,0,2) {\(   X \gitq G  \)};
\node(19) at (1,-1,2) {\( X \)};
\node(19') at (1,0,1) {\( X^s / G \)};
\node(20) at (1,-1,1) {\(  X^{ss} = X^{s} = \bigcap_{g \in G} g X^{s,T}    \)};

\draw [black,loosely dotted] (c1) rectangle (c4);
\draw [dashed, ->] (19) -- (16);
\draw[ ->] (20) -- (19');
\draw [equal] (19') -- (16);
\draw [right hook ->] (20) -- (19);

		\end{tikzpicture}
        \caption{When $X^{ss} = X^{s} \neq \emptyset$} \label{fig:classicalgit1} 
	\end{subfigure}

	\begin{subfigure}{1\textwidth}
        \centering
		\begin{tikzpicture}[scale=1,x=(15:6.5cm),y=(0:4cm), z=(90:1.8cm),
    cross line/.style={preaction={draw=white,-,line width=3pt}},
    equal/.style={double distance=1pt}]
    

\node (c1) at (1,-1.3,-0.25) {  } ;
\node (c2) at (1,-1.3,1.6) { };
\node (c3) at (1,0.35,-0.25) { };
\node (c4) at (1,0.35,1.6) { };
\node (c5) at (0,-1.3,-0.25) { } ;
\node (c6) at (0,-1.3,1.6) { } ;
\node (c7) at (0,0.35,-0.25) { } ;
\node (c8) at (0,0.35,1.6) { } ;

\node (1) at (0,-1,1.4) {\( X \)};
\node (1') at (0,0,1.4) {\( X \gitq G \)};
\node (2) at (0,-1,0.7) {\( X^{ss} \)};
\node (2') at (0,-1,0) {\( X^{s}  \)};
\node (3) at (0,0,0) {\( X^{s} / G \)};
\node (16) at (1,0,1.4) {\(   \tX \gitq G  \)};
\node(19) at (1,-1,1.4) {\( \tX \)};
\node(19') at (1,0,0.7) {\( \tX^{s} / G \)};
\node(20) at (1,-1,0.7) {\(  \tX^{ss} = \tX^{s}  \)};
\node(21) at (1,-1,0) {\( \tX^{ss} \setminus \widetilde{E} \)}; 
\node(22) at (1,0,0) {\( \left(\tX^{ss} \setminus \widetilde{E} \right) / G \)}; 
     
\filldraw  [fill = orange, draw=black,loosely dotted, opacity=0.2] (c5) rectangle (c8);
\filldraw  [fill = yellow, draw=black,loosely dotted, opacity=0.2] (c1) rectangle (c4);
  
\draw [black,loosely dotted] (c1) rectangle (c4);
\draw  [black,loosely dotted] (c5) rectangle (c8);
\draw [black, loosely dotted] (c3) -- (c7);
\draw [black, loosely dotted] (c4) -- (c8);
\draw [black, loosely dotted] (c1) -- (c5); 
\draw [black, loosely dotted] (c2) -- (c6);
     
\draw [->, dashed] (1) -- (1');
\draw [right hook ->](2) -- (1);
\draw [->] (2) -- (1');
\draw [right hook ->] (2') -- (2);
\draw [ ->](2') -- (3);
\draw [->] (21) -- (2') node[pos=0.5,below] {$\cong$}; 
\draw [right hook ->] (21) -- (20);
\draw [->] (21) -- (22) ;
\draw [->] (22) -- (3) node[pos=0.6,below] {$\cong$} ;
\draw [right hook -> ] (22) -- (19');
\draw [dashed, ->] (19) -- (16);
\draw[->] (20) -- (19');
\draw [equal] (19') -- (16);
\draw [right hook ->] (20) -- (19);
\draw [->] (19) -- (1) node[pos=0.4,above] {$\widetilde{\pi}$} ;
\draw [->, cross line] (16) -- (1') node[pos=0.4,above] {$\widetilde{\pi}_G$} ;    
\draw [ForestGreen,line width=0.5mm, right hook ->, cross line] (3) -- (16);
\draw [ForestGreen, line width=0.5mm,right hook ->, cross line ] (3) -- (1');
    
		\end{tikzpicture}
		
		\caption{When $\emptyset \neq X^s \subsetneq X^{ss}$. The green arrows denote projective completions.} \label{fig:classicalgit2}

	\end{subfigure}

	\caption{Classical GIT for reductive groups} \label{fig:classicalgit} 
\end{figure}

Let $X$ be an irreducible projective variety over an algebraically closed field of characteristic $0$ and let $G$ be a reductive group acting linearly on $X$ with respect to an ample line bundle $L \to X$. According to classical GIT, there exist open subsets $X^{s} \subseteq X^{ss} \subseteq X$ and a projective variety $X \gitq G = \operatorname{Proj} \bigoplus_{k \geq 0} H^0(X,L^{\otimes k })^G$ together with a rational map $q_G: X \dashrightarrow X \gitq G$ such that $q_G$ restricts to a good quotient $X^{ss} \to X \gitq G$ and induces a geometric quotient $X^{s} \to q_G(X^s) = X^s/G$. In the terminology of \cite{GIT}, $X^s$ is the properly stable locus; we will use the notation $X^{Ms}$ to denote the stable locus in Mumford's sense. When $X^s$ is dense in $X$, the GIT quotient $X \gitq G$ can be viewed as a projective completion of the geometric quotient $X^{s}/G$. Moreover, the semistable and stable loci are characterised by Hilbert-Mumford criteria so that $$X^{(s)s} = \bigcap_{g \in G} g X^{(s)s,T}$$ for a fixed maximal torus $T \subseteq G$. In what follows, if $X^s$ is empty but $X^{Ms}$ is not, then stability should be replaced by Mumford-stability. Note however that, unlike the stable locus, the Mumford-stable locus is not described by a Hilbert-Mumford-type criterion. 

If $\emptyset \neq X^{s} = X^{ss}$, then the GIT quotient $X \gitq G$ is a projective geometric quotient for the stable locus of $X$, and hence can have at worst only finite quotient singularities. Figure \ref{fig:classicalgit1} illustrates this situation.  

If $\emptyset \neq X^s \subsetneq X^{ss}$, then by \cite{K2,Reichstein1989} there exists a sequence of blow-ups starting from $X$ and resulting in a projective variety $\widetilde{X}$ with a blow-down map $\widetilde{\pi}: \widetilde{X} \to X$ such that $\tX^{ss} = \tX^{s} \neq \emptyset$ for a suitably linearised action of $G$ on $\tX$ (see Section \ref{section4} of the present paper for more detail). By construction of the blow-ups, $X^{s} = \widetilde{\pi}(\tX^{s} \setminus \widetilde{E})$ where $\widetilde{E}$ denotes the exceptional divisor of the blow-down map $\widetilde{\pi}$. 
Figure \ref{fig:classicalgit2} illustrates this blowing up construction. The two green arrows denote the two projective completions of $X^s/G$, namely $X \gitq G$ and $\tX \gitq G$, where the latter can be viewed as a partial desingularisation of the former. 

If $\emptyset = X^s 
\neq X^{ss}$, then although the blow-ups can again be performed, for the resulting variety $\tX$ the semistable locus may be empty (see Remark \ref{rem:specialcases} of the present paper). In general, if the semistable locus for the linearised action of a reductive group $G$ on a projective variety $X$ is empty, then we are not in a situation where classical GIT can be usefully applied. This is one of the motivations for considering Non-Reductive GIT. Indeed, given an invariant inner product on the Lie algebra of $G$, by the HKKN-stratification for $X$ we have: $$X = \bigsqcup_{\b \in \mathcal{B}} S_{\b},$$ where $S_{0} = X^{ss}$ and $S_{\b} = G Y_{\beta}^{ss} \cong G \times_{P_{\b}} Y_{\b}^{ss}$ for $\beta \neq 0$ (see Remark \ref{stratification} of the present paper and \cite[\S 1]{BHK} for more details). Note that when $G$ is connected, $S_{\beta}$ is irreducible if and only if $Y_{\beta}^{ss}$ is irreducible. If $X^{ss} = \emptyset$, then there must exist an unstable stratum $S_{\b}$ with $\b \neq 0$ such that $S_{\b} \cong G \times_{P_{\b}} Y_{\b}^{ss}$ is open in $X$. Taking a quotient of an invariant open subset of $S_{\b}$ by $G$ is thus equivalent to taking a quotient of an invariant open subset of $Y_{\b}^{ss}$ by the parabolic subgroup $P_{\b}$, which is non-reductive in general. However, the subgroup $P_{\b}$ satisfies the property that its unipotent radical is internally graded. That is, $P_{\b} = U_{\b} \rtimes L_{\b}$, where $U_{\b}$ is its unipotent radical and $L_{\b}$ a Levi subgroup, and there exists a central one-parameter subgroup $\l_{\b}: \GG_m \to L_{\b}$ such that $\l_{\b}(\GG_m)$ acts on $\Lie U_{\b}$ with strictly positive weights. Linear algebraic groups satisfying this property are said to have internally graded unipotent radicals. Results from the present paper show that GIT can be successfully extended to allow actions by such groups. In the following two subsections we summarise the results from the present paper for graded extensions of unipotent groups (Section \ref{subsec:gradedunip}) and for linear algebraic groups with internally graded unipotent radicals (Section \ref{subsec:internallygraded}), providing diagrams to illustrate them.

\subsection{GIT for graded extensions of unipotent groups}  \label{subsec:gradedunip}

\begin{figure}
\begin{subfigure}{0.55\textwidth}
       \centering
       \begin{tikzpicture}[scale=1,x=(15:8cm),y=(0:3.9cm), z=(90:1.5cm),
    cross line/.style={preaction={draw=white,-,line width=3pt}},
    equal/.style={double distance=1pt}]
  
\node (c1) at (1,-1.75,0.6) {  } ;
\node (c2) at (1,-1.75,2.35) { };
\node (c3) at (1,0.35,0.6) { };
\node (c4) at (1,0.35,2.35) { };

\node (16) at (1,0,2) {\(   X \gitq \hU \)};
\node(19) at (1,-1,2) {\( X \)};
\node(20) at (1,-1,1) {\(  X^{ss, \hU}_{\min +}  = X^{s, \hU}_{\min +} = \bigcap_{u \in U} u X^{s, \GG_m}_{\min +} \)};
\node (19') at (1,0,1) {\( X^{s,\hU}_{\min +} / \hU \)};
     
\filldraw  [fill = yellow, draw=black,loosely dotted, opacity=0.2] (c1) rectangle (c4);
  
\draw [black,loosely dotted] (c1) rectangle (c4);
\draw [equal] (16) -- (19');
\draw [dashed, ->] (19) -- (16);
\draw[right hook ->] (20) -- (19);
\draw [->] (20) -- (19');
     
		\end{tikzpicture}
        \caption{When (ss=s($\hU$)) holds and $X \neq \overline{U \zmin}$} \label{fig:unip1}
    \end{subfigure} \begin{subfigure}{0.45\textwidth}
\centering
       \begin{tikzpicture}[scale=1,x=(15:8cm),y=(0:3.5cm), z=(90:1.5cm),
    cross line/.style={preaction={draw=white,-,line width=3pt}},
    equal/.style={double distance=1pt}]
    
\node (c1) at (1,-1.35,0.6) {  } ;
\node (c2) at (1,-1.35,2.35) { };
\node (c3) at (1,0.35,0.6) { };
\node (c4) at (1,0.35,2.35) { };

\node (16) at (1,0,2) {\(   \zmin \)};
\node(19) at (1,-1,2) {\( X = \overline{U \zmin} \)};
\node(20) at (1,-1,1) {\(  U \zmin \)};
\node (19') at (1,0,1) {\( U \zmin / \hU \)};
     
\filldraw  [fill = yellow, draw=black,loosely dotted, opacity=0.2] (c1) rectangle (c4);
  
\draw [black,loosely dotted] (c1) rectangle (c4);
\draw [equal] (16) -- (19');
\draw [dashed, ->] (19) -- (16);
\draw[right hook ->] (20) -- (19);
\draw [->] (20) -- (19');
     
		\end{tikzpicture}
        \caption{When (ss=s($\hU$)) holds and $X = \overline{U \zmin}$} \label{fig:unip12}
    \end{subfigure}

    \begin{subfigure}{1\textwidth}
  \centering
       \begin{tikzpicture}[scale=1,x=(15:7cm),y=(0:4.3cm), z=(90:1.2cm),
    cross line/.style={preaction={draw=white,-,line width=3pt}},
    equal/.style={double distance=1pt}]
    
\node (c1) at (1,-1.3,-0.4) {  } ;
\node (c2) at (1,-1.3,2.4) { };
\node (c3) at (1,0.35,-0.4) { };
\node (c4) at (1,0.35,2.4) { };
\node (c5) at (0,-1.3,-0.4) { } ;
\node (c6) at (0,-1.3,2.4) { } ;
\node (c7) at (0,0.35,-0.4) { } ;
\node (c8) at (0,0.35,2.4) { } ;

\node (1) at (0,-1,2) {\( X \)};
\node (2') at (0,-1,0) {\( X^{\widehat{s}, \hU}  \)};
\node (3) at (0,0,0) {\( X^{\widehat{s}} / 
\hU \)};
\node (16) at (1,0,2) {\(   \hX \gitq \hU \)};
\node(19) at (1,-1,2) {\( \hX \)};
\node(20) at (1,-1,1) {\(   \hX^{s, \hU}_{\min +}   \)};
\node (19') at (1,0,1) {\( \hX^{s, \hU}_{\min +}  / \hU \)};
\node (21) at (1,-1,0) {\( \hX^{s, \hU}_{\min +} \setminus \widehat{E} \)}; 
\node (22) at (1,0,0) {\( \left(\hX^{s, \hU}_{\min +} \setminus \widehat{E} \right) / \hU \)}; 
     
\filldraw  [fill = red, draw=black,loosely dotted, opacity=0.2] (c5) rectangle (c8);
\filldraw  [fill = yellow, draw=black,loosely dotted, opacity=0.2] (c1) rectangle (c4);
  
\draw [black,loosely dotted] (c1) rectangle (c4);
\draw  [black,loosely dotted] (c5) rectangle (c8);
\draw [black, loosely dotted] (c3) -- (c7);
\draw [black, loosely dotted] (c4) -- (c8);
\draw [black, loosely dotted] (c1) -- (c5); 
\draw [black, loosely dotted] (c2) -- (c6);

\draw [equal] (16) -- (19');
\draw [right hook ->] (2') -- (1);
\draw [dashed, ->] (19) -- (16);
\draw[right hook ->] (20) -- (19);
\draw [->] (20) -- (19');
\draw [right hook ->] (21) -- (20);
\draw [->] (21) -- (22);
\draw [right hook ->] (22) -- (19') ;
\draw  [->] (21) -- (2') node[midway,above] {$\widehat{\pi}$} node[midway, below] {$\cong$};
\draw [->] (22) -- (3) node[midway,above] {$\widehat{\pi}_{\hU}$} node[midway, below] {$\cong$}; 
\draw [->] (2') -- (3);
\draw [->] (19) -- (1) node[midway,above] {$\widehat{\pi}$} ;
      \draw [ForestGreen,line width=0.5mm,right hook ->, cross line] (3) -- (16);
   		\end{tikzpicture} \caption{When (ss=s$\neq \emptyset$[$U$]) does not hold. The green arrow denotes a projective completion.} \label{fig:unip2}
       \end{subfigure}  
  \caption{GIT for graded extensions of unipotent groups}
\end{figure}

Let $\hU = U \rtimes \GG_m$ be the semi-direct product of a non-trivial unipotent group with the multiplicative group $\GG_m$ such that $\GG_m$ acts on $\Lie U$ with strictly positive weights. Suppose that $\hU$ acts linearly on a projective variety $X$ with ample line bundle $L$. The subvarieties $\zmin$ and $\xmino$ are defined as in Definition \ref{def:ExPr2} of the present paper. 

Since $\hU$ is non-reductive, the ring of invariants $\bigoplus_{k \geq 0} H^0(X,L^{\otimes k})^{\hU}$ may not be finitely generated and thus there is no obvious analogue of the GIT quotient. Instead, an `enveloping quotient' $X \env \hU$ can be defined, admitting a rational map $q_{\hU}: X \dashrightarrow X \env \hU$ which restricts to a morphism on a suitably defined semistable locus $X^{ss,\hU}$. Unlike the classical case, this morphism may not be surjective, and its image may not even be a subvariety of $X \env \hU$ but only a constructible subset. Nevertheless, a stable locus $X^{s,\hU}$ can be defined so that $q_{\hU}(X^{s,\hU})$ is a subvariety of $X \env \hU$ and a geometric quotient for the action of $\hU$ on $X$ (see Section \ref{section2} of the present paper for more detail). By contrast with classical GIT, the semistable and stable loci for the $\hU$ action cannot in general be described by Hilbert-Mumford-type criteria. 

Nevertheless, if the action satisfies an additional condition denoted (ss=s($U$)), namely that $\Stab_U (z)= \{e\}$ for all $z \in \zmin$, then by Theorems \ref{mainthm} and \ref{thm:ExCo1} of the present paper, after taking a tensor power of the linearisation and twisting it by an appropriate character, all of the properties of classical GIT in the case when semistability coincides with stability can be recovered. That is, the invariants form a finitely generated graded ring and $X \env \hU$ is the associated projective variety. Moreover, a stable locus $$\xminshu := \xmino \setminus U \zmin$$ can be defined such that, if non-empty, $q_{\hU}: \xminshu \to X \env \hU$ is surjective and a geometric quotient for the action of $\hU$ on $\xminshu$. Finally, the stable locus  $\xminshu$ is characterised by a Hilbert-Mumford-type criterion: $$\xminshu = \bigcap_{u \in U} u \xminsg,$$ where $\xminsg$ is the stable locus in the reductive sense for the action of $\l(\GG_m)$ on $X$ with respect to the tensored and twisted linearisation. Since the enveloping quotient $X \env \hU$ satisfies all of the properties of classical GIT in the case where semistability coincides with stability, we denote it by $X \gitq \hU$ instead. Figure \ref{fig:unip1} illustrates this situation. Note that if $X_{\min+}^{s,\hU} = \emptyset$, then $U \zmin$ must be open in $X$. In this case, $\zmin$ is a geometric quotient for the action of $\hU$ on $U \zmin$ (see Remark \ref{rmkER} of the present paper). Figure \ref{fig:unip12} illustrates this situation.

If $\Stab_U(x)$ is not trivial but of constant and strictly positive dimension for all $x \in \xmino$, and the same is true for subgroups $U^{(j)}$ appearing in a series $U \supseteq U^{(1)} \supseteq \cdots \supseteq U^{(s)} = \{e\}$ normalised by $H$ with $U^{(j)}/ U^{(j+1)}$ abelian, then the same conclusions as stated in the above paragraph hold (see Remark \ref{remweaken} of the present paper). This situation is analogous to the reductive case when $X^s = \emptyset$ but $\emptyset \neq X^{Ms} = X^{ss}$. For this reason we denote the above weaker condition for the $U$-stabiliser groups in $\xmino$ by (ss=Ms$\neq \emptyset$[$U$]).

If the condition (ss=s$\neq \emptyset$[$U$]) is not satisfied for the action of $\hU$ on $X$, then provided there is some $x \in X$ such that $\Stab_U(x) = \{e\}$ (this condition is analogous to the condition that $X^s \neq \emptyset$ in the reductive case), by Theorem \ref{mainthm2} of the present paper there is a sequence of $\hU$-equivariant blow-ups of $X$ resulting in a projective variety $\hX$ with a blow-down map $\widehat{\pi}: \hX \to X $ such that for a suitably defined linearised action of $\hU$ on $\hX$ the condition (ss=s$\neq \emptyset$[$U$]) is satisfied. Figure \ref{fig:unip2} illustrates the blowing up procedure. If $\Stab_U(x)$ is strictly positive-dimensional for all $x \in X$, then a sequence of blow-ups can still be performed to obtain a variety $\hX$ satisfying (ss=Ms$\neq \emptyset$[$U$]) instead (see Remark \ref{rmk0.6} of the present paper). This is analogous to the case in classical GIT when $\emptyset = X^s \neq X^{Ms} \subsetneq X^{ss}$. Note that while in the reductive case the semistable locus must be non-empty in order to be able to perform the sequence of blow-ups, the $\hU$-blow-ups can always be performed. Unlike in Figure \ref{fig:classicalgit2}, we obtain only one projective completion for the geometric quotient $\xhshu / \hU$, coming from $\hX$, and it cannot in general be viewed as a partial desingularisation of another projective completion.

By analogy with classical GIT, a version of the stable locus inside $X$, denoted $\xhshu$, can be defined as $$\xhshu := \widehat{\pi}(\hX_{\min +}^{s,\hU} \setminus \widehat{E})$$ where $\widehat{E}$ denotes the exceptional divisor for the blow-down map $\widehat{\pi}: \hX \to X$. If $\hX_{\min +}^{s, \hU} = \emptyset$ then we define $\xhshu := \widehat{\pi} (U \widehat{Z}_{\min} \setminus \widehat{E})$ instead, where $\widehat{Z}_{\min}$ is the analogue of $\zmin$ for $\hX$.

\subsection{GIT for linear algebraic groups with internally graded unipotent radical} \label{subsec:internallygraded}

\begin{figure}
  \begin{subfigure}{1 \textwidth} 
   \centering
     \begin{tikzpicture}[scale=1,x=(10:10cm),y=(0:5.2cm), z=(90:1.3cm),
    cross line/.style={preaction={draw=white,-,line width=3pt}},
    equal/.style={double distance=1pt}]
    
\node (c1) at (1,1.7,2.4) {  } ;
\node (c4) at (1,-1.55,2.4) { };
\node (c5) at (1,1.7,-0.4) { } ;
\node (c6) at (1,-1.55,-0.4) { } ;
      
\node (16) at (1,-0.1,2) {\(   X \gitq \hU   \)};
\node (17) at (1,1,2) {\(  \left( X \gitq \hU \right)   \gitq R_{\l} =: X \gitq H \)};
\node (17') at (1,1,0) {\( \left( X \gitq \hU \right)^{s, R_{\l}} / R_{\l}  = X^{s,H}_{\min +} /H \)};
\node (18) at (1,-0.1,1) {\( X^{s,\hU}_{\min +} / \hU \)};
\node (18') at (1,-0.1,0) {\( \left(  X \gitq \hU   \right)^{s, R_{\l}}   \)};
\node(19) at (1,-1,2) {\( X \)};
\node(20) at (1,-1,1) {\(   X^{s, \hU}_{\min +} \)};
\node(20') at (1,-1,0) {\( X^{ss,H}_{\min +} = X^{s,H}_{\min +} = \bigcap_{h \in H} h X^{s,T}_{\min +} \)};
     
  \filldraw [fill = yellow, opacity=0.2 ]  (c6) rectangle (c1);

\draw [black,loosely dotted] (c6) rectangle (c1);
\draw [ dashed, ->] (16) -- (17);
\draw [->] (20') -- (18');
\draw [right hook ->] (20') -- (20);
\draw[right hook ->] (18') -- (18);
\draw [equal] (18) -- (16);
\draw [->] (18') -- (17');
\draw [equal] (17') -- (17);
\draw [->] (20) -- (18);
\draw [dashed, ->] (19) -- (16);
\draw[right hook ->] (20) -- (19);

   		\end{tikzpicture}
		\caption{When (ss=s$\neq \emptyset$[$U$]) for the action of $\hU$, $X \neq \overline{U \zmin}$ and $\left(  X \gitq \hU   \right)^{ss, R_{\l}} =\left(  X \gitq \hU   \right)^{s, R_{\l}} \neq \emptyset$ } \label{fig:interngrad1}
    \end{subfigure}

\begin{subfigure}{1\textwidth}
 \begin{tikzpicture}[scale=0.87,x=(57:7.3cm),y=(0:4.2cm), z=(90:1.3cm),
    cross line/.style={preaction={draw=white,-,line width=3pt}},
    equal/.style={double distance=1pt}]

\node (c1) at (1,1.4,2.5) {  } ;
\node (c2) at (-1,1.4,2.5) { };
\node (c3) at (-1,-1.35,2.5) { };
\node (c4) at (1,-1.35,2.5) { };
\node (c5) at (1,1.4,-2.45) { } ;
\node (c6) at (1,-1.35,-2.45) { } ;
\node (c7) at (-1,1.4,-2.45) { } ;
\node (c8) at (-1,-1.35,-2.45) { } ;
\node (c9) at (0,-1.35,2.5) {};
\node (c11) at (0,1.4,2.5) {};
\node (c10) at (0,-1.35,-2.45) {};
\node (c12) at (0,1.4,-2.45) { };
      
\node (1) at (-1,-1,2) {\( X \)};
\node (2') at (-1,-1,-2) {\(X^{\widehat{s},H} \)};
\node (3') at (-1,1,-2) {\( X^{\widehat{s},H}/ H \)};
\node (8) at (0,-1,2) {\( \hX \)};
\node (9) at (0,0,2) {\( \hX \gitq \hU \)};
\node (10) at (0,1,2) {\(  \hX \gitq H \)};
\node (11) at (0,-1,1) {\(  \hX^{s, \hU}_{\min +} \)};
\node(11') at (0,-1,0) {\( \hX^{ss,H}_{\min +} \)};
\node (12') at (0,-1,-1) {\( \hX^{s,H}_{\min +} \)};
\node (13') at (0,-1,-2) {\( \hX^{s,H}_{\min +} \setminus \widehat{E} \)};
\node (12) at (0,0,1) {\( \hX^{s, \hU}_{\min +} / \hU \)};
\node (13) at (0,0,0) {\( \left(\hX \gitq \hU \right)^{ss, R_{\l}} \)};
\node (14) at (0,0,-1) {\(\left(\hX \gitq \hU \right)^{s, R_{\l}} \)};
\node (15) at (0,1,-1) {\( \hX^{s,H}_{\min +} / H \)};
\node (15') at (0,1,-2) {\( \left(\hX^{s,H}_{\min +} \setminus \widehat{E} \right) / H \)};
\node (16) at (1,0,2) {\(   \tX \gitq \hU   \)};
\node (17) at (1,1,2) {\(   \tX \gitq H \)};
\node (17') at (1,1,0) {\(  \tX^{s,H}_{\min +}/H \)};
\node (18) at (1,0,1) {\( \tX^{s,\hU}_{\min +} / \hU \)};
\node (18') at (1,0,0) {\( \left(   \tX \gitq \hU   \right)^{s, R_{\l}}   \)};
\node(19) at (1,-1,2) {\( \tX \)};
\node(20) at (1,-1,1) {\(  \tX^{s, \hU}_{\min +} \)};
\node(20') at (1,-1,0) {\( \tX^{ss,H}_{\min +} = \tX^{s,H}_{\min +} \)};
\node(21) at (1,-1,-1) {\( \tX^{s,H}_{\min +} \setminus \widetilde{E} \)};
\node(22) at (1,1,-1) {\( \left( \tX^{s,H}_{\min +} \setminus \widetilde{E} \right) / H \)};
\node (23) at (1,0,-1)  {\( \left(   \tX \gitq \hU   \right)^{s, R_{\l}} \setminus \widetilde{E}_{\hU} \)};
     
\filldraw  [fill = red, draw=black,loosely dotted, opacity=0.2] (c3) rectangle (c7);
\filldraw  [fill = orange, draw=black,loosely dotted, opacity=0.2] (c10) rectangle (c11);
\filldraw  [fill = yellow, draw=black,loosely dotted, opacity=0.2] (c6) rectangle (c1);
  
\draw [black,loosely dotted] (c4) -- (c9);
\draw [black,loosely dotted] (c6) rectangle (c1);
\draw  [black,loosely dotted] (c10) rectangle (c11);
\draw  [black,loosely dotted] (c3) rectangle (c7);
\draw [black,loosely dotted] (c6) -- (c10);
\draw [black,loosely dotted] (c1) -- (c11);
\draw [black,loosely dotted] (c5) -- (c12);
\draw [black,loosely dotted] (c10) -- (c8);
\draw [black,loosely dotted] (c9) -- (c3);
\draw [black,loosely dotted] (c11) -- (c2);
\draw [black,loosely dotted] (c12) -- (c7);

\draw [loosely dotted, black] (c11) -- (c12);
\draw [-> ] (2') -- (3');
\draw [right hook ->] (2') -- (1);
\draw [right hook ->](11) -- (8);
\draw [equal](12) -- (9);
\draw [right hook ->] (12') -- (11');
\draw [right hook ->] (13') -- (12');
\draw [right hook ->] (11') -- (11);
\draw [->] (13') -- (15');
\draw [right hook ->](14) -- (13);
\draw [right hook ->] (15') -- (15);
\draw [ dashed, ->] (16) -- (17);
\draw [right hook ->] (20') -- (20);
\draw[right hook ->] (18') -- (18);
\draw [equal] (18) -- (16);
\draw [->] (18') -- (17');
\draw [equal] (17') -- (17);
\draw [->] (20) -- (18);
\draw [dashed, ->] (19) -- (16);
\draw[right hook ->] (20) -- (19);
\draw [ ->] (21) -- (23);
\draw [->] (23) -- (22);
\draw [right hook -> ] (21) -- (20');
\draw [right hook ->] (22) -- (17');
\draw [->] (22) -- (15) node [pos=0.35,right] {$\cong$};
\draw [->] (13') -- (2') node [midway,right] {$\cong$};
\draw [ ->] (21) -- (12') node [pos=0.45,right] {$\cong$};
\draw [->] (8) -- (1) node[midway,left] {$\widehat{\pi}$} ;
\draw [->,cross line] (17) -- (10) node[pos=0.8,left] {$\widetilde{\pi}_{H}$};
\draw[ ->] (19) -- (8) node[pos=0.8,left] {$\widetilde{\pi}$} ;
\draw [->] (15') -- (3') node[pos=0.3,right]  {$\cong$};
\draw [right hook ->] (15) -- (10);
\draw [->, cross line] (16) -- (9) node[pos=0.8,left] {$\widetilde{\pi}_{\hU}$};
\draw [dashed,->, cross line](8) -- (9);

\draw[->, cross line](11) -- (12);

\draw [->](14) -- (15);
\draw [->] (20') -- (18');
\draw [right hook-> ](13) -- (12);
\draw [->, cross line] (11') -- (13);
\draw [->] (12') -- (14);
\draw [ForestGreen,line width=0.5mm,right hook ->, cross line] (3') -- (10);
\draw [ForestGreen,line width=0.5mm,right hook ->, cross line] (3') -- (17) ;
\draw [->] (23) -- node[pos=0.45,right]{$\cong$} (14);
\draw [dashed,->, cross line](9) -- (10);
\draw [->, cross line](13) -- (10);

		\end{tikzpicture} \caption{When $\left(\hX \gitq \hU \right)^{s, R / \l(\GG_m)} \neq \emptyset$, $\hX \neq \overline{ U \widehat{Z}_{\min}}$ and $\tX \neq \overline{U \widetilde{Z}_{\min}}$. The green arrows denote projective completions.} \label{fig:interngrad2}
		
		\end{subfigure}  \caption{GIT for linear algebraic groups with internally graded unipotent radical (in the setting of the present paper)} 
\end{figure} 

More generally, let $H = U \rtimes R$ be a linear algebraic group with internally graded unipotent radical acting linearly on $X$. We let $\l: \GG_m \to R$ denote a central one-parameter subgroup such that $\l(\GG_m)$ acts on $\Lie U$ with strictly positive weights and we let $R_{\l}$ denote the reductive quotient group $R / \l(\GG_m)$. 

When the condition (ss=s$\neq \emptyset$[$U$]) (or (ss=Ms$\neq \emptyset$[$U$])) is satisfied for the action of $\hU = U \rtimes \l(\GG_m)$ on $X$, then by Section \ref{subsec:gradedunip} there exists a  projective geometric quotient $X \gitq \hU$ for the action of $\hU$ on $\xminshu$ (if $\xminshu = \emptyset$ so that $X = \overline{U \zmin}$ then in what follows we simply replace all occurrences of $\xminshu$ and $X \gitq \hU$ with $U \zmin$ and $\zmin$ respectively). The projective variety $X \gitq \hU$ has an induced action of the reductive group $R_{\l}$ which is linear with respect to an ample line bundle $L_{\hU}$ on $X \gitq \hU$ such that $q_{\hU}^{\ast} L_{\hU}$ is a tensor power of $L$ where $q_{\hU}: \xminshu \to X \gitq \hU$ denotes the quotient map. If $(X \gitq \hU)^{ss,R_{\l}} \neq \emptyset$ then we can further quotient by $R_{\l}$ to obtain a projective variety $X \gitq H := ( X \gitq \hU ) \gitq R_{\l}$ which is a geometric quotient for the action of $R_{\l}$ on $(X \gitq \hU )^{s,R_{\l}}$, and for the action of $H$ on $$\xminsh := q_{\hU}^{-1}( (X \gitq \hU)^{s,R_{\l}}) \cap \xminshu.$$ The semistable locus is analogously defined as $$\xminssh := q_{\hU}^{-1}( (X \gitq \hU)^{ss,R_{\l}} ) \cap \xminshu$$ and both the stable and semistable loci satisfy Hilbert-Mumford-type criteria: $$ X_{\min +}^{(s)s,H} = \bigcap_{h \in H} h X^{(s)s,T}_{\min +}$$ for a fixed maximal torus $T \subseteq R$ containing the central one-parameter subgroup which grades $U$. If semistability coincides with stability for the action of $R_{\l}$ on $X \gitq \hU$, then $X \gitq H$ is a geometric quotient for the action of $H$ on $\xminsh$ since in this case $\xminssh = \xminsh$. Figure \ref{fig:interngrad1} illustrates this situation.

If the condition (ss=Ms$\neq \emptyset$[$U$]) is not satisfied, then as described in Section \ref{subsec:gradedunip} a sequence of blow-ups can be performed to obtain a projective variety $\hX$ for which either (ss=s$\neq \emptyset$[$U$]) or (ss=Ms$\neq \emptyset$[$U$]) holds. For the projective variety $\hX \gitq \hU$, if $\emptyset \neq (\hX \gitq \hU)^{s,R_{\l}} \subsetneq (\hX \gitq \hU)^{ss,R_{\l}}$, then as described in Section \ref{subsec:classicalredgit} an additional sequence of blow-ups can be performed to produce a projective variety $(\hX \gitq \hU)^{\sim}$ such that semistability coincides with stability (or Mumford-stability) for a suitably lifted action of $R_{\l}$. Equivalently we can construct a variety $\tX$ with a blow-down map to $\hX$ and a linearised action of $H$ such that $\tX \gitq \hU = ( \hX \gitq \hU )^{\sim}$ is a projective geometric quotient for the action of $H$ on $\tX_{\min+}^{s,H}$. An alternate stable locus inside $X$, denoted $\xhsh$, is then defined as $$\xhsh := \widehat{\pi}(\hX_{\min +}^{s,H} \setminus \widehat{E}) = \widehat{\pi} (\widetilde{\pi} (\tX_{\min +}^{s,H} \setminus \widetilde{E}) ) \setminus \widehat{E} )$$ where $\widehat{E}$ and $\widetilde{E}$ denote the exceptional divisors for the blow-down maps $\widehat{\pi}: \hX \to X$ and $\widetilde{\pi}: \tX \to \hX$ respectively. Figure \ref{fig:interngrad2} illustrates this sequence of blow-ups. By construction, $\xhsh / H$ is a geometric quotient for the action of $X$ on $\xhsh$ and $\tX \gitq \hU$ is a projective completion of this quotient.

\section{Algorithm for performing quotients of projective varieties by internally graded groups} \label{sec:algorithm}

In the set-up of Section \ref{subsec:internallygraded}, if $H$ acts linearly on $X$, then there exists an open subset $\xhsh \subseteq X$ such that $\xhsh / H$ is a geometric quotient for the action of $H$ on $\xhsh$, as well as a projective completion $\tX \gitq H$ which is itself a projective geometric quotient for the action of $H$ on an open subset of $\tX$. This is the main result of the present paper. However, just as in classical GIT where the stable and semistable loci can be empty, leading to empty geometric and GIT quotients, the geometric quotient $\xhsh/H$ and its projective completion $\tX \gitq H$ could be empty. The purpose of this section is to describe a generalisation of the main result of the present paper which ensures that given the linearised action of $H$ on $X$, a non-empty geometric quotient and projective completion is always obtained. The proof is in the form of an algorithm, called the Projective completion algorithm; it is based on the inductive construction from \cite{BHK}.

Let $H= U \rtimes R$ be a linear algebraic group with unipotent radical $U$ graded by a central one-parameter subgroup $\l: \GG_m \to R$ of $R$, acting linearly on an irreducible projective variety $X$. We assume moreover that $X$ is non-empty and that $H$ is connected; otherwise we replace $H$ with its identity component. We denote the trivial one-parameter subgroup by $\l_0$. As before, we let $R_{\l}$ denote the reductive quotient group $R / \l(\GG_m)$. Let $\mathcal{L}$ denote the linearisation for the action of $H$ on $X$, and by abuse of notation also the line bundle with respect to which the linearisation is defined. Given this data, together with an invariant inner product on $\Lie R$, the projective completion algorithm produces a non-empty open subset $S_0(X,H,\lambda,\mathcal{L})$ of $X$ which admits a quasi-projective geometric quotient by the action of $H$ (cf.\ \cite{BHK}), together with a projective completion $\PC(S_0(X,H,\lambda,\mathcal{L})/H)$ of this geometric quotient.

This Projective completion algorithm is based on \cite{BHK}. It will be described in Section \ref{subsec:bualg} and is illustrated in Figure \ref{fig:bualg}. It relies on another stand-alone algorithm called the Replacement algorithm, which we describe in Section \ref{subsec:repalg} and which is illustrated in Figure \ref{fig:repalg}. The Replacement algorithm also appears on the left-hand side of Figure \ref{fig:bualg}. 

\subsection{The Replacement algorithm}

\label{subsec:repalg}

\begin{figure}
	 \centering
	            
        \begin{tikzpicture}[scale=2,x=(0:1.9cm), y=(90:1.5cm)]
         
         \node (c1) at (-0.7,8.7) {};
         \node (c2) at (3.55,8.7) {};
         \node (c3) at (-0.7, 1.7) {};
         \node (c4) at (3.55,1.7) {};
          
       \node[rectangle, draw = black, loosely dotted, draw opacity = 1, fill =gray, fill opacity = 0.2, text opacity =1] (1) at (0,2.5) { $\left( \overline{Y_{\b^{(N)}}^{ss}}, P_{\b^{(N)}}, \l_{\b^{(N)}}, \mathcal{L}|_{\overline{Y_{\b^{(N)}}^{ss}}} \right)$};
       \node (2) at (0,2.5) {  };
       \node (3) at (0,3) {};

       \node (9) at (0,4) {} ;
       \node (10) at (0,4.5) {};
       \node[rectangle, draw = black, loosely dotted, draw opacity = 1, fill =gray, fill opacity = 0.2, text opacity =1] (12) at (0,5) {$\left( \overline{Y_{\b^{(2)}}^{ss}}, P_{\b^{(2)}}, \l_{\b^{(2)}}, \mathcal{L}|_{\overline{Y_{\b^{(2)}}^{ss}}} \right)$};
       \node (13) at (0,6) {};
       \node (14) at (0,6.5) {};
       \node[rectangle, draw = black, loosely dotted, draw opacity = 1, fill =gray, fill opacity = 0.2, text opacity =1] (15) at (0,7) {$\left( \overline{Y_{\b}^{ss}}, P_{\b}, \l_{\b}, \mathcal{L}|_{\overline{Y_{\b}^{ss}}} \right)$};
       \node (16) at (0,8) {};
       \node (17) at (0,8.5) {};
       \node[rectangle, draw = black, loosely dotted, draw opacity = 1, fill =gray, fill opacity = 0.2, text opacity =1,label={[xshift=1.6cm, yshift=-0.65cm]$(\star \star \star)$}] (18) at (0,8.5) {$(X,H,\l,\mathcal{L})$};
       
       \node[rectangle, draw = black, loosely dotted, draw opacity = 1, fill =red, fill opacity = 0.2, text opacity =1,label={$(\star)$}] (a) at (2.7,3) {$\left( \overline{Y_{\b^{(N)}}^{ss}}, P_{\b^{(N)}}, \l_{\b^{(N)}}, \mathcal{L}|_{\overline{Y_{\b^{(N)}}^{ss}}} \right)$};
       
        \node[rectangle, draw = black, loosely dotted, draw opacity = 1, fill =orange, fill opacity = 0.2, text opacity =1,label={$(\star \star)$}] (a') at (2.7,2) {$\left( \overline{Y_{\b^{(N)}}^{ss}}, \{e\}, \l_0, \mathcal{L}|_{\overline{Y_{\b^{(N)}}^{ss}}} \right)$};

       \node[rectangle, draw = black, loosely dotted, draw opacity = 1, fill =orange, fill opacity = 0.2, text opacity =1,label={$(\star \star)$}] (d) at (2.7,4.5) {$\left( \overline{Y_{\b^{(2)}}^{ss}}, L_{\b^{(2)}}/\l_{\b^{(2)}} (\GG_m) , \l_{0}, \mathcal{L}|_{\overline{Y_{\b^{(2)}}^{ss}}} \right)$};
       \node[rectangle, draw = black, loosely dotted, draw opacity = 1, fill =red, fill opacity = 0.2, text opacity =1,label={$(\star)$}] (e) at (2.7,5.5) {$\left( \overline{Y_{\b^{(2)}}^{ss}}, P_{\b^{(2)}}, \l_{\b^{(2)}}, \mathcal{L}|_{\overline{Y_{\b^{(2)}}^{ss}}} \right)$};
       \node[rectangle, draw = black, loosely dotted, draw opacity = 1, fill =orange, fill opacity = 0.2, text opacity =1,label={$(\star \star)$}] (f) at (2.7,6.5) {$\left( \overline{Y_{\b}^{ss}}, L_{\b}/\l_{\b}(\GG_m), \l_0, \mathcal{L}|_{\overline{Y_{\b}^{ss}}} \right)$};
       \node[rectangle, draw = black, loosely dotted, draw opacity = 1, fill =red, fill opacity = 0.2, text opacity =1,label={$(\star)$}] (g) at (2.7,7.5) {$\left( \overline{Y_{\b}^{ss}}, P_{\b}, \l_{\b}, \mathcal{L}|_{\overline{Y_{\b}^{ss}}} \right)$};

       \draw  (c1) rectangle (c4);

       \draw  (12) -- (0,4);
       
       \draw [dashed] (0,4) -- (0,3.5);
       \draw [->] (0,3.5) -- (1);

       \path[->]

        (18) edge  (15)
        
        (15) edge node[rectangle, draw,rounded corners,sloped, anchor=center, above,yshift=1mm] {$\l_{\b}(\GG_m)$ acts non-trivially} (g.west)
        (15) edge node[rectangle, draw,rounded corners,sloped, anchor=center, above,yshift=1mm] {$\l_{\b}(\GG_m)$ acts trivially} node[rectangle, draw,rounded corners,sloped,anchor=center, below,yshift=-1mm] {and $\overline{Y_{\b}^{ss}}^{ss, L_{\b}/\l_{\b}(\GG_m)} \neq \emptyset$} 
        (f.west)
        (15) edge node[rectangle, draw,rounded corners,sloped,anchor=center,above,rotate=180,yshift=1mm] {$\l_{\b}(\GG_m)$ acts trivially} node[rectangle, draw,rounded corners,sloped,anchor=center, below, rotate=180,yshift=-1mm] {and $\overline{Y_{\b}^{ss}}^{ss, L_{\b}/\l_{\b}(\GG_m)} = \emptyset$} (12)
        (12) edge node[rectangle, draw,rounded corners,sloped, anchor=center, above,yshift=1mm] {$\l_{\b^{(2)}}(\GG_m)$ acts non-trivially} (e.west)
        (12) edge node[rectangle, draw,rounded corners,sloped, anchor=center, above,yshift=1mm] {$\l_{\b^{(2)}}(\GG_m)$ acts trivially} node[rectangle, draw,rounded corners,sloped,anchor=center, below,yshift=-1mm] {and $\overline{Y_{\b^{(2)}}^{ss}}^{ss, L_{\b^{(2)}}/\l_{\b^{(2)}}(\GG_m)} \neq \emptyset$} 
        (d.west)
        (1) edge node[rectangle, draw,rounded corners,sloped, anchor=center, above,yshift=1mm] {$\l_{\b^{(N)}}(\GG_m)$ acts non-trivially} (a.west)
        (1) edge node[rectangle, draw,rounded corners,sloped, anchor=center, above,yshift=1mm] {$\l_{\b^{(N)}}(\GG_m)$ acts trivially}  
        (a'.west);

\end{tikzpicture}

\caption{Replacement algorithm. The algorithm starts in the top left gray rectangle, and proceeds by following the arrows according to the conditions specified in the framed labels. The algorithm terminates when a red or orange rectangle is reached.} \label{fig:repalg}
	
	\end{figure}

This algorithm is illustrated in Figure \ref{fig:repalg} and is described in words below.  

\begin{repalg} \thlabel{replacementalg}
Suppose that $(X,H,\l,\mathcal{L})$ encodes the data consisting of the linearised action of a non-reductive group $H$ with internally graded unipotent radical $U$ and Levi subgroup $R$ on an irreducible projective variety $X$ as above, and let $R_{\l} = R/\l(\GG_m)$. Here we fix an adjoint-invariant inner product on $\Lie R$; this restricts to an adjoint-invariant inner product on $\Lie R'$ for any $R' \leq R$. One of three conditions must be satisfied for the action of $H$ on $X$: 
\begin{enumerate}[leftmargin=2cm]
\item[($\star$)] $\l(\GG_m)$ acts non-trivially on $X$; 
\item[($\star \star$)] $\l(\GG_m)$ acts trivially on $X$ and $X^{ss,R_{\l}} \neq \emptyset$; or
\item[($\star \star \star$)] $\l(\GG_m)$ acts trivially on $X$ and $X^{ss,R_{\l}} = \emptyset$.
\end{enumerate}
Note that if $\l(\GG_m)$ acts trivially on $X$, then $U$ must also act trivially on $X$, and so $X$ is already a geometric quotient for the action of $\hU$ on $X$. Thus it suffices to consider the action of $H / \l(\GG_m) = R/ \l(\GG_m)=R_{\l}$ on $X$, and in particular we can consider the semistable locus for this reductive action which determines which of $(\star \star)$ and $(\star \star \star)$ is satisfied. 

The input of the replacement algorithm is a $4$-tuple $(X,H,\l,\mathcal{L})$ satisfying $(\star \star \star)$. The algorithm then produces another such $4$-tuple $(X',H',\l',\mathcal{L}')$ satisfying either $(\star)$ or $(\star \star)$. The construction of $(X',H',\l',\mathcal{L}')$ from $(X,H,\l,\mathcal{L})$ is depicted in Figure \ref{fig:repalg}, and can be described in words as follows: 

\begin{step} Since by assumption $(X,H,\l,\mathcal{L})$ satisfies $(\star \star \star)$, the semistable locus for the induced action of $R_{\l} = R/ \l(\GG_m)$ on $X$ is empty. Thus there exists an unstable stratum $S_{\b}$ with $\b \neq \emptyset$ which is open in $X$, such that $X = \overline{S_{\b}}= \overline{R_{\l} Y_{\b}^{ss}}$. As observed in Section \ref{subsec:classicalredgit}, this means that $\overline{Y_{\beta}^{ss}}$ is irreducible and non-empty. Let $P_{\b}$ denote the parabolic subgroup of $R_{\l}$ determined by $\b$ and let $\l_{\b}: \GG_m \to R_{\l}$ denote the corresponding one-parameter subgroup. We then replace $(X,H,\l,\mathcal{L})$ by $(\overline{Y_{\b}^{ss}}, P_{\b},\l_{\b}, \mathcal{L} |_{\overline{Y_{\b}^{ss}}} )$ and proceed to \thref{step2}. 
\end{step} 

\begin{step} \thlabel{step2}
If $(\overline{Y_{\b}^{ss}}, P_{\b},\l_{\b}, \mathcal{L} |_{\overline{Y_{\b}^{ss}}} )$ satisfies ($\star$), then we set $(X',H',\l',\mathcal{L}') = (\overline{Y_{\b}^{ss}}, P_{\b},\l_{\b}, \mathcal{L} |_{\overline{Y_{\b}^{ss}}} )$ and the algorithm terminates here. 

If $(\overline{Y_{\b}^{ss}}, P_{\b},\l_{\b}, \mathcal{L} |_{\overline{Y_{\b}^{ss}}} )$ satisfies ($\star \star$), then we set $(X',H',\l',\mathcal{L}') = (\overline{Y_{\b}^{ss}}, P_{\b}/\l_{\b}(\GG_m),\l_{0}, \mathcal{L} |_{\overline{Y_{\b}^{ss}}} )$ and the algorithm also terminates here. 

If $(\overline{Y_{\b}^{ss}}, P_{\b},\l_{\b}, \mathcal{L} |_{\overline{Y_{\b}^{ss}}} )$ satisfies ($\star \star \star$), then for the induced action of $R_{\l}$ on $X$, there exists an unstable stratum $S_{\b^{(2)}}$ with $\b^{(2)} \neq \emptyset$ which is open in $\overline{Y_{\b}^{ss}}$, such that $\overline{Y_{\b}^{ss}} = \overline{S_{\b^{(2)}}}= \overline{R_{\l} Y_{\b^{(2)}}^{ss}}$  and $\overline{R_{\l} Y_{\b^{(2)}}^{ss}}$ is a non-empty and irreducible projective variety (see Section \ref{subsec:classicalredgit}). Let $L_{\b}$ denote a Levi subgroup of $P_{\b}$ and let $P_{\b^{(2)}}$ denote the parabolic subgroup of $L_{\b}/\l_{\b}(\GG_m)$ determined by $\b^{(2)}$. Moreover, let $\l_{\b^{(2)}}: \GG_m \to L_{\b} / \l_{\b}(\GG_m)$ denote the corresponding one-parameter subgroup. We then replace $(\overline{Y_{\b}^{ss}}, P_{\b},\l_{\b}, \mathcal{L} |_{\overline{Y_{\b}^{ss}}} )$ by $(\overline{Y_{\b^{(2)}}^{ss}}, P_{\b^{(2)}},\l_{\b^{(2)}}, \mathcal{L} |_{\overline{Y_{\b^{(2)}}^{ss}}} )$ and proceed to Step $3$.

\end{step}

\begin{step} Repeat Step $2$, except with the output $( \overline{Y_{\b^{(2)}}^{ss}}, P_{\b^{(2)}}, \l_{\b^{(2)}}, \mathcal{L} |_{ \overline{Y_{\b^{(2)}}^{ss}}})$ of the previous step as the starting $4$-tuple.  
\end{step} 

More generally, assuming that for each Step $i \leq k$ the starting $4$-tuple satisfies $(\star \star \star)$, then Step $k+1$ can be defined inductively. Let $(\overline{Y_{\b^{(i)}}^{ss}}, P_{\b^{(i)}}, \l_{\b^{(i)}}, \mathcal{L} |_{\overline{Y_{\b^{(i)}}^{ss}}})$ be the output of Step $i$ for $1 \leq i \leq k$ (for $i=1$ we let $\b^{(1)} = \b$).

\begin{stepk+1} Repeat Step $k$, replacing $(\overline{Y_{\b^{(k-1)}}^{ss}}, P_{\b^{(k-1)}}, \l_{\b^{(k-1)}}, \mathcal{L} |_{\overline{Y_{\b^{(k-1)}}^{ss}}})$ with $(\overline{Y_{\b^{(k)}}^{ss}}, P_{\b^{(k)}}, \l_{\b^{(k)}}, \mathcal{L} |_{\overline{Y_{\b^{(k)}}^{ss}}})$ as the starting $4$-tuple. If $(\star \star \star)$ holds, then $\l_{\b^{(k)}}(\GG_m)$ acts trivially on $\overline{Y_{\b^{(k)}}^{ss}}$ and the semistable locus for the induced action of $L_{\b^{(k)}}/\l_{\b^{(k)}}(\GG_m)$ on  $\overline{Y_{\b^{(k)}}^{ss}}$ is empty, where $L_{\b^{(k)}}$ denotes a Levi subgroup of $P_{\b^{(k)}}$. In this case, we let $\b^{(k+1)}$ denote the non-trivial character determining the unstable stratum $S_{\b^{(k+1)}}$ for the action of $L_{\b^{(k)}}/\l_{\b^{(k)}}(\GG_m)$ on $\overline{Y_{\b^{(k)}}^{ss}}$ such that $\overline{Y_{\b^{(k)}}^{ss}} = \overline{S_{\b^{(k+1)}}}$. Moreover, let $Y_{\b^{(k+1)}}^{ss}$ denote the locally closed subset of $\overline{Y_{\b^{(k)}}^{ss}}$ determined by this action such that $S_{\b^{(k+1)}} = L_{\b^{(k)}} / \l_{\b^{(k)}}(\GG_m)  Y_{\b^{(k+1)}}^{ss}$. Finally, let $\l_{\b^{(k+1)}}: \GG_m \to L_{\b^{(k)}} / \l_{\b^{(k)}}(\GG_m)$ and $P_{\b^{(k+1)}}$ denote the non-trivial one-parameter subgroup and parabolic subgroup respectively of $L_{\b^{(k)}} / \l_{\b^{(k)}}(\GG_m)$ determined by $\b^{(k+1)}$. With this notation, repeating Step $k$ then produces the $4$-tuple $(\overline{Y_{\b^{(k+1)}}^{ss}}, P_{\b^{(k+1)}}, \l_{\b^{(k+1)}}, \mathcal{L} |_{\overline{Y_{\b^{(k+1)}}^{ss}}} ).$ 

\end{stepk+1}

\end{repalg}

\begin{prop} The replacement algorithm terminates after a finite number of steps.
\end{prop} 

\begin{proof}

If we arrive at Step $i$ for $i \geq 1$, then the only way for the algorithm to continue on to Step $i+1$ is if the starting $4$-tuple $(\overline{Y_{\b^{(i-1)}}^{ss}}, P_{\b^{(i-1)}}, \l_{\b^{(i-1)}}, \mathcal{L} |_{\overline{Y_{\b^{(i-1)}}^{ss}}})$ satisfies $(\star \star \star)$. If so, then the $4$-tuple must be replaced with $(\overline{Y_{\b^{(i)}}^{ss}}, P_{\b^{(i)}}, \l_{\b^{(i)}}, \mathcal{L} |_{\overline{Y_{\b^{(i)}}^{ss}}})$. But we have that $\operatorname{dim} L_{\b^{(i)}} \leq \operatorname{dim} P_{\b^{(i)}} \leq \operatorname{dim}  L_{\b^{(i-1)}} / \l_{\b^{(i-1)}}(\GG_m) < \operatorname{dim} L_{\b^{(i-1)}},$ since $\l_{\b^{(i-1)}}(\GG_m)$ is a non-trivial one-parameter subgroup of $L_{\b^{(i-1)}}$. Hence at each step the dimension of the Levi subgroup appearing in the $4$-tuple strictly decreases. This implies that if at each step the starting $4$-tuple satisfies $(\star \star \star)$, we must eventually reach a Step $N$ with starting $4$-tuple $(\overline{Y_{\b^{(N-1)}}^{ss}}, P_{\b^{(N-1)}}, \l_{\b^{(N-1)}}, \mathcal{L} |_{\overline{Y_{\b^{(N)}}^{ss}}})$ such that $L_{\b^{(N-1)}} / \l_{\b^{(N-1)}}(\GG_m)$ is trivial. Hence the semistable locus for the induced action of $L_{\b^{(N-1)}} / \l_{\b^{(N-1)}}(\GG_m) = \{e\}$ on $\overline{Y_{\b^{(N-1)}}^{ss}}$ is all of $\overline{Y_{\b^{(N-1)}}^{ss}}$ and so $(\star \star)$ is automatically satisfied. 
\end{proof}

\subsection{The Projective completion algorithm} 

\label{subsec:bualg}
	
	This algorithm is illustrated in Figure \ref{fig:bualg} and described in words below.

\begin{figure}
    \centering

        \begin{tikzpicture}[scale=2,x=(0:1.5cm), y=(90:1cm),
    equal/.style={double distance=1pt}]
          
\node (c1) at (-3,7.3) {};
\node (c2) at (-1,7.3) {};
\node (c3) at (-3,-0.6) {};
\node (c4) at (-1,-0.6) {};

\draw (c1) rectangle (c4);

\node (b1) at (-0.2,7.3) {};
\node (b2) at (2.2, 7.3) {};
\node (b3) at (-0.2,-0.6) {};
\node (b4) at (2.2,-0.6) {};

\draw (b1) rectangle (b4);

\node (d1) at (-1.9,5.75) {};
\node (d2) at (-1.1,5.75) {};
\node (d3) at (-1.9,-0.35) {};
\node (d4) at (-1.1,-0.35) {};
\draw[loosely dotted] (d1) rectangle (d4);
          
\node (a1) at (0,-0.2) {  } ;
\node (a2) at (2,-0.2) { };
\node (a3) at (0,1.8) { };
\node (a4) at (2,1.8) { };

\node (a) at (1,0.8) {$(\star)$};

\node (e1) at (0,4.8) {};
\node (e2) at (2,4.8) {};
\node (e3) at (0,6.8) {};
\node (e4) at (2,6.8) {};

\node (e) at (1,5.8) {};

\node (f1) at (0,2.3) {};
\node (f2) at (2,2.3) {};
\node (f3) at (0,4.3) {};
\node (f4) at (2,4.3) {};

\node (f) at (1,3.4) {$(\star \star)$};

\node[rectangle,draw,minimum width = 2cm, minimum height = 0.7cm,fill=gray, fill opacity = 0.2, draw=black,loosely dotted, text opacity = 1] (r1) at (-2.5,0.4) {}; 

\node[rectangle,draw,minimum width = 2cm, minimum height = 0.7cm,fill=gray, fill opacity = 0.2, draw=black,loosely dotted, text opacity = 1] (r2) at (-2.5,5) {};

\node[rectangle,draw,minimum width = 2cm, minimum height = 0.7cm,fill=gray, fill opacity = 0.2, draw=black,loosely dotted, text opacity = 1] (r3) at (-2.5,6.8) {$(\star \star \star)$};

\node[rectangle, fill = red, draw=black,loosely dotted, fill opacity=0.2, text opacity = 1,minimum width = 2cm, minimum height = 0.7cm]  (r1r) at (-1.5,0.9) {$(\star)$};

\node[rectangle, fill = orange, draw=black,loosely dotted, fill opacity=0.2, text opacity = 1,minimum width = 2cm, minimum height = 0.7cm]  (r1o) at (-1.5,-0.1) {$(\star \star)$};

\node[rectangle, fill = orange, draw=black,loosely dotted, fill opacity=0.2, text opacity = 1,minimum width = 2cm, minimum height = 0.7cm]  (r2o) at (-1.5,4.5) {$(\star \star)$};

\node[rectangle, fill = red, draw=black,loosely dotted, fill opacity=0.2, text opacity = 1,minimum width = 2cm, minimum height = 0.7cm]  (r2r) at (-1.5,5.5) {$(\star)$};

\node[rectangle,draw,minimum width = 2cm, minimum height = 0.7cm,fill=gray, fill opacity = 0.2, draw=black,loosely dotted, text opacity = 1] (r4) at (-2.5,3.2) {};

\node[rectangle, fill = red, draw=black,loosely dotted, fill opacity=0.2, text opacity = 1,minimum width = 2cm, minimum height = 0.7cm]  (r4r) at (-1.5,3.7) {$(\star)$};

\node[rectangle, fill = orange, draw=black,loosely dotted, fill opacity=0.2, text opacity = 1,minimum width = 2cm, minimum height = 0.7cm]  (r4o) at (-1.5,2.7) {$(\star \star)$};

\node (1) at (-2.5,2.3) {};
\node (2) at (-2.5,1.3) {};

\node (alg2) at (-0.2,3.5) {};
\node (alg1) at (-1.1,3.5) {};

\draw[->, line width = 0.3mm] (alg1) -- (alg2);

\draw [->] (r2) -- (r4);
\draw [->] (r4.east) -- (r4r.west);
\draw [->] (r4.east) -- (r4o.west);
\draw [->] (r1.east) -- (r1r.west);
\draw [->] (r1.east) -- (r1o.west);
\draw (r4) -- (1);
\draw [dashed] (1) -- (2);
\draw [->] (2) -- (r1);

\draw [->]
(r3) -- (r2);

\draw [->] (r2.east) -- (r2r.west);
\draw [->] (r2.east) -- (r2o.west);

    
\filldraw  [fill = red, draw=black,loosely dotted, opacity=0.2] (a1) rectangle (a4);

\filldraw  [fill = yellow, draw=black,loosely dotted, opacity=0.2] (e1) rectangle (e4);

\filldraw  [fill = orange, draw=black,loosely dotted, opacity=0.2] (f1) rectangle (f4);

\draw [->,line width=0.3mm] (a) to  node[draw,circle,right, xshift=1mm] {1} (f) ;

\draw [->,line width=0.3mm] (f) -- (e) node[draw,circle,right,pos=0.5,xshift=1mm] {1'};

\draw [->,line width=0.3mm] (a) to [bend left = 32] node[draw,circle,below,xshift=23mm,yshift=-22mm] {2}  (r3)  ;

\draw [->,line width=0.3mm] (f) to [bend right = 30]  node[draw,circle,right, above,xshift = -8mm,yshift=6mm] {2'} (r3);

		\end{tikzpicture}
        \caption{Projective completion algorithm for the action of an internally graded linear algebraic group on a projective variety. The left-hand side corresponds to the replacement algorithm, described in Figure \ref{fig:repalg}. The sequence of blow-ups corresponding to each of the four paths $1$, $2$, $1'$ and $2'$ are described in Figure \ref{fig:bups}. Given a $4$-tuple $(X,H,\l,\mathcal{L})$, the algorithm starts in the grey square in the top left or in the red or orange squares on the right-hand side, depending on which of the three conditions $(\star)$, $(\star \star)$ and $(\star \star \star)$ is satisfied for the $4$-tuple. If at any point in the algorithm we find ourselves on the grey square in the top left (either at the start or after taking paths $2$ or $2'$), then we run the replacement algorithm. The output determines whether we proceed to the red or to the orange rectangle on the right-hand side: if the output is red we move to the red rectangle, if it is orange we move to the orange rectangle. Figure \ref{fig:proof} depicts the different paths which can be taken within the algorithm and illustrates why the algorithm must terminate in the yellow rectangle after a finite number of steps (\thref{prop:bualgterm}).}  \label{fig:bualg} 
	\end{figure}

    \begin{figure} 
\begin{subfigure}[t]{1\hsize}
        \centering
      \begin{tikzpicture}[scale=0.97,x=(60:6.6cm),y=(0:5cm), z=(90:1.2cm),
    cross line/.style={preaction={draw=white,-,line width=3pt}},
    equal/.style={double distance=1pt}]

\node (c2) at (-1,1.5,2.5) { };
\node (c3) at (-1,-1.45,2.5) { };

\node (c7) at (-1,1.5,-2.5) { } ;
\node (c8) at (-1,-1.45,-2.5) { } ;
\node (c9) at (0,-1.45,2.5) {};
\node (c11) at (0,1.5,2.5) {};
\node (c10) at (0,-1.45,-2.5) {};
\node (c12) at (0,1.5,-2.5) { };
      
\node (1) at (-1,-1,2) {\( X \)};
\node (2') at (-1,-1,-2) {\( S_0(X,H) \)};
\node (3') at (-1,1,-2) {\( S_0(X,H)/ H \)};
\node (8) at (0,-1,2) {\( \hX \)};
\node (9) at (0,0,2) {\( \hX \gitq \hU \)};
\node (11) at (0,-1,1) {\(  \hX^{s, \hU}_{\min +} \)};
\node(11') at (0,-1,0) {\( \hX^{ss,H} \)};
\node (12') at (0,-1,-1) {\( q_{\hU}^{-1} ( S_0(\hX \gitq \hU, R_{\l})  )  \)};
\node (13') at (0,-1,-2) {\(  q_{\hU}^{-1} ( S_0(\hX \gitq \hU, R_{\l})  )  \setminus \widehat{E} \)};
\node (12) at (0,0,1) {\( \hX^{s, \hU}_{\min +} / \hU \)};
\node (13) at (0,0,0) {\( \left(\hX \gitq \hU \right)^{ss, R_{\l}} \)};
\node (14) at (0,0,-1) {\(S_0(\hX \gitq \hU, R_{\l}) \)};
\node (15) at (0,1,-1) {\(S_0(\hX \gitq \hU, R_{\l}) / R_{\l} \)};
\node (15') at (0,1,-2) {\( q_{\hU}^{-1} ( S_0(\hX \gitq \hU, R_{\l}) )  \setminus \widehat{E}  / H \)};

\filldraw  [fill = red, draw=black,loosely dotted, opacity=0.2] (c3) rectangle (c7);
\filldraw  [fill = orange, draw=black,loosely dotted, opacity=0.2] (c10) rectangle (c11);

\draw  [black,loosely dotted] (c10) rectangle (c11);
\draw  [black,loosely dotted] (c3) rectangle (c7);

\draw [black,loosely dotted] (c10) -- (c8);
\draw [black,loosely dotted] (c9) -- (c3);
\draw [black,loosely dotted] (c11) -- (c2);
\draw [black,loosely dotted] (c12) -- (c7);

\draw [loosely dotted, black] (c11) -- (c12);
\draw [-> ] (2') -- (3');
\draw [right hook ->] (2') -- (1);
\draw [right hook ->](11) -- (8);
\draw [equal](12) -- (9);
\draw [right hook ->] (12') -- (11');
\draw [right hook ->] (13') -- (12');
\draw [right hook ->] (11') -- (11);
\draw [->] (13') -- (15');
\draw [right hook ->](14) -- (13);
\draw [right hook ->] (15') -- (15);

\draw [->] (13') -- (2') node [pos=0.2,right] {$\cong$};

\draw [->] (8) -- (1) node[midway,left] {$\widehat{\pi}$} ;

\draw [->] (15') -- (3') node[pos=0.2,left] {$\widehat{\pi}_{H}$} node[pos=0.2,right] {$\cong$};

\draw [dashed,->](8) -- (9);

\draw[->](11) -- (12) node[midway,above]{$q_{\hU}$};

\draw [->](14) -- (15);

\draw [right hook-> ](13) -- (12);
\draw [->] (11') -- (13);
\draw [->] (12') -- (14);

		\end{tikzpicture}
        \caption{Blow-up corresponding to Path $1$ in Figure \ref{fig:bualg} (when $\hX \neq \overline{U \widehat{Z}_{\min}}$)} \label{fig:ro}
    \end{subfigure}

 \begin{subfigure}[t]{1\hsize}
        \centering
      \begin{tikzpicture}[scale=0.97,x=(60:5.2cm),y=(0:5cm), z=(90:1.2cm),
    cross line/.style={preaction={draw=white,-,line width=3pt}},
    equal/.style={double distance=1pt}]

\node (c2) at (-1,1.5,1.5) { };
\node (c3) at (-1,-1.45,1.5) { };
\node (c7) at (-1,1.5,-2.5) { } ;
\node (c8) at (-1,-1.45,-2.5) { } ;
\node (c9) at (0,-1.45,1.5) {};
\node (c11) at (0,1.5,1.5) {};
\node (c10) at (0,-1.45,-2.5) {};
\node (c12) at (0,1.5,-2.5) { };
      
\node (1) at (-1,-1,1) {\( X \)};
\node (2') at (-1,-1,-2) {\( S_0(X,H) \)};
\node (3') at (-1,1,-2) {\( S_0(X,H) / H \)};
\node (8) at (0,-1,1) {\( \hX \)};
\node (9) at (0,0,1) {\( \hX \gitq \hU  = \overline{\widehat{S}_{\b}} =  \overline{R_{\l} \widehat{Y}_{\b}^{ss} }  \)};
\node (11) at (0,-1,0) {\(  \hX^{s, \hU}_{\min +} \)};
\node (12') at (0,-1,-1) {\( q_{\hU}^{-1}(S_0(\hX \gitq \hU, R_{\l}) ) \)};
\node (13') at (0,-1,-2) {\( q_{\hU}^{-1}(S_0(\hX \gitq \hU, R_{\l}) ) \setminus \widehat{E} \)};
\node (12) at (0,0,0) {\( \hX^{s, \hU}_{\min +} / \hU \)};
\node (14) at (0,0,-1) {\( S_0(\hX \gitq \hU,R_{\l})    \)};
\node (15) at (0,1,-1) {\( S_0(\hX \gitq \hU, R_{\l}) / R_{\l} \)};
\node (15') at (0,1,-2) {\( \left( q_{\hU}^{-1}(S_0(\hX \gitq \hU, R_{\l}) ) \setminus \widehat{E}  \right) / H \)};

\filldraw  [fill = red, draw=black,loosely dotted, opacity=0.2] (c3) rectangle (c7);
\filldraw  [fill = gray, draw=black,loosely dotted, opacity=0.2] (c10) rectangle (c11);

\draw  [black,loosely dotted] (c10) rectangle (c11);
\draw  [black,loosely dotted] (c3) rectangle (c7);

\draw [black,loosely dotted] (c10) -- (c8);
\draw [black,loosely dotted] (c9) -- (c3);
\draw [black,loosely dotted] (c11) -- (c2);
\draw [black,loosely dotted] (c12) -- (c7);

\draw [loosely dotted, black] (c11) -- (c12);

\draw [right hook ->] (14) -- (12);

\draw [-> ] (2') -- (3');
\draw [right hook ->] (2') -- (1);
\draw [right hook ->](11) -- (8);
\draw [equal](12) -- (9);
\draw [right hook ->] (13') -- (12');
\draw [->] (13') -- (15');
\draw [right hook ->] (15') -- (15);
\draw [right hook ->] (12') -- (11);

\draw [->] (13') -- (2') node [pos=0.3,right] {$\cong$};
\draw [->] (8) -- (1) node[pos=0.3,left] {$\widehat{\pi}$} ;

\draw [->] (15') -- (3') node[pos=0.4,left] {$\widehat{\pi}_{H}$} node[pos=0.4,right] {$\cong$};

\draw [dashed,->](8) -- (9);
\draw[->](11) -- (12);
\draw [->](14) -- (15);

\draw [->] (12') -- (14);

		\end{tikzpicture}
        \caption{Blow-up corresponding to Path $2$ in Figure \ref{fig:bualg}} \label{fig:rg} 
    \end{subfigure}
    
    \end{figure} 
    
    \begin{figure} \ContinuedFloat
      
  \begin{subfigure}[t]{1\hsize}
        \centering
         	\begin{tikzpicture}[scale=1,x=(15:6.5cm),y=(0:4cm), z=(90:1.8cm),
    cross line/.style={preaction={draw=white,-,line width=3pt}},
    equal/.style={double distance=1pt}]
    

\node (c1) at (1,-1.4,-0.25) {  } ;
\node (c2) at (1,-1.4,1.65) { };
\node (c3) at (1,0.45,-0.25) { };
\node (c4) at (1,0.45,1.65) { };
\node (c5) at (0,-1.4,-0.25) { } ;
\node (c6) at (0,-1.4,1.65) { } ;
\node (c7) at (0,0.45,-0.25) { } ;
\node (c8) at (0,0.45,1.65) { } ;

\node (1) at (0,-1,1.4) {\( X \)};
\node (1') at (0,0,1.4) {\( X \gitq R_{\l} \)};
\node (2) at (0,-1,0.7) {\( X^{ss,R_{\l}} \)};
\node (2') at (0,-1,0) {\( S_0(X,H)   \)};
\node (3) at (0,0,0) {\( S_0(X,H) / H \)};
\node (16) at (1,0,1.4) {\(   \tX \gitq R_{\l}  \)};
\node(19) at (1,-1,1.4) {\( \tX \)};
\node(19') at (1,0,0.7) {\( \tX^{s,R_{\l}} / R_{\l} \)};
\node(20) at (1,-1,0.7) {\(  \tX^{ss,R_{\l}} = \tX^{s,R_{\l}}  \)};
\node(21) at (1,-1,0) {\( \tX^{ss,R_{\l}} \setminus \widetilde{E} \)}; 
\node(22) at (1,0,0) {\( \left(\tX^{ss,R_{\l}} \setminus \widetilde{E} \right) / R_{\l} \)}; 
     
\filldraw  [fill = orange, draw=black,loosely dotted, opacity=0.2] (c5) rectangle (c8);
\filldraw  [fill = yellow, draw=black,loosely dotted, opacity=0.2] (c1) rectangle (c4);
  
\draw [black,loosely dotted] (c1) rectangle (c4);
\draw  [black,loosely dotted] (c5) rectangle (c8);
\draw [black, loosely dotted] (c3) -- (c7);
\draw [black, loosely dotted] (c4) -- (c8);
\draw [black, loosely dotted] (c1) -- (c5); 
\draw [black, loosely dotted] (c2) -- (c6);
     
\draw [->, dashed] (1) -- (1');
\draw [right hook ->](2) -- (1);
\draw [->] (2) -- (1');
\draw [right hook ->] (2') -- (2);
\draw [ ->](2') -- (3);
\draw [->] (21) -- (2') node[pos=0.5,below] {$\cong$}; 
\draw [right hook ->] (21) -- (20);
\draw [->] (21) -- (22) ;
\draw [->] (22) -- (3) node[pos=0.6,below] {$\cong$} ;
\draw [right hook -> ] (22) -- (19');
\draw [dashed, ->] (19) -- (16);
\draw[->] (20) -- (19');

\draw [equal] (19') -- (16);
\draw [right hook ->] (20) -- (19);
\draw [->] (19) -- (1) node[pos=0.4,above] {$\widetilde{\pi}$} ;
\draw [->, cross line] (16) -- (1') node[pos=0.4,above] {$\widetilde{\pi}_G$} ;    
\draw [ForestGreen,line width=0.5mm, right hook ->, cross line] (3) -- (16);
\draw [ForestGreen, line width=0.5mm,right hook ->, cross line ] (3) -- (1');
    
		\end{tikzpicture}
        \caption{Blow-up corresponding to Path $1'$ in Figure \ref{fig:bualg} (when $\tX \neq \overline{U \widetilde{Z}_{\min}}$)} \label{fig:oy}
    \end{subfigure}
    
      \begin{subfigure}{1\hsize}
        \centering
         	\begin{tikzpicture}[scale=1,x=(15:6.5cm),y=(0:4cm), z=(90:1.8cm),
    cross line/.style={preaction={draw=white,-,line width=3pt}},
    equal/.style={double distance=1pt}]
    

\node (c1) at (1,-1.4,-0.25) {  } ;
\node (c2) at (1,-1.4,1.65) { };
\node (c3) at (1,0.45,-0.25) { };
\node (c4) at (1,0.45,1.65) { };
\node (c5) at (0,-1.4,-0.25) { } ;
\node (c6) at (0,-1.4,1.65) { } ;
\node (c7) at (0,0.45,-0.25) { } ;
\node (c8) at (0,0.45,1.65) { } ;

\node (1) at (0,-1,1.4) {\( X \)};
\node (2') at (0,-1,0) {\( S_0(X,H)   \)};
\node (3) at (0,0,0) {\( S_0(X,H) / H \)};
\node(19) at (1,-1,1.4) {\( \tX = \overline{\widetilde{S}_{\beta}}  \)};
\node(19') at (1,0,0.7) {\( S_0(\tX, R_{\l}) / R_{\l} \)};
\node(20) at (1,-1,0.7) {\(  S_0(\tX, R_{\l})  \)};
\node(21) at (1,-1,0) {\( S_0(\tX, R_{\l}) \setminus \widetilde{E} \)}; 
\node(22) at (1,0,0) {\( \left(S_0(\tX, R_{\l}) \setminus \widetilde{E} \right) / R_{\l} \)}; 
     
\filldraw  [fill = orange, draw=black,loosely dotted, opacity=0.2] (c5) rectangle (c8);
\filldraw  [fill = gray, draw=black,loosely dotted, opacity=0.2] (c1) rectangle (c4);
  
\draw [black,loosely dotted] (c1) rectangle (c4);
\draw  [black,loosely dotted] (c5) rectangle (c8);
\draw [black, loosely dotted] (c3) -- (c7);
\draw [black, loosely dotted] (c4) -- (c8);
\draw [black, loosely dotted] (c1) -- (c5); 
\draw [black, loosely dotted] (c2) -- (c6);
     
\draw [right hook ->] (2') -- (1);
\draw [ ->](2') -- (3);
\draw [->] (21) -- (2') node[pos=0.5,below] {$\cong$}; 
\draw [right hook ->] (21) -- (20);
\draw [->] (21) -- (22) ;
\draw [->] (22) -- (3) node[pos=0.6,below] {$\cong$} ;
\draw [right hook -> ] (22) -- (19');
\draw[->] (20) -- (19');
\draw [right hook ->] (20) -- (19);
\draw [->] (19) -- (1) node[pos=0.4,above] {$\widetilde{\pi}$} ;
    
		\end{tikzpicture}
        \caption{Blow-up corresponding to Path $2'$ in Figure \ref{fig:bualg} (when $\tX \neq \overline{U \widetilde{Z}_{\min}}$)} \label{fig:og}
    \end{subfigure}

 \caption{Blow-ups from the Projective completion algorithm} \label{fig:bups}
\end{figure}

\begin{pcalg}  \thlabel{blowupalg} 

Given a $4$-tuple $(X,H,\l,\mathcal{L})$ encoding the data consisting of the linearised action of a non-reductive group $H$ with internally graded unipotent radical $U$ and Levi subgroup $R$ on an irreducible projective variety $X$, this algorithm produces: \begin{enumerate}[label = \roman*)]
\item a non-empty open subset $S_0(X,H,\lambda,\mathcal{L}) \subseteq X$ such that $S_0(X,H,\lambda,\mathcal{L}) \to S_0(X,H,\lambda,\mathcal{L}) / H$ is a geometric for the action of $H$ on $S_0(X,H,\lambda,\mathcal{L})$; and
\item a projective variety $X'$ with an action by an internally graded non-reductive group $H' = U' \rtimes R'$ satisfying (ss=s($U'$)) and such that $X' \gitq H' = X'^{(M)s,H}_{\operatorname{min+}} / H'$ is a projective geometric quotient for the action of $H'$ on $X$ and a projective completion of $S_0(X,H,\lambda,\mathcal{L})/H$. This projective completion is denoted $\PC(S_0(X,H,\lambda,\mathcal{L})/H)$. 
\end{enumerate}

For simplicity, we denote the open subset $S_0(X,H,\lambda,\mathcal{L})$ by $S_0(X,H)$, leaving implicit the dependence on the linearisation and choice of grading one-parameter subgroup.

\begin{rk} \thlabel{rkredtonr}
In the familiar case where a reductive group $G$ (in this case the grading one-parameter subgroup is the trivial one) acts linearly on a projective variety $X$ with non-empty stable locus, the algorithm produces the GIT-stable locus $X^s$ as the $G$-invariant open subset $S_0(X,G)$, and the GIT quotient $X \gitq G$ as the projective completion. 

\end{rk} 

\begin{case} \thlabel{case1} Suppose that $(X,H,\l, \mathcal{L})$ satisfies $(\star)$. We start in the red square of the right-hand rectangle in Figure \ref{fig:bualg}. According to the diagram two paths can be taken, $1$ or $2$. Both correspond to performing the sequence of $\hU$-blow-ups described in Section \ref{subsec:gradedunip}. If $\hxminshu \neq \emptyset$, then we obtain a non-empty projective variety $\hX \gitq \hU$ which is a geometric quotient for the action of $\hU$ on $\hxminshu  \subseteq \hX$. If $\hxminshu = \emptyset$, then we consider instead the projective variety $\widehat{Z}_{\min}$ which is a geometric quotient for the action of $\hU$ on $U \widehat{Z}_{\min}$. In either case, the non-empty projective variety admits an induced action of $R_{\l} = R/\l(\GG_m)$, and the semistable locus for this linearised action determines the path to be taken, as described by the following subcases.  For simplicity, we describe these subcases under the assumption that $\hxminshu \neq \emptyset$, but if this is not the case then it suffices to replace all occurrences of $\hxminshu$ and $\hX \gitq \hU$ with $U \widehat{Z}_{\min}$ and $\widehat{Z}_{\min}$ respectively.

\begin{subcase} \thlabel{case11} If $(\hX \gitq \hU)^{ss, R_{\l}} \neq \emptyset$, then path $1$ must be taken, as illustrated in Figure \ref{fig:ro}. By assumption, the $4$-tuple $(\hX \gitq \hU,R_{\l}, \l_0, \widehat{\mathcal{L}})$ then satisfies $(\star \star)$, where $\l_0$ denotes the trivial one-parameter subgroup of $R_{\l}$. We then proceed to \thref{case2} applied to the $4$-tuple $(\hX \gitq \hU,R_{\l}, \l_0, \widehat{\mathcal{L}})$ in order to construct the open subset $S_0(\hX \gitq \hU, R_{\l}) \subseteq \hX \gitq \hU$ and the projective completion $\PC(S_0(\hX \gitq \hU, R_{\l})  / R_{\l})$. Having done so, we set \begin{equation}
S_0(X,H) := \widehat{\pi} \left(  q_{\hU}^{-1} ( S_0(\hX \gitq \hU, R_{\l}) )  \setminus \widehat{E} \right) 
\end{equation} 
where $\widehat{E}$ is the exceptional divisor corresponding to the sequence of $\hU$-blow-ups $\widehat{\pi}: \hX \to X$. Moreover, we define $$\PC \left( S_0(X,H) / H \right) = \PC ( S_0(\hX \gitq \hU, R_{\l}) / R_{\l} ).$$

\end{subcase} 

\begin{subcase} \thlabel{case12}
If $(\hX \gitq \hU)^{ss, R_{\l}} = \emptyset$, then path $2$ must be taken, as illustrated in Figure \ref{fig:rg}. From the assumption determining this subcase, the semistable locus for the action of $R_{\l}$ on $\hX \gitq \hU$ is empty. Thus the $4$-tuple $(\hX \gitq \hU,R_{\l}, \l_0, \widehat{\mathcal{L}})$ satisfies $(\star \star \star)$. We then proceed to \thref{case3} applied to the $4$-tuple $(\hX \gitq \hU,R_{\l}, \l_0, \widehat{\mathcal{L}})$ in order to construct the open set $S_0(\hX \gitq \hU, R_{\l}) \subseteq \hX \gitq \hU$ and the variety $\PC(S_0(\hX \gitq \hU, R_{\l}))$. Having done so, as in \thref{case11} we set \begin{equation}
S_0(X,H) := \widehat{\pi} \left(  q_{\hU}^{-1} ( S_0(\hX \gitq \hU, R_{\l}) )  \setminus \widehat{E} \right)
\end{equation} and define $$\PC \left(S_0(X,H) / H \right) := \PC (S_0(\hX \gitq \hU, R_{\l}) / R_{\l}).$$

\end{subcase}

\end{case}

\begin{case}  \thlabel{case2} Suppose that $(X,H,\l, \mathcal{L})$ satisfies $(\star \star)$. We start in the orange square of the right-hand rectangle in Figure \ref{fig:bualg}. According to the diagram two paths can be taken, $1'$ or $2'$. Both correspond to performing the sequence of reductive blow-ups described in Section \ref{subsec:classicalredgit}. Since $X^{ss,R_{\l}}$ is non-empty, we can perform a reductive blow-up of $X$ to obtain a projective variety $\tX$ acted upon linearly by $R_{\l}$ such that semistability coincides with stability (or Mumford-stability if there are no properly stable points) for the action of $R_{\l}$ on $\tX$. The semistable locus for this linearised action determines the path to be taken, as described in the following subcases. 
 
\begin{subcase} \thlabel{case21}
If $\tX^{ss,R_{\l}} \neq \emptyset$, then path $1'$ must be taken, as illustrated in Figure \ref{fig:oy}. Taking this path corresponds to performing the $R_{\l}$-blow-up of $X$ to obtain a projective variety $\tX$ for which $\tX^{ss,R_{\l}} = \tX^{(M)s,R_{\l}} \neq \emptyset$. We then define \begin{equation} 
S_0(X,H) : =  \widetilde{\pi} \left( \tX^{ss,R_{\l}} \setminus \widetilde{E} \right),
\end{equation}  
where $\widetilde{E}$ is the exceptional divisor corresponding to the blow-up $\widetilde{\pi}: \tX \to X$, and $$\PC \left( S_0(X,H)/H \right): = \tX \gitq R_{\l}.$$ By assumption we have that $\tX \gitq R_{\l} = \tX^{ss,R_{\l}} / R_{\l}$ is a projective geometric quotient for the action of $R_{\l}$ on $\tX^{ss,R_{\l}}$. The algorithm terminates here. 
\end{subcase}

\begin{subcase} \thlabel{case22}
If $\tX^{ss,R_{\l}} = \emptyset$, then path $2'$ must be taken, as illustrated in Figure \ref{fig:og}. Taking this path corresponds to performing the $R_{\l}$-blow-up of $X$ to obtain a projective variety $\tX$ for which $\tX^{ss,R_{\l}} = \emptyset$. Thus the $4$-tuple $(\tX, R_{\l}, \l_0, \widetilde{\mathcal{L}})$ satisfies $(\star \star \star)$. We then proceed to \thref{case3} applied to the $4$-tuple $(\tX, R_{\l}, \lambda_0, \widetilde{\mathcal{L}})$ in order to construct the open set $S_0(\tX,R_{\l}) \subseteq \tX$ and the variety $\PC(S_0(\tX,R_{\l})/R_{\l})$. Having done so, we set \begin{equation*}
S_0(X,H): = \widetilde{\pi} \left( \tX^{\widehat{s}, R_{\l}} \setminus \widetilde{E} \right) 
\end{equation*} 
where $\widetilde{E}$ is the exceptional divisor corresponding to the blow-up $\widetilde{\pi}: \tX \to X$, and $$\PC \left( S_0(X,H)/H \right) := \PC ( S_0(\tX,R_{\l}) / R_{\l} ).$$ 
\end{subcase} 

\end{case} 

\begin{case} \thlabel{case3}
Suppose that $(X,H,\l,\mathcal{L})$ satisfies $(\star \star \star)$. Then we can apply the replacement algorithm to this $4$-tuple to obtain another $4$-tuple $(X',H',\l',\mathcal{L}')$ satisfying $(\star)$ or $(\star \star)$. 

\begin{subcase} \thlabel{case31} If $(\star)$ is satisfied, then $(X',H',\l',\mathcal{L}')$ must be of the form $(\overline{Y_{\b^{(k)}}^{ss}}, P_{\b^{(k)}}, \l_{\b^{(k)}}, \mathcal{L} |_{\overline{Y_{\b^{(k)}}^{ss}}})$ for some $k \geq 1$ (see Figure \ref{fig:repalg}). We then proceed to \thref{case1} applied to the $4$-tuple $(X',H',\l',\mathcal{L}')$ in order to construct the open set $S_0(X',H') = S_0(\overline{Y_{\b^{(k)}}^{ss}}, P_{\b^{(k)}}) \subseteq \overline{Y_{\b^{(k)}}^{ss}}$ and the variety $\PC(S_0(X',H')/H')$. Having done so, we set \begin{equation*} 
S_0(X,H) : = R \left( L_{\b} \left(  \cdots \left( L_{\b^{(k-1)}} \left( S_0(\overline{Y_{\b^{(k)}}^{ss}}, P_{\b^{(k)}}) \cap Y_{\b^{(k)}}^{ss} \right) \cdots  \right) Y_{\b^{(2)}}^{ss} \right) \cap  Y_{\b^{ss}} \right).
\end{equation*} Moreover, we define $$\PC \left( S_0(X,H)/H \right) = \PC ( S_0(\overline{Y_{\b^{(k)}}^{ss}}, P_{\b^{(k)}})/ P_{\beta^{(k)}} ).$$

\end{subcase}

\begin{subcase} \thlabel{case32}

If $(X',H',\l', \mathcal{L}')$ satisfies $(\star \star)$ instead, then it is of the form $(\overline{Y_{\b^{(k)}}^{ss}}, P_{\b^{(k)}}/ \l_{\b^{(k)}}(\GG_m), \l_0, \mathcal{L} |_{\overline{Y_{\b^{(k)}}^{ss}}})$ for some $k \geq 1$. We then proceed to \thref{case1} applied to the $4$-tuple $(X',H',\l',\mathcal{L}')$ in order to construct the open set $S_0(X',H') = S_0(\overline{Y_{\b^{(k)}}^{ss}}, P_{\b^{(k)}}/\l_{\b^{(k)}}(\GG_m)) \subseteq \overline{Y_{\b^{(k)}}^{ss}}$ and the variety $\PC(S_0(X',H')/H')$. Having done so, we set \begin{equation*} 
S_0(X,H)  = R\left( L_{\b} \left(  \cdots \left( L_{\b^{(k-1)}} \left( S_0(\overline{Y_{\b^{(k)}}^{ss}}, P_{\b^{(k)}}/\l_{\b^{(k)}}(\GG_m)) \cap Y_{\b^{(k)}}^{ss} \right) \cdots  \right) Y_{\b^{(2)}}^{ss} \right) \cap  Y_{\b^{ss}} \right)
\end{equation*} and define $$\PC \left( S_0(X,H) / H \right): = \PC ( S_0(\overline{Y_{\b^{(k)}}^{ss}}, P_{\b^{(k)}}/\l_{\b^{(k)}}(\GG_m)) /  \left( P_{\beta^{(k)}}/\l_{\beta^{(k)}}(\GG_m) \right) ).$$

\end{subcase} 

\end{case}

\end{pcalg} 

\begin{prop} \thlabel{prop:bualgterm}
The projective completion algorithm terminates after a finite number of steps. 
\end{prop}

\begin{proof}

\begin{figure}
	            \centering
 \begin{tikzpicture}[scale=1.2,x=(0:1.6cm), y=(90:1cm)]
          
         \node[rectangle, draw = black,draw opacity = 1, fill = red, fill opacity =0.2, text opacity = 1] (c1) at (0,0) {Case $1$};
         \node[rectangle, draw = black,draw opacity = 1, fill = gray, fill opacity =0.2, text opacity = 1] (c2) at (-2,1) {Case $3$};
         \node[rectangle, draw = black,draw opacity = 1, fill = orange, fill opacity =0.2, text opacity = 1] (c3) at (0,2) {Case $2$};
         \node[rectangle, draw = black,draw opacity = 1, fill = yellow, fill opacity =0.2, text opacity = 1] (stop) at (0,4) {Stop};

 \draw [->] (c3) to node[rectangle,draw,rounded corners,above,yshift=1mm,sloped] {Case $2.1$} (stop) ;
 
 \draw [->,bend right = 20,color=blue] (c2) to node[rectangle,draw,rounded corners,below,sloped,yshift=-1mm] {Case $3.1$} (c1);
 
  \draw [->,bend right = 20,color=green] (c1) to node[rectangle,draw,rounded corners,below,yshift=-1mm,sloped] {Case $1.2$} (c2);
 
 \draw [->,color=green] (c1) to node[rectangle,draw,rounded corners,right,sloped,above,yshift=1mm] {Case $1.1$} (c3);
 
  \draw [->, bend right = 20,color=green] (c2) to node[rectangle,draw,rounded corners,above,sloped,yshift=1mm] {Case $3.2$} (c3);
 
   \draw [->, bend right = 20,color=blue] (c3) to node[rectangle,draw,rounded corners,above,yshift=1mm,sloped] {Case $2.2$} (c2);

\end{tikzpicture}

\caption{Proof that the projective completion algorithm must terminate. 
A green (respectively blue) arrow indicates that the dimension of the group decreases (respectively decreases or remains the same) upon passing from one case to the other. Since each loop contains at least one green arrow, when going through a loop the dimension of the group must always decrease.
} \label{fig:proof} 
\end{figure}

Suppose that the initial $4$-tuple is $(X,H,\l,\mathcal{L})$. By construction, the projective completion algorithm can only terminate by reaching \thref{case21}. 

If $(\star)$ is satisfied for $(X,H,\l,\mathcal{L})$, then we start in \thref{case1}, according to which $(X,H,\l,\mathcal{L})$ is replaced by $(\hX \gitq \hU, R_{\l}, \l_0, \widehat{\mathcal{L}})$ (or by $(\widehat{Z}_{\min}, R_{\l}, \l_0, \widehat{\mathcal{L}}|_{\zmin})$ if $\hxminshu = \emptyset$). Since $(\star)$ holds for $(X,H,\l,\mathcal{L})$, the one-parameter subgroup $\l(\GG_m)$ acts non-trivially on $X$ and hence is itself non-trivial. Thus we have that $\operatorname{dim} R_{\l} = \operatorname{dim} R/\l(\GG_m) < \operatorname{dim} R \leq \operatorname{dim} H$, and so whether we move to \thref{case2} (if \thref{case11} holds) or \thref{case3} (if \thref{case12} holds), the new starting $4$-tuple involves a group of dimension strictly smaller than the dimension of $H$. 

If $(\star \star)$ is satisfied for $(X,H,\l,\mathcal{L})$, then we start in \thref{case2}. If \thref{case21} holds, then the algorithm terminates at this step. If \thref{case22} holds, then according to the algorithm we must proceed to \thref{case3}, using the $4$-tuple $(\tX, R_{\l}, \l_0, \widetilde{\mathcal{L}})$. Since $\operatorname{dim} R_{\l} = \operatorname{dim} R / \l(\GG_m)  \leq \operatorname{dim} R \leq \operatorname{dim } H$ where equality may hold as $\l(\GG_m)$ acts trivially on $X$, we are applying \thref{case3} to a $4$-tuple for which the dimension of the group is smaller than or equal to the dimension of $H$. 

If $(\star \star \star)$ is satisfied for $(X,H,\l,\mathcal{L})$, then we start in \thref{case3}, according to which the $4$-tuple must be replaced either by $(\overline{Y_{\b^{(k)}}^{ss}}, P_{\b^{(k)}}, \l_{\b^{(k)}}, \mathcal{L} |_{\overline{Y_{\b^{(k)}}^{ss}}})$ (\thref{case31}) or by $(\overline{Y_{\b^{(k)}}^{ss}}, P_{\b^{(k)}}/\l_{\b^{(k)}}(\GG_m), \l_0, \mathcal{L} |_{\overline{Y_{\b^{(k)}}^{ss}}})$ (\thref{case32}) for some $k \geq 1$. Note that by construction, we have that $\operatorname{dim} P_{\b^{(k)}} \leq \operatorname{dim} R/\l(\GG_m)  \leq \operatorname{dim} R \leq  \operatorname{dim} H$, where equality may hold throughout (since $\l(\GG_m)$ could be trivial). However, since $\l_{\b^{(k)}}(\GG_m)$ is a non-trivial one-parameter subgroup, we have that $\operatorname{dim} P_{\b^{(k)}}/\l_{\b^{(k)}}(\GG_m) < \operatorname{dim} P_{\b^{(k)}}.$ Thus in \thref{case31} we obtain a new $4$-tuple for which the dimension of the group is smaller than or equal to the dimension of $H$, while in \thref{case32} the dimension of the group becomes strictly smaller.

The only way in which the algorithm might not terminate is if we could travel between \thref{case1,case2,case3} infinitely many times. In Figure \ref{fig:proof}, this corresponds to circling through \thref{case1,case2,case3} via loops formed by the arrows coloured green (indicating that the dimension of the group decreases) and blue (indicating that the dimension of the group decreases or stays the same). However, as the diagram shows, every such closed loop contains at least one green arrow. Thus each time we travel through a loop, the dimension of the group must decrease. If we did this ad infinitum, we would eventually reach a case for which the starting $4$-tuple $(X',H',\l',\mathcal{L}')$ satisfies $H' = \{e\}$. Since $\l'(\GG_m)$ is then trivial, and since $X'^{ss,R'/\l'(\GG_m)} = X'^{ss,\{e\}} = X' \neq \emptyset$, \thref{case2} applies.  Note that we also then have $X'^{s,R'/\l'(\GG_m)} = X'^{ss,R'/\l'(\GG_m)} \neq \emptyset$, so the reductive blow-up is $X'$ itself, and hence $\widetilde{X'}^{ss,R'/\l'(\GG_m)} \neq \emptyset$. Thus \thref{case21} would be automatically satisfied, and so the algorithm would terminate.
\end{proof} 


 \newpage
\section{Index of notation} \label{index}

\begin{tabular}{lll}
\textbf{Notation} & \textbf{Short description} & \textbf{}\\
$\kk$ & algebraically closed field of char $0$ & \\ 
$\GG_m$ & multiplicative group over $\kk$ & \\ 
$G$ & reductive linear algebraic group over $\kk$ & \\ 
$U$ & graded unipotent linear group over $\kk$ & $\S$\ref{section1} \\ 
$\hU=U \rtimes \l(\GG_m)$ & semi-direct product where $\l:\GG_m \to \hU$ grades $U$ & $\S$\ref{section1} \\ 
$H=U \rtimes R$ & \makecell[lt]{linear algebraic group over $\kk$ with unipotent radical $U$ \\ and Levi subgroup $R$} & \\ 
$\l:\GG_m \to Z(R)$ & central $1$-parameter subgroup of $R$ which grades $U$ & $\S$\ref{section1} \\ 
& $H$ \textsl{has internally graded unipotent radical} & $\S$\ref{section2.2} \\ 
$\hH=H \rtimes \GG_m$ & $H$ \textsl{has externally graded unipotent radical}  & $\S$\ref{section2.1} \\ 
$T$ & maximal torus of $R$ (and thus of $H$) & \\ 
$\chi: H \to \GG_m$ & character of $H$ & \\ 
$\chi \in \Lie(H)^*$ & (rational) character of $H$ & \\
$\o_0<\o_1<\ldots <\o_h$ & weights of linear $\GG_m$ action on $X$ & Definition \ref{defn2.1}  \\  
$\o_{\min}=\o_0, \o_{\max}=\o_h$ & minimal and maximal weights & \\
$X^{\GG_m}$ or $X^{\l(\GG_m)}$ & fixed point locus for $\GG_m$-action via $\l$ & \\
$Z_{\min}$ or $Z_{\min}(X)$ & sublocus of $X^{\GG_m}$ with weight $\o_{\min}$ & Definition \ref{def:ExPr2} \\ 
$p=\pi_{\min}:X_{\min}^0 \to Z_{\min}$ & limit of $\GG_m$ orbit $p(x)=\lim_{t\to 0} \l(t) \cdot x$ & Definition \ref{def:ExPr2} \\ 
$\hat{X}$ & blow-up to resolve singularities of quotient from $U$-action & $\S$\ref{section5} \\
$\tilde{X}$ & blow-up to resolve singularities of quotient from $H$-action & $\S$\ref{externalsection} \\ 
\textbf{(Semi)stable loci} & & \\
$X^{(s)s,H}$ &  &$\S$\ref{section2} \\ 
$X^{s,\GG_m}_{\min +}$, \,\, $X^{s,\hU}_{\min +}$ & stable loci for well-adapted actions of $\GG_m,\, \hU$ & $\S$\ref{section2.1} \\ 
$X^{(s)s,\hat{T}}_{\min +}$, \,\,  
$X^{(s)s,\hH}_{\min +}$  & (semi)stable loci for well-adapted actions of $\hat{T},\, \hH$  & Remark \ref{defnchio} \\ 
$X^{\hat{s},H}$ & hat-stable locus & Definition \ref{defn9.1} \\ 
$X^{M\hat{s}, H \geqslant U \geqslant U^{(1)} \geqslant \cdots \geqslant U^{(s)}}$ & Mumford-hat-stable locus & $\S$\ref{externalsection}\\
\textbf{Quotients} & & \\
$X \gitq G$ & classical GIT quotient  & \\
$X \env H$ & enveloping quotient & $\S$\ref{section2} \\ 
$X\, \widehat{\env} \, H$ & $(\widehat{X \times \PP^1} \env \hU)/\!/R$ \, (external case) & Definition \ref{defn9.1} \\
  & $(\widehat{X} \env \hU)/\!/ (R/\l(\GG_m))$ \, (internal case) & Remark \ref{remark9.4} \\ 
$X \, \widetilde{\env}\, H$ & 
 $(\widetilde{X \times \PP^1} \env \hU)/\!/R$ \, (external case) & Definition \ref{defn9.1} \\
  & $(\widetilde{X} \env \hU)/\!/ (R/\l(\GG_m))$ \, (internal case) & Remark \ref{remark9.4} \\ 
\end{tabular}

\newpage
\bibliographystyle{abbrv}
\bibliography{ProjectiveCompletions.bib}

\end{document}